\documentclass[draft=false,bibliography=totoc,index=totoc]{scrartcl}

\usepackage{amsmath}
\usepackage{amssymb}
\usepackage{amsthm}
\usepackage{cite}
\usepackage[show]{ed}
\usepackage[T1]{fontenc}
\usepackage{hyperref}
\usepackage{mathtools}
\usepackage{microtype}
\usepackage[automark]{scrpage2}

\usepackage{114arxiv}

 \newtheorem{thm}{Theorem}[section]
 
 \newtheorem{lem}[thm]{Lemma}
 
\theoremstyle{definition}
 \newtheorem{defn}[thm]{Definition}
 
\theoremstyle{remark}
 \newtheorem{rem}[thm]{Remark}
 
 \newtheorem{exa}[thm]{Example}
 
\mathtoolsset{mathic=true}

\numberwithin{equation}{section}
\title{On a Simultaneous Approach to the Even and Odd Truncated Matricial \tStieltjes{} Moment Problem~I:}
\subtitle{An $\alpha$-\tSchur{}-\tStieltjes{}-type algorithm for sequences of complex matrices}
\author{Bernd Fritzsche \and Bernd Kirstein \and Conrad M\"adler}
\date{}

\begin{document}
\maketitle

\begin{abstract}
 The characterization of the solvability of matrix versions of truncated \tStieltjes{}-Type moment problems led to the class of \taSnnd{\alpha} sequences of complex \tqqa{matrices}. In~\zita{MR3014201}, a parametrization of this class was introduced, the so-called \tasp{\alpha}. The main topic of this first part of the paper is the construction of a \tSchur{}-type algorithm which produces exactly the \tasp{\alpha}.
\end{abstract}

\section{Introduction}\label{S1417}
 This paper which is divided into two parts is a direct continuation of the authors' former investigations on matrix versions of classical power moment problems and related questions (see~\zitas{MR2570113,MR2735313,MR2805417,MR3014201,MR3014199,MR3014197,MR2988005,MR3133464,MR3380267}). Now our aim is to study two truncated matricial power moment problems on semi-infinite intervals. The approach is based on \tSchur{} analysis. More precisely, we will work out two interrelated versions of \tSchur{}-type algorithms similar as in our former investigations on truncated matrix moment problems of \tHamburger{}-Type (see~\zitas{MR3014199,MR3380267}). In the first part, we will construct a \tSchur{}-type algorithm for finite sequences of complex \tqqa{matrices}. In the second part of the paper, we will construct a \tSchur{}-type algorithm for a special class of holomorphic matrix functions, which turns out to be intimately related to the truncated matricial moment problems under consideration. The starting point of studying power moment problems on semi-infinite intervals was the famous two parts memoir of \tStieltjes{}~\zitas{MR1508159,MR1508160} where the author's investigations on questions for special continued fractions led him to the power moment problem on the interval \([0,+\infty)\). A complete theory of the treatment of power moment problems on semi-infinite intervals in the scalar case was developed by M.~G.~Krein in collaboration with A.~A.~Nudelman (see~\zitaa{MR0044591}{\cSect{10}},~\zita{MR0233157},~\zitaa{MR0458081}{\cchap{V}}). What concerns a modern operator-theoretic treatment of the power moment problems named after Hamburger and \tStieltjes{} and its interrelations, we refer the reader to Simon~\zita{MR1627806}. The matrix version of the classical \tStieltjes{} moment problem was studied in Adamyan/Tkachenko~\zitas{MR2155645,MR2215856}, And\^o~\zita{MR0290157}, Bolotnikov~\zitas{MR975671,MR1362524,MR1433234}, Bolotnikov/Sakhnovich~\zita{MR1722780}, Chen/Hu~\zita{MR1807884}, Chen/Li~\zita{MR1670527}, Dyukarev~\zitas{Dyu81,MR686076}, Dyukarev/Katsnel{\('$}son~\zitas{MR645305,MR752057}, and Hu/Chen~\zita{MR2038751}. The considerations of this paper deal with the case of a semi-infinite interval \([\alpha,+\infty)\) and continue former work done in~\zitas{MR2735313,MR3014201,MR3133464}.

 In order to properly formulate these problems, we first review some notation. Let \(\C\)\index{c@$\C$}, \(\R\)\index{r@$\R$}, \(\NO\)\index{n@$\NO$}, and \(\N\)\index{n@$\N$} be the set of all complex numbers, the set of all real numbers, the set of all \tnn{} integers, and the set of all positive integers, respectively. Further, for all \(\alpha,\beta\in\R\cup\set{-\infty,+\infty}\), let \(\mn{\alpha}{\beta}\)\index{z@$\mn{\alpha}{\beta}$} be the set of all integers \(k\) for which \(\alpha\leq k\leq\beta\) holds. Throughout this paper, let \(p,q\in\N\)\index{p@$p$}\index{p@$p$}. If \(\cX\) is a nonempty set, then \(\cX^\pxq\)\index{$\cX^\pxq$} stands for the set of all \tpqa{matrices} each entry of which belongs to \(\cX\) and \(\cX^p\)\index{$\cX^p$} is short for \(\cX^\xx{p}{1}\). If \(k\in\Z\) and if \(\kappa\in\mn{k}{+\infty}\cup\set{+\infty}\), then we denote by \(\setseqauu{\mathcal{X}}{k}{\kappa}\)\index{$\setseqauu{\mathcal{X}}{k}{\kappa}$} the set of all sequences \(\seq{x_j}{j}{k}{\kappa}\) where \(x_j\in\mathcal{X}\) for all \(j\in\mn{k}{\kappa}\). If \((\cX,\gA)\) is a measurable space, then each countably additive mapping defined on \(\gA\) with values in the set \(\Cggq\)\index{c@$\Cggq$} of all \tnnH{} complex \tqqa{matrices} is called a \tnnH{} \tqqa{measure} on \((\cX,\gA)\).

 Let \(\alpha\in\R\) and let \(\Bori{[\alpha,+\infty)}\)\index{b@$\Bori{[\alpha,+\infty)}$} be the \(\sigma\)\nbd algebra of all Borel subsets of \([\alpha,+\infty)\). Further, let \(\Mggqa{[\alpha,+\infty)}\)\index{m@$\Mggqa{[\alpha,+\infty)}$} be the set of all \tnnH{} \tqqa{measures} on \(([\alpha,+\infty),\Bori{[\alpha,+\infty)})\) and, for all \(\kappa\in\NO\cup\set{+\infty}\), let \(\Mgguqa{\kappa}{[\alpha,+\infty)}\)\index{m@$\Mgguqa{\kappa}{[\alpha,+\infty)}$} be the set of all \(\sigma\in\Mggqa{[\alpha,+\infty)}\) such that the integral
\[
 \suo{j}{\sigma}
 \defg\int_{[\alpha,+\infty)} t^j\sigma(\dif t)
\]
\index{s@$\suo{j}{\sigma}$}exists for all \(j\in\mn{0}{\kappa}\). The above-mentioned matricial moment problems can be formulated as follows:
\begin{description}
 \item[\mproblem{\iraa{\alpha}}{\kappa}{=}]\index{k@\mproblem{\iraa{\alpha}}{\kappa}{=}} Let \(\alpha\in\R\), let \(\kappa\in\NO\cup\set{+\infty}\), and let \(\seq{\su{j}}{j}{0}{\kappa}\) be a sequence of complex \tqqa{matrices}. Parametrize the set \(\Mggqaag{\iraa{\alpha}}{\seq{\su{j}}{j}{0}{\kappa}}\)\index{k@$\Mggqaag{\iraa{\alpha}}{\seq{\su{j}}{j}{0}{\kappa}}$} of all \tnnH{} measures \(\sigma\in\Mgguqa{\kappa}{\iraa{\alpha}}\) for which \(\suo{j}{\sigma}=\su{j}\) is fulfilled for all \(j\in\mn{0}{\kappa}\).
 \item[\mproblem{\iraa{\alpha}}{m}{\leq}]\index{m@\mproblem{\iraa{\alpha}}{m}{\leq}} Let \(\alpha\in\R\), let \(m\in\NO\), and let \(\seq{\su{j}}{j}{0}{m}\) be a sequence of complex \tqqa{matrices}. Parametrize the set \(\Mggqaakg{\iraa{\alpha}}{\seq{\su{j}}{j}{0}{m}}\)\index{m@$\Mggqaakg{\iraa{\alpha}}{\seq{\su{j}}{j}{0}{m}}$} of all \tnnH{} measures \(\sigma\in\Mgguqa{m}{\iraa{\alpha}}\) for which \(\su{m}-\suo{m}{\sigma}\) is \tnnH{} and, in the case \(m\geq1\), moreover \(\suo{j}{\sigma}=\su{j}\) is fulfilled for all \(j\in\mn{0}{m-1}\).
\end{description}
The particular case \(\alpha=0\) is treated by several authors. Parametrizations of the set \(\MggqSakg{\seq{\su{j}}{j}{0}{m}}\) are given, under a certain non-degeneracy condition, in~\zita{MR2053150}. In the general case, descriptions of \(\MggqSakg{\seq{\su{j}}{j}{0}{m}}\) are stated by Bolotnikov~\zita{MR975671},~\zitaa{MR1362524}{\ctheo{1.5}} and by Chen/Hu in~\zitaa{MR1807884}{\ctheo{2.4}}, whereas descriptions of the set \(\MggqSag{\seq{\su{j}}{j}{0}{m}}\) are given by Hu/Chen in~\zitaa{MR2038751}{\ctheo{4.1}, \clemmss{2.3}{2.4}}. Moreover, for arbitrary real numbers \(\alpha\), the cases \(\Mggqaag{\iraa{\alpha}}{\seq{\su{j}}{j}{0}{m}}\neq\emptyset\) and \(\Mggqaakg{\iraa{\alpha}}{\seq{\su{j}}{j}{0}{m}}\neq\emptyset\) are characterized in~\zitaa{MR2735313}{\ctheoss{1.3}{1.4}}.

Before we explain these criterions of solvability, we introduce certain sets of sequences of complex \tqqa{matrices} which are determined by the properties of particular \tHankel{} matrices built of them. For all \(n\in\NO\), let \(\Hggqu{2n}\)\index{h@$\Hggqu{2n}$} (resp.\ \(\Hgqu{2n}\)\index{h@$\Hgqu{2n}$}) be the set of all sequences \(\seq{\su{j}}{j}{0}{2n}\) of complex \tqqa{matrices} such that the block Hankel matrix
\bgl{H}
 \Hu{n}
 \defg\matauuo{\su{j+k}}{j,k}{0}{n}
\eg
 \index{h@$\Huo{n}{s}\), \(\Hu{n}$}is \tnnH{} (resp.\ \tpH{}). Furthermore, let \(\Hggqinf\)\index{h@$\Hggqinf$} (resp.\ \(\Hgqinf\)\index{h@$\Hgqinf$}) be the set of all sequences \(\seq{\su{j}}{j}{0}{\infty}\) of complex \tqqa{matrices} such that for all \(n\in\NO\) the sequence \(\seq{\su{j}}{j}{0}{2n}\) belongs to \(\Hggqu{2n}\) (resp.\ \(\Hgqu{2n}\)). For all \(n\in\NO\), let \(\Hggequ{2n}\)\index{h@$\Hggequ{2n}$} be the set of all sequences \(\seq{\su{j}}{j}{0}{2n}\) of complex \tqqa{matrices} for which there exist complex \tqqa{matrices} \(\su{2n+1}\) and \(\su{2n+2}\) such that \(\seq{\su{j}}{j}{0}{2(n+1)}\) belongs to \(\Hggqu{2(n+1)}\). Furthermore, for all \(n\in\NO\), we will use \(\Hggequ{2n+1}\)\index{h@$\Hggequ{2n+1}$} to denote the set of all sequences \(\seq{\su{j}}{j}{0}{2n+1}\) of complex \tqqa{matrices} for which there exists a complex \tqqa{matrix} \(\su{2n+2}\) such that \(\seq{\su{j}}{j}{0}{2(n+1)}\) belongs to \(\Hggqu{2(n+1)}\). Let \(\Hggeqinf\defg\Hggqinf\)\index{h@$\Hggeqinf$}. If \(\kappa\in\NO\cup\set{+\infty}\), then we call a sequence \(\seq{\su{j}}{j}{0}{2\kappa}\) of complex \tqqa{matrices} \emph{\tHnnd{}} (resp.\ \emph{\tHpd{}}) if it belongs to \(\Hggqu{2\kappa}\) (resp.\ \(\Hgqu{2\kappa}\)), and we call a sequence \(\seq{\su{j}}{j}{0}{\kappa}\) of complex \tqqa{matrices} \emph{\tHnnde} if it belongs to \(\Hggequ{\kappa}\).

Besides the just introduced classes of sequences of complex \tqqa{matrices} we need analogous classes of sequences of complex \tqqa{matrices} which take into account the influence of the prescribed number \(\alpha\in\R\): Let \(\seq{\su{j}}{j}{0}{\kappa}\) be a sequence of complex \tpqa{matrices}. Then, for all \(n\in\NO\) with \(2n+1\leq\kappa\), we introduce the block Hankel matrix
\bgl{K}
 \Ku{n}
 \defg\matauuo{\su{j+k+1}}{j,k}{0}{n}.
\eg
\index{k@$\Kuo{n}{s}\), \(\Ku{n}$}If \(\kappa\geq1\), then, for all \(\alpha\in\C\), let the sequence \(\seq{\sau{j}}{j}{0}{\kappa-1}\)\index{$\seq{\sau{j}}{j}{0}{\kappa-1}$} be given by
\bgl{s_a}%
 \sau{j}
 \defg-\alpha\su{j}+\su{j+1}
\eg
for all \(j\in\mn{0}{\kappa-1}\). Then \(\seq{\sau{j}}{j}{0}{\kappa-1}\) is called the \emph{\tsgfbrsas{\(\seq{\su{j}}{j}{0}{\kappa}$}{\alpha}}. (An analogous left-sided version is discussed in~\zitaa{MR3014201}{\cdefn{2.1}}.)

Let \(\alpha\in\R\). Now we will introduce several classes of finite or infinite sequences of complex \tqqa{matrices} which are characterized by the sequences \(\seq{\su{j}}{j}{0}{\kappa}\) and \(\seq{\sau{j}}{j}{0}{\kappa-1}\). Let \(\Kggqualpha{0}\defg\Hggqu{0}\)\index{k@$\Kggqualpha{0}$}, and, for all \(n\in\N\), let \(\Kggquu{2n}{\alpha}\) be the set of all sequences \(\seq{\su{j}}{j}{0}{2n}\) of complex \tqqa{matrices} for which the block \tHankel{} matrices \(\Hu{n}\) and \(-\alpha\Hu{n-1}+\Ku{n-1}\) are both \tnnH{}, i.\,e.,
\bgl{Kgg2n}
 \Kggqualpha{2n}
 \defg\SetaA{\seq{\su{j}}{j}{0}{2n}\in\Hggqu{2n}}{\seq{\sau{j}}{j}{0}{2(n-1)}\in\Hggqu{2(n-1)}}.
\eg
\index{k@$\Kggquu{2n}{\alpha}$}Furthermore, for all \(n\in\NO\), let \(\Kggquu{2n+1}{\alpha}\) be the set of all sequences \(\seq{\su{j}}{j}{0}{2n+1}\) of complex \tqqa{matrices} for which the block \tHankel{} matrices \(\Hu{n}\) and \(-\alpha\Hu{n}+\Ku{n}\) are both \tnnH{}, i.\,e.,
\bgl{Kgg2n+1}
 \Kggqualpha{2n+1}
 \defg\SetaA{\seq{\su{j}}{j}{0}{2n+1}\in\setseqauu{(\Cqq)}{0}{2n+1}}{\Set{\seq{\su{j}}{j}{0}{2n},\seq{\sau{j}}{j}{0}{2n}}\subseteq\Hggqu{2n}}.
\eg
\index{k@$\Kggqualpha{2n+1}$}Formulas \eqref{Kgg2n} and \eqref{Kgg2n+1} show that the sets \(\Kggqualpha{2n}\) and \(\Kggqualpha{2n+1}\) are determined by two conditions. The condition \(\seq{\su{j}}{j}{0}{2n}\in\Hggqu{2n}\) ensures that a particular \tHamburger{} moment problem associated with the sequence \(\seq{\su{j}}{j}{0}{2n}\) is solvable (see, e.\,g.~\zitaa{MR2570113}{\ctheo{4.16}}). The second condition \(\seq{\sau{j}}{j}{0}{2(n-1)}\in\Hggqu{2(n-1)}\) (resp.\ \(\seq{\sau{j}}{j}{0}{2n}\in\Hggqu{2n}\)) controls that the original sequences \(\seq{\su{j}}{j}{0}{2n}\) and \(\seq{\su{j}}{j}{0}{2n+1}\) are well adapted to the interval \([\alpha,+\infty)\). Let \(m\in\NO\). Then, let \(\Kggequu{m}{\alpha}\)\index{k@$\Kggequu{m}{\alpha}$} be the set of all sequences \(\seq{\su{j}}{j}{0}{m}\) of complex \tqqa{matrices} for which there exists a complex \tqqa{matrix} \(\su{m+1}\) such that \(\seq{\su{j}}{j}{0}{m+1}\) belongs to \(\Kggquu{m+1}{\alpha}\). We call a sequence \(\seq{\su{j}}{j}{0}{m}\) of complex \tqqa{matrices} \emph{\taSnnd{\alpha}} if it belongs to \(\Kggqualpha{m}\) and \emph{\taSnnde{\alpha}} if it belongs to \(\Kggequalpha{m}\).

 Let us recall solvability criterions for the Problems~\mproblem{\iraa{\alpha}}{m}{=} and~\mproblem{\iraa{\alpha}}{m}{\leq}:
\begin{thm}[\zitaa{MR2735313}{\ctheoss{1.3}{1.4}}]\label{T1423}
 Let \(\alpha\in\R\), let \(m\in\NO\), and let \(\seq{\su{j}}{j}{0}{m}\) be a sequence of complex \tqqa{matrices}. Then:
 \benui
  \item \(\Mggqaag{\iraa{\alpha}}{\seq{\su{j}}{j}{0}{m}}\neq\emptyset\) if and only if \(\seq{\su{j}}{j}{0}{m}\in\Kggequu{m}{\alpha}\).
  \item \(\Mggqaakg{\iraa{\alpha}}{\seq{\su{j}}{j}{0}{m}}\neq\emptyset\) if and only if \(\seq{\su{j}}{j}{0}{m}\in\Kggquu{m}{\alpha}\).
 \eenui
\etheo
 \rtheo{T1423} provides a first impression on the importance of the classes of \taSnnd{\alpha} sequences and \taSnnde{\alpha} sequences in the framework of truncated matricial \tStieltjes{}-type moment problems related to the interval \([\alpha,+\infty)\). The role of these classes is by far not restricted to these solvability conditions. A thorough study of these classes forms an essential cornerstone in our approach to describe the solution sets of both moment problems via \tSchur{} analysis methods. For this reason, we extend our former investigations (see~\zitas{MR2735313,MR3014201}) on \taSnnd{\alpha} sequences and remarkable subclasses by constructing a special \tSchur{}-type algorithm for finite sequences of complex \tpqa{matrices}. It should be mentioned that we will consider the matricial moment problems under consideration in the most general case. The so-called non-degenerate case is connected to the class of so-called \taSpd{\alpha} sequences (see \rsect{S1432}).

 In this paper, we construct a \tSchur{}-type algorithm for sequences belonging to the class \(\Kggquu{m}{\alpha}\). It turns out that this algorithm produces exactly the \tasp{\alpha} of such sequences which was introduced in~\zita{MR3014201}. Particular attention is paid to several natural subclasses of \(\Kggquu{m}{\alpha}\). We indicate that the \tSchur{}-type algorithm under consideration is used in~\zita{141arxiv} to give a parametrization of the solution set of Problem~\mproblem{\iraa{\alpha}}{m}{=}.

 This paper is organized as follows: In \rsect{S1432}, we introduce the classes of sequences of complex \tpqa{matrices} under consideration and recall the notion of \tasp{\alpha}. Similarly as in~\zita{MR3014199}, the algorithm we are going to construct will be based on the concept of \tsraute{}s. For this reason, we summarize in \rsect{S0850} some basic facts on the arithmetics of \tsraute{}s. Furthermore, we introduce the classes of \tftd{} sequences and \tnftd{} sequences of complex \tpqa{matrices} (see \rdefiss{D1658}{D1459}, respectively). The main result of \rsect{S0850} is \rprop{P1442}, which contains several interrelations between important subclasses of \(\Kggqkappaalpha\) and the classes of (nearly) \tftd{} sequences. \rsect{S1441} marks an essential stage for the construction of our algorithm. Let us first mention what happened in~\zita{MR3014199}. There we started with a sequence \(\seq{\su{j}}{j}{0}{2n}\) from \(\Cpq\) and considered its corresponding \tsraute{}. Now we have to ensure that our construction takes into account the influence of \(\alpha\) from the very beginning. In view of the definition of the class \(\Kggqkappaalpha\), we should use, in some way, the operation of \trsas{\alpha} of sequences given via \eqref{s_a}. If \(\seq{\su{j}}{j}{0}{\kappa}\) is a sequence from \(\Cpq\), then the \trsas{\alpha} \(\seq{\sau{j}}{j}{0}{\kappa-1}\) of \(\seq{\su{j}}{j}{0}{\kappa}\) is a sequence of length \(\kappa-1\). Since our construction should start with a sequence of length \(\kappa\), we consider a slight modification of \(\seq{\sau{j}}{j}{0}{\kappa-1}\). This leads us to the \tsraute{} \(\seq{\su{j}^\reza{\alpha}}{j}{0}{\kappa}\) corresponding to the \tsplusat{\alpha} \(\seq{\su{j}^\splusalpha}{j}{0}{\kappa}\). \rsect{S1459} plays an analogous role as \cSect{6} in~\zita{MR3014199}. The main goal is to derive identities between various block \tHankel{} matrices built from the sequences \(\seq{\su{j}^\splusalpha}{j}{0}{\kappa}\) and \(\seq{\su{j}^\reza{\alpha}}{j}{0}{\kappa}\) (see \rthmss{T1640}{T1701}). \rsect{S1356} should be compared with~\zitaa{MR3014199}{\cSect{7}}, where we considered a sequence \(\seq{\su{j}}{j}{0}{2n}\) from \(\Cqq\) together with its corresponding \tsraute{} \(\seq{\su{j}^\rez}{j}{0}{2n}\). From~\zitaa{MR3014199}{\cpropss{8.25}{8.26}} we see that the membership of \(\seq{\su{j}}{j}{0}{2n}\) to the class of \tHnnd{} sequences and some of its distinguished subclasses will be preserved for the sequence \(\seq{-\su{j+2}^\rez}{j}{0}{2(n-1)}\). \rsect{S1356} is aimed to realize a construction suitable for the purposes of this paper, where a one-step algorithm is used. The main result of \rsect{S1356} is \rprop{P1542}, which indicated that the membership of \(\seq{\su{j}}{j}{0}{\kappa}\) to the class of \taSnnd{\alpha} sequences and its prominent subclasses is preserved by the sequence \(\seq{-\su{j+1}^\rezalpha}{j}{0}{\kappa-1}\). In \rsect{S1649}, the shortened negative \tsraute{} \(\seq{-\su{j+1}^\rezalpha}{j}{0}{\kappa-1}\) will be replaced by a slightly modified sequence. In this way, we consider a transformation which ensures that an arbitrary sequence \(\seq{\su{j}}{j}{0}{\kappa}\) belonging to \(\Kggqkappaalpha\) and \(\Kggeqkappaalpha\), respectively, is transformed into a sequence \(\seq{\su{j}^\sntaa{1}{\alpha}}{j}{0}{\kappa-1}\) belonging to \(\Kggqualpha{\kappa-1}\) and \(\Kggequalpha{\kappa-1}\), respectively (see \rtheo{P1546}). The iteration of this transformation, discussed in \rsect{S1649}, leads us to a particular algorithm for finite or infinite sequences of complex \tpqa{matrices} (see \rsect{S1401}). We show that this algorithm preserves the membership of a sequence to the class of \taSnnd{\alpha} sequences and its prominent subclasses (see \rtheo{P1410}). In \rsect{S1402}, we prove that this algorithm produces exactly the \trasp{\alpha}. Central results of the paper are \rthmss{T1016}{T1615}, which contain explicit constructions of the \trasp{\alpha} of a sequence belonging to one of the classes \(\Kggqualpha{m}\) or \(\Kggequalpha{m}\). \rthmss{T1000}{T1658} are 
focused on the 
determination of the \trasp{\alpha} of the described transform of a sequence belonging to one of the classes \(\Kggqkappaalpha\) or \(\Kggeqkappaalpha\). These results indicate that the \trasp{\alpha} can be interpreted as a \tSchur{}-type parametrization. In \rsect{S1009}, we study several aspects of some inverse transformation.

As already mentioned above, the main goal of this paper is to construct a special \tSchur{}-type algorithm for finite or infinite sequences of complex \tpqa{matrices}. We are guided by our former experiences in constructing another version of \tSchur{}-type algorithm for finite or infinite sequences of complex \tpqa{matrices} in~\zita{MR3014199}, which is a two-step algorithm. In this paper, we discuss a one-step algorithm. In the case \(\alpha=0\), such an algorithm was developed already in~\zita{MR1807884}. An essential feature of the \tSchur{}-type algorithm constructed in~\zita{MR3014199} is its intimate connection to the \tcHp{} of \tHnnd{} sequences. We will demonstrate that the \tSchur{}-type algorithm discussed in this paper has a similar connection to the \trasp{\alpha} of \taSnnd{\alpha} sequences. Although the basic techniques in this paper are strongly influenced by the investigations in~\zita{MR3014199}, it should be remarked that the situation is now more complicated. Indeed, now we have to control the interrelation between two \tHnnd{} sequences from \(\Cqq\). This interplay is governed by a real number \(\alpha\). For this reason, we will meet rather new effects in comparison with~\zita{MR3014199}. In order to manage them, a substantial refinement of the techniques used in~\zita{MR3014199} is necessary. In this way, we are forced to introduce various constructions and transforms depending on \(\alpha\).

\section{Right \hasp{\alpha}}\label{S1432}
In this section, we recall the concept of \trasp{\alpha} of finite or infinite sequences which was developed in~\zitas{MR2735313,MR3014201}.

\bremal{R1440}
 Let \(\alpha\in\R\), let \(m\in\NO\), and let \(\seq{\su{j}}{j}{0}{m}\in\Kggqualpha{m}\) (resp.\ \(\Kggequalpha{m}\)). Then we easily see that \(\seq{\su{j}}{j}{0}{l}\in\Kggqualpha{l}\) (resp.\ \(\Kggequalpha{l}\)) for all \(l\in\mn{0}{m}\).
\erema

Let \(\alpha\in\R\). In view of \rrema{R1440}, let \(\Kggqualpha{\infty}\)\index{k@$\Kggqinfalpha$} be the set of all sequences \(\seq{\su{j}}{j}{0}{\infty}\) of complex \tqqa{matrices} such that \(\seq{\su{j}}{j}{0}{m}\in\Kggqualpha{m}\) for all \(m\in\NO\). Further, let \(\Kggequalpha{\infty}\defg\Kggqualpha{\infty}\)\index{k@$\Kggequalpha{\infty}$}. Obviously, for all \(n\in\N\), we have
\[
 \Kggequalpha{2n}
 =\SetaA{\seq{\su{j}}{j}{0}{2n}\in\Hggqu{2n}}{\seq{\sau{j}}{j}{0}{2n-1}\in\Hggequ{2n-1}}
\]
and, for all \(n\in\NO\), furthermore
\bgl{Kgge2n+1}
 \Kggequalpha{2n+1}
 =\SetaA{\seq{\su{j}}{j}{0}{2n+1}\in\Hggequ{2n+1}}{\seq{\sau{j}}{j}{0}{2n}\in\Hggqu{2n}}.
\eg
Let \(\Kgqualpha{0}\defg\Hgqu{0}\)\index{k@$\Kgqualpha{0}$}, and, for all \(n\in\N\), let \(\Kgquu{2n}{\alpha}\) be the set of all sequences \(\seq{\su{j}}{j}{0}{2n}\) of complex \tqqa{matrices} for which the block \tHankel{} matrices \(\Hu{n}\) and \(-\alpha\Hu{n-1}+\Ku{n-1}\) are \tpH{}, i.\,e., \(\Kgquu{2n}{\alpha}\defg\setaa{\seq{\su{j}}{j}{0}{2n}\in\Hgqu{2n}}{\seq{\sau{j}}{j}{0}{2(n-1)}\in\Hgqu{2(n-1)}}\)\index{k@$\Kgquu{2n}{\alpha}$}. Furthermore, for all \(n\in\NO\), let \(\Kgquu{2n+1}{\alpha}\) be the set of all sequences \(\seq{\su{j}}{j}{0}{2n+1}\) of complex \tqqa{matrices} for which the block \tHankel{} matrices \(\Hu{n}\) and \(-\alpha\Hu{n}+\Ku{n}\) are \tpH{}, i.\,e.,
\[%
 \Kgqualpha{2n+1}
 \defg\SetaA{\seq{\su{j}}{j}{0}{2n+1}\in\setseqauu{(\Cqq)}{0}{2n+1}}{\Set{\seq{\su{j}}{j}{0}{2n},\seq{\sau{j}}{j}{0}{2n}}\subseteq\Hgqu{2n}}.
\]
 \index{k@$\Kgqualpha{2n+1}$}Let \(\Kgqinfalpha\)\index{k@$\Kgqinfalpha$} be the set of all sequences \(\seq{\su{j}}{j}{0}{\infty}\) of complex \tqqa{matrices} such that \(\seq{\su{j}}{j}{0}{m}\in\Kgqualpha{m}\) for all \(m\in\NO\). For all \(n\in\NO\), let
\begin{gather}\label{Hggcd}
 \Hggdqu{2n}
 \defg\SetAa{\seq{\su{j}}{j}{0}{2n}\in\Hggqu{2n}}{\Lu{n}=\Oqq},\\
 \label{Kggcd2n}
 \Kggdqualpha{2n}
 \defg\Kggqualpha{2n}\cap\Hggdqu{2n},
\end{gather}
\index{h@$\Hggdqu{2n}$}\index{k@$\Kggdqualpha{2n}$}and let
\bgl{Kggcd2n+1}
 \Kggdqualpha{2n+1}
 \defg\SetaA{\seq{\su{j}}{j}{0}{2n+1}\in\Kggqualpha{2n+1}}{\seq{\sau{j}}{j}{0}{2n}\in\Hggdqu{2n}}.
\eg
\index{k@$\Kggdqualpha{2n+1}$}Furthermore, let \(\Kggdqinfalpha\) be the set of all sequences \(\seq{\su{j}}{j}{0}{\infty}\in\Kggqinfalpha\)\index{k@$\Kggdqinfalpha$} for which there exists a number \(m\in\NO\) such that \(\seq{\su{j}}{j}{0}{m}\in\Kggdqualpha{m}\). For all \(\kappa\in\NO\cup\set{+\infty}\) and all \(m\in\mn{0}{\kappa}\), let
\bgl{Kggcdm}
 \Kggdoqkappaalpha{m}
 \defg\SetaA{\seq{\su{j}}{j}{0}{\kappa}\in\Kggqkappaalpha}{\seq{\su{j}}{j}{0}{m}\in\Kggdqualpha{m}}.
\eg
\index{k@$\Kggdoqkappaalpha{m}$}Obviously,
\bgl{G1023}
 \bigcup_{m=0}^\infty\Kggdoqinfalpha{m}
 =\Kggdqinfalpha.
\eg

\bremal{R1303}
 Let \(\alpha\in\R\), let \(\kappa\in\NO\cup\set{+\infty}\), let \((s_j)_{j=0}^{\kappa}\in\Kggqualpha{\kappa}\) (resp.\ \(\Kggequalpha{\kappa}\)), and let \(A\in\Cqp\). Then \((A^\ad s_jA)_{j=0}^{\kappa}\in\Kgguuu{p}{\kappa}{\alpha}\) (resp.\ \(\Kggeuuu{p}{\kappa}{\alpha}\)).
\erema

We write \(\CHq\)\index{c@$\CHq$} for the set of all \tH{} complex \tqqa{matrices} and we use the L\"owner semi-ordering in \(\CHq\), i.\,e., we write \(A\geq B\)\index{$A\geq B$} (resp.\ \(A>B\)\index{$A>B$}) in order to indicate that \(A\) and \(B\) are \tH{} complex matrices such that \(A-B\) is \tnn{} (resp.\ positive) \tH{}. We denote by \(\Kerna{A}\)\index{n@$\Kerna{A}$} and \(\Bilda{A}\)\index{r@$\Bilda{A}$} the \tKern{} and the \tBild{} of a complex matrix \(A\), respectively. For all \(\beta\in\R\), let \(\lfloor\beta\rfloor\defg\max\setaa{k\in\Z}{k\leq\beta}\)\index{$\lfloor\beta\rfloor$} be the largest integer not greater than \(\beta\). It is useful to state the following technical results:
\begin{lem}[\zitaa{MR3014201}{\clemm{2.9}}]\label{L1738}
 Let \(\alpha\in\R\), let \(\kappa\in\NO\cup\set{+\infty}\), and let \(\seq{\su{j}}{j}{0}{\kappa}\in\Kggqualpha{\kappa}\). Then:
 \benui
  \il{L1738.a} \(\su{j}\in\CHq\) for all \(j\in\mn{0}{\kappa}\) and \(\sau{j}\in\CHq\) for all \(j\in\mn{0}{\kappa-1}\).
  \il{L1738.b} \(\su{2k}\in\Cggq\) for all \(k\in\mn{0}{\frac{\kappa}{2}}\) and \(\sau{2k}\in\Cggq\) for all \(k\in\mn{0}{\frac{\kappa-1}{2}}\).
  \il{L1738.c} \(\Kerna{\su{2k}}\subseteq\Kerna{\su{j}}\) and \(\Bilda{\su{j}}\subseteq\Bilda{\su{2k}}\) for all \(k\in\mn{0}{\frac{\kappa}{2}}\) and all \(j\in\mn{2k}{2\lfloor\frac{\kappa}{2}\rfloor-1}\).
  \il{L1738.d}  \(\Kerna{\sau{2k}}\subseteq\Kerna{\sau{j}}\) and \(\Bilda{\sau{j}}\subseteq\Bilda{\sau{2k}}\) for all \(k\in\mn{0}{\frac{\kappa-1}{2}}\) and all \(j\in\mn{2k}{2\lfloor\frac{\kappa-1}{2}\rfloor-1}\).
 \eenui
\end{lem}

\begin{lem}[see~\zitaa{MR2735313}{\clemmss{4.7}{4.11}}]\label{L1722}
 Let \(\alpha\in\R\) and let \(n\in\NO\). Then \(\Kggequalpha{2n}\subseteq\Hggequ{2n}\). Furthermore, if \(\seq{\su{j}}{j}{0}{2n+1}\in\Kggequalpha{2n+1}\), then \(\seq{\sau{j}}{j}{0}{2n}\in\Hggequ{2n}\).
\end{lem}

 In the sequel, we will need the Moore-Penrose inverse \(A^\MP\)\index{$A^\MP$} of a complex \tpqa{matrix} \(A\). In order to formulate some essential observations on the above introduced sets, it is useful to introduce now some further constructions of matrices, which will play a central role in the following. If \(n\in\N\), if \(\seq{p_j}{j}{1}{n}\) is a sequence of positive integers, and if \(A_j\in\Coo{p_j}{q}\) for all \(j\in\mn{1}{n}\), then let
\[
 \col\seq{A_j}{j}{1}{n}
 \defg
 \bCol
  A_1\\
  A_2\\
  \vdots\\
  A_n
 \eCol.
\]
\index{c@$\col\seq{A_j}{j}{1}{n}$}Let \(\kappa\in\NO\cup\set{+\infty}\) and let \(\seq{\su{j}}{j}{0}{\kappa}\) be a sequence of complex \tpqa{matrices}. We will associate with \(\seq{\su{j}}{j}{0}{\kappa}\) several matrices which we will often need in our subsequent considerations: For all \(l,m\in\NO\) with \(l\leq m\leq\kappa\), let
\begin{align}\label{yz}
 \yuuo{l}{m}{s}&\defg\col\seq{\su{j}}{j}{l}{m}&
 &\text{and}&
 \zuuo{l}{m}{s}&\defg\brow\su{l},\su{l+1},\dotsc,\su{m}\erow.
\end{align}
\index{y@$\yuuo{l}{m}{s}\), \(\yuu{l}{m}$}\index{z@$\zuuo{l}{m}{s}\), \(\zuu{l}{m}$}For all \(n\in\NO\) with \(2n\leq\kappa\), let
\bgl{Hs}
 \Huo{n}{s}
 \defg\matauuo{\su{j+k}}{j,k}{0}{n},
\eg
\index{h@$\Huo{n}{s}\), \(\Hu{n}$}for all \(n\in\NO\) with \(2n+1\leq\kappa\), let
\bgl{Ks}
 \Kuo{n}{s}
 \defg\matauuo{\su{j+k+1}}{j,k}{0}{n},
\eg
\index{k@$\Kuo{n}{s}\), \(\Ku{n}$}and, for all \(n\in\NO\) with \(2n+2\leq\kappa\), let
\bgl{G}
 \Guo{n}{s}
 \defg\matauuo{\su{j+k+2}}{j,k}{0}{n}.
\eg
\index{g@$\Guo{n}{s}\), \(\Gu{n}$}Let
\begin{align}\label{L}
 \Luo{0}{s}&\defg\su{0},&
 \Luo{n}{s}&\defg\su{2n}-\zuuo{n}{2n-1}{s}(\Huo{n-1}{s})^\MP\yuuo{n}{2n-1}{s},
\end{align}
\index{l@$\Luo{n}{s}\), \(\Lu{n}$}and let
\bgl{LL}
 \LLuo{n}{s}
 \defg\Guo{n-1}{s}-\yuuo{1}{n}{s}\su{0}^\MP\zuuo{1}{n}{s}
\eg
\index{l@$\LLuo{n}{s}\), \(\LLu{n}$}for all \(n\in\N\) with \(2n\leq\kappa\). Let
\begin{align}\label{The}
 \Thetauo{0}{s}&\defg\Opq&
 &\text{and}&
 \Thetauo{n}{s}&\defg\zuuo{n}{2n-1}{s}(\Huo{n-1}{s})^\MP\yuuo{n}{2n-1}{s}
\end{align}
\index{t@$\Thetauo{n}{s}\), \(\Thetau{n}$}for all \(n\in\N\) with \(2n-1\leq\kappa\), where \(\Opq\)\index{0@$\Opq$} denotes the zero matrix in \(\Cpq\). In situations in which it is clear which sequence \(\seq{\su{j}}{j}{0}{\kappa}\) of complex matrices is meant, we will write \(\yuu{l}{m}\), \(\zuu{l}{m}\), \(\Hu{n}\), \(\Ku{n}\), \(\Gu{n}\), \(\Lu{n}\), \(\LLu{n}\), and  \(\Thetau{n}\) instead of \(\yuuo{j}{k}{s}\), \(\zuuo{j}{k}{s}\), \(\Huo{n}{s}\), \(\Kuo{n}{s}\), \(\Guo{n}{s}\), \(\Luo{n}{s}\), \(\LLuo{n}{s}\), and \(\Thetauo{n}{s}\), respectively.

Let \(\alpha\in\C\) and let \(\kappa\geq1\). Then the sequence \(\seq{v_j}{j}{0}{\kappa-1}\) given by \(v_j\defg\sau{j}\) and \eqref{s_a} for all \(j\in\mn{0}{\kappa-1}\) plays a key role in our subsequent considerations. We define
\bgl{Thea}
 \Thetaau{n}
 \defg\Thetauo{n}{v}
\eg
\index{t@$\Thetaau{n}$}for all \(n\in\NO\) with \(2n\leq\kappa\),
\bgl{Ha}
 \Hau{n}
 \defg\Huo{n}{v}
\eg
and
\bgl{La}
 \Lau{n}
 \defg\Luo{n}{v}
\eg
\index{h@$\Hau{n}$}\index{l@$\Lau{n}$}for all \(n\in\NO\) with \(2n+1\leq\kappa\),
\[
 \Kau{n}
 \defg\Kuo{n}{v}
\]
\index{k@$\Kau{n}$}for all \(n\in\NO\) with \(2n+2\leq\kappa\), and
\begin{align}\label{yaza}
 \yauu{l}{m}&\defg\yuuo{l}{m}{v}&
 &\text{and}&
 \zauu{l}{m}&\defg\zuuo{l}{m}{v}
\end{align}
\index{y@$\yauu{l}{m}$}\index{z@$\zauu{l}{m}$}for all \(l,m\in\NO\) with \(l\leq m\leq\kappa\). In view of \eqref{H}, \eqref{K}, \eqref{s_a}, and \eqref{Ha}, then
\bgl{-aH+K=Ha}
 -\alpha\Hu{n}+\Ku{n}
 =\Hau{n}
\eg
for all \(n\in\NO\) with \(2n+1\leq\kappa\).

Let us recall useful characterizations of the set \(\Kggequalpha{m}\):
\begin{lem}[\zitaa{MR2735313}{\clemm{4.15}}]\label{L0912}
 Let \(\alpha\in\R\), let \(n\in\NO\), and let \(\seq{\su{j}}{j}{0}{2n+1}\in\Kggqualpha{2n+1}\). Then \(\seq{\su{j}}{j}{0}{2n+1}\in\Kggequalpha{2n+1}\) if and only if \(\Kerna{\Lu{n}}\subseteq\Kerna{\Lau{n}}\).
\end{lem}

\begin{lem}[\zitaa{MR2735313}{\clemm{4.16}}]\label{L0919}
 Let \(\alpha\in\R\), let \(n\in\N\), and let \(\seq{\su{j}}{j}{0}{2n}\in\Kggqualpha{2n}\). Then \(\seq{\su{j}}{j}{0}{2n}\in\Kggequalpha{2n}\) if and only if \(\Kerna{\Lau{n-1}}\subseteq\Kerna{\Lu{n}}\).
\end{lem}

Our current focus on matricial moment problems for intervals of the type \([\alpha,+\infty)\) or \((-\infty,\alpha]\), where \(\alpha\) is an arbitrary real number, motivates us to look for corresponding one-sided \(\alpha\)\nobreakdash-analogues of the \tcHp{} of sequences of complex \tpqa{matrices}. Our above considerations show that, for reasons of symmetry, we can mainly concentrate on the right case. In this case, we use the matrices defined in \eqref{L} and \eqref{La} to introduce the notion in \rdefi{D1021} below, which will turn out to be one of the central objects of this paper.

\bremal{R1702}
 Let \(\alpha\in\C\), let \(\kappa\in\NO\cup\set{+\infty}\), and let \(\seq{\su{j}}{j}{0}{\kappa}\) be a sequence of complex \tpqa{matrices}. Then one can easily see that there is a unique sequence \(\seq{\Spu{j}}{j}{0}{\kappa}\) of complex \tpqa{matrices} such that \(\su{2k}=\Thetau{k}+\Spu{2k}\) for all \(k\in\NO\) with \(2k\leq\kappa\) and \(\su{2k+1}=\alpha\su{2k}+\Thetaau{k}+\Spu{2k+1}\) for all \(k\in\NO\) with \(2k+1\leq\kappa\). In particular, we see that \(\Spu{2k}=\Lu{k}\) for all \(k\in\NO\) with \(2k\leq\kappa\) and moreover \(\Spu{2k+1}=\Lau{k}\) for all \(k\in\NO\) with \(2k+1\leq\kappa\).
\erema

\rrema{R1702} leads us to the following notion, which was introduced in~\zita{MR3014201}.
\bdefil{D1021}
 Let \(\alpha\in\C\), let \(\kappa\in\NO\cup\set{+\infty}\), and let \(\seq{\su{j}}{j}{0}{\kappa}\) be a sequence of complex \tpqa{matrices}. Then the sequence \(\seq{\Spu{j}}{j}{0}{\kappa}\)\index{q@$\seq{\Spu{j}}{j}{0}{\kappa}$} given by \(\Spu{2k}\defg\Lu{k}\) for all \(k\in\NO\) with \(2k\leq\kappa\) and by \(\Spu{2k+1}\defg\Lau{k}\) for all \(k\in\NO\) with \(2k+1\leq\kappa\) is called the \emph{\trasp{\alpha} of \(\seq{\su{j}}{j}{0}{\kappa}$}. In the case \(\alpha=0\), the sequence \(\seq{\Spu{j}}{j}{0}{\kappa}\) is simply called the \emph{\trSpa{\seq{\su{j}}{j}{0}{\kappa}}}.
\edefi

\bremal{R0929}
 Let \(\alpha\in\C\), let \(\kappa\in\NO\cup\set{+\infty}\), and let \(\seq{\Spu{j}}{j}{0}{\kappa}\) be a sequence of complex \tpqa{matrices}. Then it can be immediately checked by induction that there is a unique sequence \(\seq{\su{j}}{j}{0}{\kappa}\) from \(\Cpq\) such that \(\seq{\Spu{j}}{j}{0}{\kappa}\) is the \traspa{\alpha}{\seq{\su{j}}{j}{0}{\kappa}}, namely the sequence \(\seq{\su{j}}{j}{0}{\kappa}\) recursively given by \(\su{2k}=\Thetau{k}+\Spu{2k}\) for all \(k\in\NO\) with \(2k\leq\kappa\) and \(\su{2k+1}=\alpha\su{2k}+\Thetaau{k}+\Spu{2k+1}\) for all \(k\in\NO\) with \(2k+1\leq\kappa\).
\erema

\bremal{R1516}
 Let \(\alpha\in\C\), let \(\kappa\in\NO\cup\set{+\infty}\), and let \(\seq{\su{j}}{j}{0}{\kappa}\) be a sequence of complex \tpqa{matrices}. Denote by \(\seq{\Spu{j}}{j}{0}{\kappa}\) the \traspa{\alpha}{\seq{\su{j}}{j}{0}{\kappa}}. For all \(m\in\mn{0}{\kappa}\), then \(\seq{\Spu{j}}{j}{0}{m}\) is exactly the \traspa{\alpha}{\seq{\su{j}}{j}{0}{m}} .
\erema

The following result shows that the membership of a sequence \(\seq{\su{j}}{j}{0}{\kappa}\) of complex \tqqa{matrices} to one of the classes \(\Kggqkappaalpha\) and \(\Kggeqkappaalpha\) can be effectively characterized in terms of its \trasp{\alpha}.
\begin{thm}[see~\zitaa{MR3014201}{\ctheo{4.12}}]\label{121.P1337}
 Let \(\alpha\in\R\), let \(\kappa\in\NO\cup\set{+\infty}\), and let \(\seq{\si{j}}{j}{0}{\kappa}\)  be a sequence of complex \tqqa{matrices} with \tasp{\alpha} \((\Spu{j})_{j=0}^\kappa\). Then:
 \benui
  \il{121.P1337.b} The sequence \(\seq{\si{j}}{j}{0}{\kappa}\) belongs to \(\Kggquu{\kappa}{\alpha}\) if and only if \(\Spu{j}\in\Cggq\) for all \(j\in\mn{0}{\kappa}\) and, in the case \(\kappa\geq2\), furthermore \(\Kerna{\Spu{j}}\subseteq\Kerna{\Spu{j+1}}\) for all \(j\in\mn{0}{\kappa-2}\).
  \il{121.P1337.c} The sequence \(\seq{\si{j}}{j}{0}{\kappa}\) belongs to \(\Kggequu{\kappa}{\alpha}\) if and only if \(\Spu{j}\in\Cggq\) for all \(j\in\mn{0}{\kappa}\) and, in the case \(\kappa\geq1\), furthermore \(\Kerna{\Spu{j}}\subseteq\Kerna{\Spu{j+1}}\) for all \(j\in\mn{0}{\kappa-1}\).
  \il{121.P1337.d} The sequence \(\seq{\si{j}}{j}{0}{\kappa}\) belongs to \(\Kgquu{\kappa}{\alpha}\) if and only if \(\Spu{j}\in\Cgq\) for all \(j\in\mn{0}{\kappa}\).
 \eenui
\end{thm}

At the end of this section, it should be mentioned that, in the case \(\alpha=0\), there is a particular parametrization of sequences belonging to the class \(\Kgqinfalpha\) which originates in Yu.~M.~Dyukarev's paper~\zita{MR2053150}, where the moment problem~\mproblem{[0,+\infty)}{\infty}{=} was studied. One of his main results is a generalization of a classical criterion due to \tStieltjes{}~\zitas{MR1508159,MR1508160} for the indeterminacy of this moment problem. Yu.~M.~Dyukarev had to look for a convenient matricial generalization of the parameter sequence which \tStieltjes{} obtained from the considerations of particular continued fractions associated with the sequence \(\seq{\su{j}}{j}{0}{\infty}\). In this way, Yu.~M.~Dyukarev found an interesting inner parametrization of sequences belonging to \(\Kgqinfalpha\). This parametrization was reconsidered by the authors in~\zitaa{MR3133464}{\cdefn{8.2}}. The main theme of~\zitaa{MR3133464}{\cSect{8}} was to state interrelations between the Dyukarev-\tStieltjes{} parametrization and our \trasp{0} introduced in \rdefi{D1021}.

\section{The concept of \tsraute{s}}\label{S0850}
The concept used in this section of constructing a special transformation for finite and infinite sequences of complex \tpqa{matrices} is presented in~\zita{MR3014197}. The cited paper deals with the question of invertibility as it applies to finite and infinite sequences of complex \tpqa{matrices}. Two special notions, introduced there will prove to be of particular importance throughout this paper. That's why we recall the definitions.
  
\begin{defn}[\zitaa{MR3014197}{\cdefn{4.13}}]\label{D1430}
 Let \(\kappa \in \NO \cup \set{+\infty}\) and let \(\seq{\su{j}}{j}{0}{\kappa}\) be a sequence of complex \tpqa{matrices}. The sequence \(\seq{\su{j}^\rez}{j}{0}{\kappa}\)\index{$\seq{\su{j}^\rez}{j}{0}{\kappa}$} given by \(s_0^\rez\defg\su{0}^\MP\) and \(s_j^\rez\defg-\su{0}^\MP \sum_{l=0}^{j-1} s_{j-l} \su{l}^\rez\) for all \(j\in\mn{1}{\kappa}\) is called the \emph{\tsrautea{\(\seq{\su{j}}{j}{0}{\kappa}$}}.
\end{defn}

\bremal{R1022}
 Let \(\kappa\in\NO\cup\set{+\infty}\) and let \(\seq{\su{j}}{j}{0}{\kappa}\) be a sequence of complex \tpqa{matrices} with \tsraute{} \(\seq{\su{j}^\rez}{j}{0}{\kappa}\). For all \(m\in\mn{0}{\kappa}\) then \(\seq{\su{j}^\rez}{j}{0}{m}\) is the \tsrautea{\(\seq{\su{j}}{j}{0}{m}$}.
\erema

\begin{defn}[\zitaa{MR3014197}{\cdefn{4.3}}]\label{D1658}
 Let \(\kappa\in\NO\cup\set{+\infty}\) and let \(\seq{\su{j}}{j}{0}{\kappa}\) be a sequence of complex \tpqa{matrices}. We then say that \(\seq{\su{j}}{j}{0}{\kappa}\) is \emph{\tftd{}} if \(\Kerna{\su{0}}\subseteq\bigcap_{j=0}^\kappa\Kerna{\su{j}}\) and \(\bigcup_{j=0}^\kappa\Bilda{\su{j}}\subseteq\Bilda{\su{0}}\). The set of all \tftd{} sequences \(\seq{\su{j}}{j}{0}{\kappa}\) of complex \tpqa{matrices} will be denoted by \(\Dpqu{\kappa}\)\index{d@$\Dpqu{\kappa}$}.
\end{defn}

\begin{rem}\label{R1446}
 Let \(\kappa\in\NO\cup\set{+\infty}\) and let \(\seq{\su{j}}{j}{0}{\kappa}\) be a sequence of complex \tpqa{matrices}. Then \(\seq{\su{j}}{j}{0}{\kappa}\in\Dpqu{\kappa}\) if and only if for \(j\in\mn{0}{\kappa}\) the equations \(\su{j}\su{0}^\MP\su{0}=\su{j}\) and \(\su{0}\su{0}^\MP\su{j}=\su{j}\) hold true (see \rrema{R1631}).
\end{rem}

Given a number \(\kappa \in \NO \cup \set{ +\infty }\) and a sequence \(\seq{\su{j}}{j}{0}{\kappa}\) of complex \tpqa{matrices}, we consider, for all \(m \in \mn{0}{\kappa}\), the triangular block Toeplitz matrices \(\Ssuo{m}{s}\) and \(\SSuo{m}{s}\) defined by
\begin{align}\label{S}%
      \Ssuo{m}{s}
      &\defg
      \bma
        \su{0}    & 0               & 0             & \hdots    & 0   \\
        \su{1}    & \su{0}      & 0             & \hdots    & 0   \\
        \su{2}    & \su{1}      & \su{0}    & \hdots    & 0   \\
        \vdots        & \vdots          & \vdots        & \ddots    & \vdots  \\
        \su{m}    & \su{m-1}  & \su{m-2}& \hdots    & \su{0}
      \ema,&%
   \SSuo{m}{s}
   &\defg
   \bma
     \su{0}    & \su{1}      & \su{2}    & \hdots    & \su{m}   \\
     0             & \su{0}      & \su{1}    & \hdots    & \su{m-1}   \\
     0             & 0               & \su{0}    & \hdots    & \su{m-2}  \\
     \vdots        & \vdots          & \vdots        & \ddots    & \vdots  \\
     0             & 0               & 0             & \hdots    & \su{0}
   \ema,
\end{align}
\index{s@$\Ssuo{m}{s}\), \(\Ssu{m}$}\index{s@$\SSuo{m}{s}\), \(\SSu{m}$}respectively. Whenever it is clear which sequence is meant, we will write \(\Ssu{m}\) and \(\SSu{m}\) instead of \(\Ssuo{m}{s}\) and \(\SSuo{m}{s}\), respectively. Furthermore, for all \(m\in\mn{0}{\kappa}\), we set
\begin{align}\label{Srez}
 \Ssu{m}^\rez&\defg\Ssuo{m}{r}&
 &\text{and}&
 \SSu{m}^\rez&\defg\SSuo{m}{r},
\end{align}
\index{s@$\Ssu{m}^\rez$}\index{s@$\SSu{m}^\rez$}where \(\seq{r_j}{j}{0}{\kappa}\) denotes the \tsrautea{\(\seq{\su{j}}{j}{0}{\kappa}$}, i.\,e., \(r_j\defg\su{j}^\rez\) for all \(j\in\mn{0}{\kappa}\). Now we introduce some block \tHankel{} matrices. We will use the notations
\bgl{Hrez}
 \Hu{n}^\rez
 \defg\matauuo{\su{j+k}^\rez}{j,k}{0}{n}
\eg
\index{h@$\Hu{n}^\rez$}for all \(n\in\NO\) with \(2n\leq\kappa\) and \(\Ku{n}^\rez\defg\matauuo{\su{j+k+1}^\rez}{j,k}{0}{n}\)\index{k@$\Ku{n}^\rez$} for all \(n\in\NO\) with \(2n+1\leq\kappa\). Furthermore, let
\bgl{Grez}
 \Gu{n}^\rez
 \defg\matauuo{\su{j+k+2}^\rez}{j,k}{0}{n}
\eg
\index{g@$\Gu{n}^\rez$}for all \(n\in\NO\) with \(2n+2\leq\kappa\) and let
\begin{align}\label{yzrez}
 \yuu{l}{m}^\rez
 &\defg
 \col\seq{\su{j}^\rez}{j}{l}{m}&
 &\text{and}&
 \zuu{l}{m}^\rez
 &\defg\brow\su{l}^\rez,\su{l+1}^\rez,\dotsc,\su{m}^\rez\erow
\end{align}
\index{y@$\yuu{l}{m}^\rez$}\index{z@$\zuu{l}{m}^\rez$}for all \(l,m\in\NO\) with \(l\leq m\leq\kappa\). In the following, \(\Iq\)\index{i@$\Iq$} stands for the identity matrix in \(\Cqq\) and we will write \(A\kp B\) for the Kronecker product \(A\kp B \defg\brow a_{jk} B\erow_{\substack{j=1,\dotsc,p\\k=1,\dotsc, q}}\)\index{$A\kp B$} of two complex matrices 
\(A =\brow a_{jk}\erow_{\substack{j=1,\ldots ,p\\ k=1,\ldots , q}} \in\Cpq\) and \(B\in\Coo{r}{s}\). Observe that, for all \(m\in\N\) and all \(B\in\Coo{r}{s}\), then \(\Iu{m}\otimes B\) coincides with the block diagonal matrix \(\diag(B,\dotsc,B)\) with \(m\) diagonal blocks \(B\).

\bpropl{P1256}
 Let \(\kappa\in\NO\cup\set{+\infty}\) and let \(\seq{\su{j}}{j}{0}{\kappa}\in\Dpqu{\kappa}\). For all \(m\in\mn{0}{\kappa}\), then \(\seq{\su{j}}{j}{0}{m}\) belongs to \(\Dpqu{m}\) and the equations
 \begin{align}\label{P1256.A}
  \Ssu{m}^\MP&=\Ssu{m}^\rez,&
  \SSu{m}^\MP&=\SSu{m}^\rez
 \end{align}
 and
 \begin{align*}
  \Ssu{m}\Ssu{m}^\MP&=I_{m+1}\kp (s_0s_0^\MP)=\SSu{m} \SSu{m}^\MP,&
  \Ssu{m}^\MP \Ssu{m}&=I_{m+1}\kp (s_0^\MP s_0)=\SSu{m}^\MP \SSu{m}
 \end{align*}
 hold true. If, furthermore, \(p=q\) and if \(\Bilda{\su{0}^\ad}=\Bilda{\su{0}}\), then \(\Ssu{m}\Ssu{m}^\MP=\Ssu{m}^\MP\Ssu{m}\) and \(\SSu{m}\SSu{m}^\MP=\SSu{m}^\MP\SSu{m}\) for all \(m\in\mn{0}{\kappa}\).
\eprop
\bproof
 Use~\zitaa{MR3014197}{\crema{4.6}, \ctheo{4.21}, and \clemm{3.6}}.
\eproof

The equations in \eqref{P1256.A} are characteristic for the class \(\Dpqu{m}\), in some sense:
\bpropl{P1437}
 Let \(m\in\NO\) and let \(\seq{\su{j}}{j}{0}{m}\) be a sequence of complex \tpqa{matrices}. Suppose that there is a sequence \(\seq{t_j}{j}{0}{m}\) of complex \tqpa{matrices} such that \(\Ssu{m}^\MP=\Ssuo{m}{t}\) or \(\SSu{m}^\MP=\SSuo{m}{t}\) holds true. Then \(\seq{\su{j}}{j}{0}{m}\) belongs to \(\Dpqu{m}\).
\eprop
\bproof
 Use~\zitaa{MR3014197}{\cdefn{2.2}, \cnota{2.5}, and \cpropss{2.18}{4.4}}.
\eproof

In anticipation of later applications to \tHnnd{} sequences of matrices, we introduced in~\zitaa{MR3014199}{\cdefn{4.16}} a slight modification of first term domination.
\begin{defn}\label{D1459}
 Let \(m\in\N\). A sequence \(\seq{\su{j}}{j}{0}{m}\) of complex \tpqa{matrices} with \(\seq{\su{j}}{j}{0}{m-1}\in\Dpqu{m-1}\) is called \emph{\tnftd{}}. The set of all \tnftd{} sequences  \(\seq{\su{j}}{j}{0}{m}\) will be denoted by \(\Dtpqu{m}\)\index{d@$\Dtpqkappa$}. We also set \(\Dtpqu{0} \defg \Dpqu{0}\) and \(\Dtpqinf\defg\Dpqinf\).
\end{defn}

Now we state some inclusions between the classes of sequences of complex \tqqa{matrices} introduced above.
\bpropl{P1442}
 Let \(\alpha\in\R\) and let \(\kappa\in\NO\cup\set{+\infty}\). Then:
 \benui
  \il{P1442.a} \(\Kggequalpha{\kappa}\subseteq\Hggequ{\kappa}\subseteq\Dqqu{\kappa}\subseteq\Dtqqu{\kappa}\).
  \il{P1442.b} \(\Kggqualpha{2\kappa}\subseteq\Hggqu{2\kappa}\subseteq\Dtqqu{2\kappa}\).
  \il{P1442.c} \(\Kggequalpha{\kappa}\subseteq\Kggqualpha{\kappa}\subseteq\Dtqqu{\kappa}\).
  \il{P1442.d} \(\Kgqualpha{\kappa}\cup\Kggdqkappaalpha\subseteq\Kggequalpha{\kappa}\).
 \eenui
\eprop
\bproof
 For all \(n\in\NO\) we have \(\Kggequalpha{2n}\subseteq\Hggequ{2n}\) by \rlemm{L1722} and \(\Kggequalpha{2n+1}\subseteq\Hggequ{2n+1}\) by \eqref{Kgge2n+1}. In view of~\zitaa{MR3014199}{\cprop{4.24}} and the \rdefiss{D1459}{D1658}, hence \(\Kggequalpha{m}\subseteq\Hggequ{m}\subseteq\Dqqu{m}\subseteq\Dtqqu{m}\) for all \(m\in\NO\). We have \(\Kggqualpha{0}\subseteq\Hggqu{0}\subseteq\Dtqqu{0}\) by the definition of \(\Kggqualpha{0}\), \(\Hggqu{0}\), and \(\Dtqqu{0}\) and, for all \(n\in\N\), furthermore \(\Kggqualpha{2n}\subseteq\Hggqu{2n}\subseteq\Dtqqu{2n}\) by \eqref{Kgg2n} and~\zitaa{MR3014199}{\cprop{4.25}}. Because of the definition of \(\Kggqinfalpha\) and \(\Hggqinf\), hence \(\Kggqinfalpha\subseteq\Hggqinf\), which, in view of the definition of \(\Kggeqinfalpha\), \(\Hggeqinf\), and \(\Dtqqinf\) and~\zitaa{MR3014199}{\cprop{4.24}}, implies \(\Kggeqinfalpha\subseteq\Hggeqinf\subseteq\Dqqinf\subseteq\Dtqqinf\) and \(\Kggqinfalpha\subseteq\Hggqinf\subseteq\Dtqqinf\). Thus,~\eqref{P1442.a} and~\eqref{P1442.b} are proved. From the definition of \(\Kggeqkappaalpha\) and \rrema{R1440} we conclude \(\Kggequalpha{\kappa}\subseteq\Kggqualpha{\kappa}\). Now let \(n\in\NO\) and \(\seq{\su{j}}{j}{0}{2n+1}\in\Kggqualpha{2n+1}\). Then we get \(\set{\seq{\su{j}}{j}{0}{2n},\seq{\sau{j}}{j}{0}{2n}}\subseteq\Hggqu{2n}\subseteq\Dtqqu{2n}\) by \eqref{Kgg2n+1} and~\eqref{P1442.b}. In the case \(n=0\), we have then \(\seq{\su{j}}{j}{0}{0}\in\Dqqu{0}\), which, in view of \rdefi{D1459}, implies \(\seq{\su{j}}{j}{0}{1}\in\Dtqqu{1}\). Now suppose \(n\geq1\). Because of \eqref{s_a} and \rdefi{D1459}, we obtain then \(\set{\seq{\su{j}}{j}{0}{2n-1},\seq{-\alpha\su{j}+\su{j+1}}{j}{0}{2n-1}}\subseteq\Dqqu{2n-1}\), which, in view of \rdefi{D1658}, implies
 \bsp
  \Kerna{\su{0}}
  &\subseteq\bigcap_{j=0}^{2n-1}\Kerna{\su{j}}
  \subseteq\lek\bigcap_{j=0}^{2n-1}\Kerna{\su{j}}\rek\cap\Kerna{-\alpha\su{0}+\su{1}}\\
  &\subseteq\lek\bigcap_{j=0}^{2n-1}\Kerna{\su{j}}\rek\cap\Kerna{-\alpha\su{2n-1}+\su{2n}}
  \subseteq\bigcap_{j=0}^{2n}\Kerna{\su{j}}
 \esp
 and, analogously, \(\bigcup_{j=0}^{2n}\Bilda{\su{j}}\subseteq\Bilda{\su{0}}\). Hence, in view of \rdefi{D1658}, we have then \(\seq{\su{j}}{j}{0}{2n}\in\Dqqu{2n}\). Consequently, because of \rdefi{D1459}, we get \(\seq{\su{j}}{j}{0}{2n+1}\in\Dtqqu{2n+1}\). Taking additionally into account~\eqref{P1442.b}, we have thus proved~\eqref{P1442.c}. Furthermore,  we know from~\zitaa{MR3014201}{\cpropss{2.20}{5.9}} that~\eqref{P1442.d} holds true. 
\eproof

\section{The \hsplusalphat{} of a matrix sequence}\label{S1441}
This section is of technical nature. We study a particular transform of finite or infinite sequences of \tpqa{matrices}. This transform plays an intermediate role in our construction of a \tSchur{}-\tStieltjes{}-type algorithm for sequences of matrices.

\begin{defn}\label{D1455}
 Let \(\alpha\in\C\), let \(\kappa\in\NO\cup\set{+\infty}\), and let \(\seq{\su{j}}{j}{0}{\kappa}\) be a sequence of complex \tpqa{matrices}. Then we call the sequence \(\seq{\su{j}^\splusalpha}{j}{0}{\kappa}\)\index{$\seq{\su{j}^\splusalpha}{j}{0}{\kappa}$} which is given, for all \(j\in\mn{0}{\kappa}\), by \(\su{j}^\splusalpha\defg-\alpha\su{j-1}+\su{j}\) where \(\su{-1}\defg\Opq\), the \emph{\tsplusalphata{\seq{\su{j}}{j}{0}{\kappa}}}.
\end{defn}

Obviously, the \tsplusalphata{\seq{\su{j}}{j}{0}{\kappa}} is connected with the sequence \(\seq{\sau{j}}{j}{0}{\kappa-1}\) given in \eqref{s_a} via \(\su{j+1}^\splusalpha=\sau{j}\) for all \(j\in\mn{0}{\kappa-1}\). Furthermore, we have
\bgl{[+]_0}
 \su{0}^\splusalpha
 =\su{0}.
\eg
Let \(n\in\NO\). Then, let
\bgl{T}
 \Tqu{n}
 \defg\matauuo{\Kronu{j,k+1}\Iq}{j,k}{0}{n}
\eg
\index{t@$\Tqu{n}$}where \(\Kronu{lm}\) is the Kronecker delta given by \(\Kronu{lm}\defg1\) if \(l=m\) and \(\Kronu{lm}\defg0\) if \(l\neq m\)\index{d@$\Kronu{lm}$}. Obviously, the function \(\Rqu{n}\colon\C\to\Coo{(n+1)q}{(n+1)q}\)\index{r@$\Rqu{n}$} defined by
\[
 \Rqua{n}{z}
 \defg(\Iu{(n+1)q}-z\Tqu{n})^\inv
\]
admits the block representation
\bgl{Rb}
 \Rqua{n}{z}
 =
 \bMat
  \Iq & 0 & 0 &\hdots & 0\\
  z\Iq & \Iq & 0 &\hdots & 0\\
 z^2\Iq & z\Iq & \Iq&\hdots & 0\\
  \vdots & \vdots & \vdots & \ddots & \vdots\\
  z^n\Iq & z^{n-1}\Iq & z^{n-2}\Iq & \hdots &\Iq
 \eMat
\eg
for all \(z\in\C\).

\bremal{R0816:A-C}
 Let \(\alpha\in\C\), let \(\kappa\in\NO\cup\set{+\infty}\), and let \((s_j)_{j=0}^\kappa\) be a sequence of complex \tpqa{matrices} with \tsplusalphat{} \(\seq{t_j}{j}{0}{\kappa}\). Then:
 \benui
  \il{R0816.a} In view of \rdefi{D1455}, for all \(j\in\mn{0}{\kappa}\), the equation \(\su{j}=\sum_{l=0}^j \alpha^{j-l}t_l\) holds true and, in particular, \(\Bilda{\su{j}}\subseteq\sum_{l=0}^j\Bilda{t_l}\) and \(\bigcap_{l=0}^j\Kerna{t_l}\subseteq\Kerna{\su{j}}\).
  \il{R0816.b} In view of \eqref{Rb}, \eqref{S}, and~\eqref{R0816.a}, for all \(m\in\mn{0}{\kappa}\), the block Toeplitz matrices \(\Ssu{m}^\sra{\alpha}\defg\Ssuo{m}{t}\)\index{s@$\Ssu{m}^\sra{\alpha}$} and \(\SSu{m}^\sra{\alpha}\defg\SSuo{m}{t}\)\index{s@$\SSu{m}^\sra{\alpha}$} admit the representations
  \(
   \swu{m}^\sra{\alpha}
   =\swu{m}[\Rqu{m} (\alpha)]^{-1}
   =[\Rpua{m}{\alpha}]^\inv\Ssu{m}
  \)
  and
  \(
   \neu{m}^\sra{\alpha}
   =\SSu{m}[\Rqua{m}{\ko\alpha}]^\invad
   =[\Rpu{m} (\ko\alpha)]^{-\ast}\neu{m}.
  \)
  \il{R0816.c} In view of~\eqref{R0816.a}, \eqref{[+]_0}, and \rdefi{D1455}, the inclusion \(\bigcup_{j=0}^\kappa\Bilda{t_j}\subseteq\Bilda{t_0}\) holds true if and only if \(\bigcup_{j=0}^\kappa\Bilda{\su{j}}\subseteq\Bilda{\su{0}}\), and, furthermore, \(\Kerna{t_0}\subseteq \bigcap_{j=0}^\kappa\Kerna{t_j}\) if and only if \(\Kerna{\su{0}}\subseteq\bigcap_{j=0}^\kappa\Kerna{\su{j}}\).
 \eenui
\erema

\bremal{R0816:D-H}
 Let \(\alpha\in\C\), let \(\kappa\in\NO\cup\set{+\infty}\), and let \((s_j)_{j=0}^\kappa\) be a sequence of complex \tpqa{matrices} with \tsplusalphat{} \(\seq{t_j}{j}{0}{\kappa}\). Then:
 \benui
  \il{R0816.d} \rdefi{D1658} and \rremp{R0816:A-C}{R0816.c} show that \(\seq{t_j}{j}{0}{\kappa}\in\Dpqu{\kappa}\) if and only if \(\seq{\su{j}}{j}{0}{\kappa}\in\Dpqu{\kappa}\).
  \il{R0816.e} \rdefi{D1459} and~\eqref{R0816.d} show that \(\seq{t_j}{j}{0}{\kappa}\in\Dtpqu{\kappa}\) if and only if \(\seq{\su{j}}{j}{0}{\kappa}\in\Dtpqu{\kappa}\).
  \il{R0816.f} By \rdefi{D1455}, then \(\seq{t_j^\ad}{j}{0}{\kappa}\) is exactly the \tsplusata{\ko\alpha}{\seq{\su{j}^\ad}{j}{0}{\kappa}}.
  \il{R0816.g} If \(p=q\), in view of~\eqref{R0816.f} and \rremp{R0816:A-C}{R0816.a}, then \(s_j^\ad=s_j\) for all \(j\in\mn{0}{\kappa}\) if and only if \(\seq{t_j^\ad}{j}{0}{\kappa}\) is the \tsplusata{\ko\alpha}{\seq{\su{j}}{j}{0}{\kappa}}.
 \il{R0816.h} If \(p=q\) and \(\alpha\in\R\), in view of~\eqref{R0816.g}, then \(s_j^\ad=s_j\) for all \(j\in\mn{0}{\kappa}\) if and only if \(t_j^\ad=t_j\) for all \(j\in\mn{0}{\kappa}\).
 \eenui
\erema

\bremal{R1439}
 Let \(\alpha\in\C\), let \(\kappa\in\NO\cup\set{+\infty}\), and let \(n\in\N\). For all \(m\in\mn{1}{n}\), let \(p_m,q_m\in\N\), let \(\seq{\su{j}^{(m)}}{j}{0}{\kappa}\) be a sequence of complex \taaa{p_m}{q_m}{matrices}, let \(\seq{t_j^{(m)}}{j}{0}{\kappa}\) be its \tsplusalphat{}, let \(L_m\in\Coo{p}{p_m}\), and let \(R_m\in\Coo{q_m}{q}\). Then \(\seq{\sum_{m=1}^nL_mt_j^{(m)}R_m}{j}{0}{\kappa}\) is exactly the \tsplusalphata{\seq{\sum_{m=1}^nL_m\su{j}^{(m)}R_m}{j}{0}{\kappa}}.
\erema

\bremal{R1438}
 Let \(\alpha\in\C\), let \(\kappa\in\NO\cup\set{+\infty}\), and let \(n\in\N\). For all \(m\in\mn{1}{n}\), let \(p_m,q_m\in\N\) and let \(\seq{\su{j}^{(m)}}{j}{0}{\kappa}\) be a sequence of complex \taaa{p_m}{q_m}{matrices} with \tsplusalphat{} \(\seq{t_j^{(m)}}{j}{0}{\kappa}\). Then \(\seq{\diag\matauuo{t_j^{(m)}}{m}{1}{n}}{j}{0}{\kappa}\) is exactly the \tsplusalphata{\seq{\diag\matauuo{\su{j}^{(m)}}{m}{1}{n}}{j}{0}{\kappa}}.
\erema

Let \(\alpha\in\C\). In order to prepare the basic construction in \rSect{S1356}, we study the \tsraute{} corresponding to the \tsplusalphat{} of a sequence. Let \(\kappa\in\NO\cup\set{+\infty}\), and let \(\seq{\su{j}}{j}{0}{\kappa}\) be a sequence of complex \tpqa{matrices} with \tsplusalphat{} \(\seq{u_j}{j}{0}{\kappa}\). Then we define \(\seq{\su{j}^\reza{\alpha}}{j}{0}{\kappa}\)\index{$\seq{\su{j}^\reza{\alpha}}{j}{0}{\kappa}$} by
\bgl{reza}
 \su{j}^\reza{\alpha}
 \defg u_j^\rez
\eg
for all \(j\in\mn{0}{\kappa}\), i.\,e., the sequence \(\seq{\su{j}^\reza{\alpha}}{j}{0}{\kappa}\) is defined to be the \tsrautea{the \tsplusalphata{\seq{\su{j}}{j}{0}{\kappa}}}. The following result describes a useful interrelation between the sequence introduced in \eqref{reza} and the \tsraute{} corresponding to the original sequence \(\seq{\su{j}}{j}{0}{\kappa}\):
\begin{lem}\label{L1445}
 Let \(\alpha\in\C\), let \(\kappa\in\NO\cup\set{+\infty}\), and let \((s_j)_{j=0}^{\kappa}\) be a sequence of complex \tpqa{matrices}. For all \(j\in\mn{0}{\kappa}\), then
 \bgl{L1445.A}
  s_j^\reza{\alpha}
  =\sum_{l=0}^j \alpha^{j-l} s_l^\rez.
 \eg
\end{lem}
\begin{proof}
 For \(j=0\) equation \eqref{L1445.A} holds obviously. If \(\kappa =0\), then the proof is finished.
 
 Let \(\kappa\geq 1\). According to \rdefi{D1430}, we have \(s_j^\rez=s_0^\MP s_0 s_j^\rez\) for all \(j\in\mn{0}{\kappa}\) and hence, in view of \eqref{[+]_0}, furthermore
 \[
  s_1^\reza{\alpha}
  =-s_0^\MP  (-\alpha s_0 + s_1) s_0^\rez
  =\alpha s_0^\rez  + s_1^\rez
  =\sum_{l=0}^1 \alpha^{1-l} s_l^\rez .
 \]
 Thus, if \(\kappa =1\), then the proof is complete.
   
 Now suppose \(\kappa \geq 2\). Then the considerations above show that there is a number \(k\in\mn{1}{\kappa-1}\) such that \eqref{L1445.A} holds true for all \(j\in\mn{0}{k}\). Consequently,
 \bsp
  s_{k+1}^\reza{\alpha}
  &=- \sum_{j=0}^{k} s_0^\MP  (-\alpha s_{k-j} + s_{k+1-j}) \sum_{l=0}^j \alpha^{j-l} s_l^\rez\\
  & =s_0^\MP  s_0 \sum_{l=0}^k \alpha^{k+1-l} s_l^\rez  + \sum_{j=0}^{k-1} s_0^\MP s_{k-j} \sum_{l=0}^j \alpha^{j-l+1} s_l^\rez  - \sum_{j=0}^k s_0^\MP  s_{k+1-j}\sum_{l=0}^j \alpha^{j-l} s_l^\rez\\
  &=\sum_{l=0}^k \alpha^{k+1-l} s_l^\rez  + \sum_{j=1}^{k} s_0^\MP  s_{k+1-j} \sum_{l=0}^{j-1} \alpha^{j-l} s_l^\rez  - \sum_{j=0}^k s_0^\MP  s_{k+1-j}\sum_{l=0}^j \alpha^{j-l} s_l^\rez\\
  &=\sum_{l=0}^k \alpha^{k+1-l} s_l^\rez  - \sum_{j=0}^{(k+1)-1} s_0^\MP  s_{k+1-j} s_j^\rez
  =\sum_{l=0}^{k+1} \alpha^{k+1-l} s_l^\rez.
 \esp
 Thus, the assertion follows inductively.
\end{proof}

\begin{rem}\label{R1456}
 Let \(\alpha\in\C\), let \(\kappa\in\NO\cup\set{+\infty}\), and let \((s_j)_{j=0}^\kappa\) be a sequence of complex \tpqa{matrices}. Let the sequence \(\seq{r_j}{j}{0}{\kappa}\) be given by \(r_j\defg\su{j}^\rezalpha\) for all \(j\in\mn{0}{\kappa}\). Then we define, for all \(n\in\mn{0}{\kappa}\), the block Toeplitz matrices \(\Ssu{n}^\reza{\alpha}\defg\Ssuo{n}{r}\)\index{s@$\Ssu{n}^\reza{\alpha}$} and \(\SSu{n}^\reza{\alpha}\defg\SSuo{n}{r}\)\index{s@$\SSu{n}^\reza{\alpha}$}. In view of \eqref{S}, \eqref{Rb}, and \rlemm{L1445}, one can easily see then that
 \begin{align*}
  \swu{n}^\reza{\alpha}&=\lek\Rqu{n} (\alpha)\rek\swu{n}^\rez=\Ssu{n}^\rez\lek\Rpua{n}{\alpha}\rek&
  &\text{and}&
  \neu{n}^\reza{\alpha}&=\lek\Rqua{n}{\ko\alpha}\rek^\ad\SSu{n}^\rez=\neu{n}^\rez\lek\Rpua{n}{\ko\alpha}\rek^\ad
 \end{align*}
 hold true for all \(n\in\mn{0}{\kappa}\). If \((s_j)_{j=0}^\kappa\in\Dpqkappa\), then from \rprop{P1256} we see that, for all \(n\in\mn{0}{\kappa}\) these equations admit the reformulation
 \begin{align*}
  \swu{n}^\reza{\alpha}&=\lek\Rqu{n} (\alpha)\rek\swu{n}^\MP=\Ssu{n}^\MP\lek\Rpua{n}{\alpha}\rek&
  &\text{and}&
  \neu{n}^\reza{\alpha}&=\lek\Rqua{n}{\ko\alpha}\rek^\ad\SSu{n}^\MP=\neu{n}^\MP\lek\Rpua{n}{\ko\alpha}\rek^\ad.
 \end{align*}
\end{rem}

\blemml{L1352}
 Let \(\alpha\in\C\), let \(\kappa\in\NO\cup\set{+\infty}\), and let \(\seq{\su{j}}{j}{0}{\kappa}\) be a sequence of complex \tpqa{matrices}. Then \([\Bilda{\su{0}}]^\bot\subseteq\bigcap_{j=0}^{\kappa}\Kerna{\su{j}^\rezalpha}\) and \(\bigcup_{j=0}^{\kappa}\Bilda{\su{j}^\rezalpha}\subseteq[\Kerna{\su{0}}]^\bot\). In particular, \(\seq{\su{j}^\rezalpha}{j}{0}{\kappa}\) belongs to \(\Dqpu{\kappa}\). 
\elemm
\bproof
 \rremp{R1631}{R1631.a} shows that
 \begin{align}\label{L1352.1}
  \lek\Bilda{\su{0}}\rek^\bot&=\Kerna{\su{0}^\ad}=\Kerna{\su{0}^\MP}&
  &\text{and}&
  \lek\Kerna{\su{0}}\rek^\bot&=\Bilda{\su{0}^\ad}=\Bilda{\su{0}^\MP}.
 \end{align}
 Thus, from \rdefi{D1430} we see that \([\Bilda{\su{0}}]^\bot\subseteq\Kerna{\su{k}^\rez}\) for all \(k\in\mn{0}{\kappa}\). Taking \rlemm{L1445} into account, we then get \([\Bilda{\su{0}}]^\bot\subseteq\bigcap_{j=0}^{\kappa}\Kerna{\su{j}^\rezalpha}\). Furthermore, from \rlemm{L1445}, \rdefi{D1430}, and \eqref{L1352.1} we see that \(\bigcup_{j=0}^{\kappa}\Bilda{\su{j}^\rezalpha}\subseteq\bigcup_{j=0}^{\kappa}\Bilda{\su{j}^\rez}\subseteq\Bilda{\su{0}^\MP}=[\Kerna{\su{0}}]^\bot\). Taking into account that \(\su{0}^\rezalpha=(\su{0}^\splusalpha)^\MP=\su{0}^\MP\) holds true, the proof is complete.
\eproof

 Now we study the \tsplusalphat{} of the sequence introduced in \eqref{reza}.
\blemml{L1453}
 Let \(\alpha\in\C\), let \(\kappa\in\NO\cup\set{+\infty}\), and let \(\seq{\su{j}}{j}{0}{\kappa}\) be a sequence of complex \tpqa{matrices}. Then the \tsplusalphata{\seq{\su{j}^\rezalpha}{j}{0}{\kappa}} is exactly the \tsraute{} \(\seq{\su{j}^\rez}{j}{0}{\kappa}\) corresponding to \(\seq{\su{j}}{j}{0}{\kappa}\).
\elemm
\bproof
 Let the sequence \(\seq{r_j}{j}{0}{\kappa}\) be given by \(r_j\defg\su{j}^\rezalpha\) for all \(j\in\mn{0}{\kappa}\). From \rlemm{L1445} we know then that \(r_j=\sum_{l=0}^j \alpha^{j-l}s_l^\rez\) holds true for all \(j\in\mn{0}{\kappa}\). Using \eqref{[+]_0} and \rdefi{D1455}, we obtain then \(r_0^\splusalpha=\su{0}^\rez\) and, in the case \(\kappa\geq1\), for all \(j\in\mn{1}{\kappa}\), furthermore
 \[
  r_j^\sra{\alpha}
  =-\alpha r_{j-1}+r_j
  =-\alpha\sum_{l=0}^{j-1}\alpha^{(j-1)-l}\su{l}^\rez+\sum_{l=0}^j\alpha^{j-l}\su{l}^\rez
  =\su{j}^\rez.\qedhere
 \]
\eproof

\blemml{L1454}
 Let \(\alpha\in\C\), let \(\kappa\in\NO\cup\set{+\infty}\), and let \(\seq{\su{j}}{j}{0}{\kappa}\) and \(\seq{t_j}{j}{0}{\kappa}\) be sequences of complex \tpqa{matrices}. Then \(\su{j}^\rezalpha=t_j^\rezalpha\) for all \(j\in\mn{0}{\kappa}\) if and only if \(\su{0}\su{0}^\MP\su{j}\su{0}^\MP\su{0}=t_0t_0^\MP t_jt_0^\MP t_0\) for all \(j\in\mn{0}{\kappa}\).
\elemm
\bproof
 According to \rlemmss{L1453}{L1445}, the equation \(\su{j}^\rezalpha=t_j^\rezalpha\) holds for all \(j\in\mn{0}{\kappa}\) if and only if \(\su{j}^\rez=t_j^\rez\) for all \(j\in\mn{0}{\kappa}\). The application of~\zitaa{MR3014197}{\cprop{5.11}} completes the proof.
\eproof

We finish this section with some considerations on the arithmetics of the \tsplusalphat{}.
\blemml{L1451}
 Let \(\alpha\in\C\), let \(\kappa\in\N\cup\set{+\infty}\), and let \(\seq{\su{j}}{j}{0}{\kappa}\) be a sequence of complex \tpqa{matrices}.
 \benui
  \il{L1451.a} If \(\gamma\in\C\), then \(\seq{(\gamma\su{j})^\rezalpha}{j}{0}{\kappa}=\seq{\gamma^\MP\su{j}^\rezalpha}{j}{0}{\kappa}\) and \(\seq{(\gamma^j\su{j})^\rezalpha}{j}{0}{\kappa}=\seq{\gamma^j\su{j}^\rezalpha}{j}{0}{\kappa}\).
  \il{L1451.b} If \(m\in\N\) and \(L\in\Coo{m}{p}\) with \(\Bilda{L^\ad}=\Bilda{\su{0}}\), then \(\seq{(L\su{j})^\rezalpha}{j}{0}{\kappa}=\seq{\su{j}^\rezalpha L^\MP}{j}{0}{\kappa}\).
  \il{L1451.c} If \(n\in\N\) and \(R\in\Coo{q}{n}\) with \(\Bilda{\su{0}^\ad}=\Bilda{R}\), then \(\seq{(\su{j}R)^\rezalpha}{j}{0}{\kappa}=\seq{R^\MP\su{j}^\rezalpha}{j}{0}{\kappa}\).
  \il{L1451.d} If \(m,n\in\N\), \(L\in\Coo{m}{p}\) with \(\Bilda{L^\ad}=\Bilda{\su{0}}\), and \(R\in\Coo{q}{n}\) with \(\Bilda{\su{0}^\ad}=\Bilda{R}\), then \(\seq{(L\su{j}R)^\rezalpha}{j}{0}{\kappa}=\seq{R^\MP\su{j}^\rezalpha L^\MP}{j}{0}{\kappa}\).
  \il{L1451.e} If \(m,n\in\N\), \(U\in\Coo{m}{p}\) with \(U^\ad U=\Ip\), and \(V\in\Coo{q}{n}\) with \(VV^\ad=\Iq\), then \(\seq{(U\su{j}V)^\rezalpha}{j}{0}{\kappa}=\seq{V^\ad\su{j}^\rezalpha U^\ad}{j}{0}{\kappa}\).
  \il{L1451.f} \((\su{j}^\rezalpha)^\ad=t_j^\reza{\ko\alpha}\) for all \(j\in\mn{0}{\kappa}\) where the sequence \(\seq{t_j}{j}{0}{\kappa}\) is given by \(t_j\defg\su{j}^\ad\) for all \(j\in\mn{0}{\kappa}\).
 \eenui
\elemm
\bproof
 \rPart{L1451.a} follows from \rlemm{L1445} and~\zitaa{MR3014197}{\cremass{5.8}{5.9}}. \rPartsss{L1451.b}{L1451.c}{L1451.d} are immediate consequences of \rlemm{L1445} and~\zitaa{MR3014197}{\clemm{5.18}}. In view of \rlemm{L1445} and~\zitaa{MR3014197}{\clemm{5.19}}, we see that~\eqref{L1451.e} holds true. Finally, using \eqref{reza},~\zitaa{MR3014197}{\cprop{5.16}}, and \rremp{R0816:D-H}{R0816.f}, we obtain~\eqref{L1451.f}.
\eproof

\bremal{R1452}
 Let \(\alpha\in\C\), let \(\kappa\in\NO\cup\set{+\infty}\), and let \(n\in\N\). For all \(m\in\mn{1}{n}\), let \(p_m,q_m\in\N\) and let \(\seq{\su{j}^{(m)}}{j}{0}{\kappa}\) be a sequence of complex \taaa{p_m}{q_m}{matrices}. Then from \rrema{R1022}, \rlemm{L1445}, and~\zitaa{MR3014197}{\crema{5.20}} we easily see that \(\seq{\diag\matauuo{(\su{j}^{(m)})^\rezalpha}{m}{1}{n}}{j}{0}{\kappa}=\seq{(\diag\matauuo{\su{j}^{(m)}}{m}{1}{n})^\rezalpha}{j}{0}{\kappa}\).
\erema

\section{Some identities for special block \tHankel{} matrices}\label{S1459}
Similar as in~\zitaa{MR3014199}{\cSect{6}}, the main theme of this section is the investigation of the interplay between various block \tHankel{} matrices. In order to describe a typical situation, let \(n\in\N\) and let \(\seq{\su{j}}{j}{0}{2n}\) be a sequence from \(\Cqq\). Denote by \((u_j)_{j=0}^{2n}\) the \tsplusalphata{\seq{\su{j}}{j}{0}{2n}} and by \((u_j^\rez)_{j=0}^{2n}\) the \tsrautea{\((u_j)_{j=0}^{2n}$}. Then we are going to find formulas which connect various block \tHankel{} matrices built from the sequences \((u_j)_{j=0}^{2n}\) and \((u_j^\rez)_{j=0}^{2n}\). This section should be compared with~\zitaa{MR3014199}{\cSect{6}}, where formulas connecting several block \tHankel{} matrices formed from the sequence \((\su{j})_{j=0}^{2n}\) and its corresponding \tsraute{} \((\su{j}^\rez)_{j=0}^{2n}\) were derived. Our strategy of proof is essentially based on using results from~\zitaa{MR3014199}{\cSect{6}}.

Let \(\vqu{0}\defg\Iq\) and, for all \(n\in\N\), let
\begin{align}\label{vDN}
 \vqu{n}
 &\defg
 \bma
  \Iq\\
  \Ouu{nq}{q}
 \ema,&
 \IOqu{n}
 &\defg
 \bma
  \Iu{nq}\\
  \Ouu{q}{nq}
 \ema,&
 &\text{and}&
 \OIqu{n}
 &\defg
 \bma
  \Ouu{q}{nq}\\
  \Iu{nq}
 \ema.
\end{align}
\index{v@$\vqu{n}$}\index{d@$\IOqu{n}$}\index{$\OIqu{n}$}

\blemml{L1606}
 Let \(\alpha\in\C\), let \(\kappa\in\NO\cup\set{+\infty}\), and let \((s_j)_{j=0}^\kappa\) be a sequence of complex \tpqa{matrices} with \tsplusalphat{} \(\seq{t_j}{j}{0}{\kappa}\).
 \benui
  \il{L1606.a} Let \(n\in\N\) with \(2n\leq\kappa\). Then the block Hankel matrix \(\Hu{n}^\sra{\alpha}\defg\Huo{n}{t}\)\index{h@$\Hu{n}^\sra{\alpha}$} admits the representations
  \bspl{H[+]}
   \Hu{n}^\sra{\alpha}
   &=\lek \Rpua{n}{\alpha}\rek^\inv\Hu{n}\lek\Rqua{n}{\ko\alpha}\rek^\invad+\alpha\OIpu{n}(-\alpha\Hu{n-1} +\Ku{n-1}) \OIqu{n}^\ad\\
   &=\lek \Rpua{n}{\alpha}\rek^\inv\Hu{n}\lek\Rqua{n}{\ko\alpha}\rek^\invad+\alpha\OIpu{n}\Hau{n-1}\OIqu{n}^\ad.
  \esp
  \il{L1606.b} Let \(n\in\NO\) with \(2n+1\leq\kappa\). Then the block Hankel matrix \(\Ku{n}^\sra{\alpha}\defg\Kuo{n}{t}\)\index{k@$\Ku{n}^\sra{\alpha}$} admits the representation \(\Ku{n}^\sra{\alpha}=-\alpha\Hu{n} +\Ku{n}\), i.\,e., \(\Ku{n}^\sra{\alpha}=\Hau{n}\) holds true.
  \il{L1606.c} Let \(n\in\NO\) with \(2n+2\leq\kappa\). Then the block Hankel matrix \(\Gu{n}^\sra{\alpha}\defg\Guo{n}{t}\)\index{g@$\Gu{n}^\sra{\alpha}$} admits the representation \(\Gu{n}^\sra{\alpha}=-\alpha\Ku{n} +\Gu{n}\), i.\,e., \(\Gu{n}^\sra{\alpha}=\Kau{n}\) holds true.
  \il{L1606.d} Let \(n\in\N\) with \(2n\leq\kappa\). Then the matrix \(\LLu{n}^\splusalpha\defg\LLuo{n}{t}\)\index{l@$\LLu{n}^\splusalpha$} admits the representation \(\LLu{n}^\splusalpha=\Kau{n-1}-\yauu{0}{n-1}\su{0}^\MP\zauu{0}{n-1}\).
 \eenui
\elemm
\begin{proof}
 Using the block representation \(\Hu{n}=\btmat\su{0}&\zuu{1}{n}\\\yuu{1}{n}&\Gu{n-1}\etmat\) of \(\Hu{n}\), one can easily see that
 \bspl{L1606.1}
  H_n^\sra{\alpha}
  &=\bma
  s_0 & -\alpha z_{0,n-1} + z_{1,n}\\
  -\alpha y_{0,n-1} + y_{1,n} & -\alpha K_{n-1} +\Gu{n-1}
  \ema\\
   & =-\alpha
  \bma
  \Oqq  & \Ouu{q}{nq}\\
  y_{0,n-1} & K_{n-1}
  \ema  -\alpha \bma
  \Oqq  & z_{0,n-1}\\
  \Ouu{nq}{q}& K_{n-1}
  \ema  + \alpha \bma
  \Oqq  & \Ouu{q}{nq}\\
  \Ouu{nq}{q} & K_{n-1}
  \ema  + H_n\\
  & =-\alpha\Tpu{n}H_n -\alpha H_n\Tqu{n}^\ad + \alpha \OIpu{n} K_{n-1} \OIqu{n}^\ad + H_n.
 \esp
 Taking into account
 \bsp
  \lek\Rpu{n} (\alpha)\rek^{-1} H_n\lek\Rqu{n} (\ko\alpha)\rek^{-\ast}
  & =(\Iu{(n+1)p}- \alpha\Tpu{n}) H_n (\Iu{(n+1)q}-\ko\alpha\Tqu{n})^\ad\\
  & =H_n -\alpha\Tpu{n}H_n - \alpha H_n\Tqu{n}^\ad + \alpha^2\Tpu{n}H_{n}\Tqu{n}^\ad
 \esp
 and \(\Tpu{n}H_n\Tqu{n}^\ad =\OIpu{n} H_{n-1} \OIqu{n}^\ad\), from \eqref{L1606.1} we then conclude that \eqref{H[+]} holds true. Thus,~\eqref{L1606.a} is verified. The proof of the \rpartsss{L1606.b}{L1606.c}{L1606.d} is straightforward.
\end{proof}

In order to state interesting identities for block \tHankel{} matrices, it seems to be useful to introduce some further notation. Let \(\kappa\in\NO\cup\set{+\infty}\) and let \(\seq{\su{j}}{j}{0}{\kappa}\) be a sequence of complex \tpqa{matrices}. For all \(n\in\NO\) and all \(m\in\mn{0}{\kappa}\), we then set
\bgl{Xi}
 \Thetauuo{n}{m}{s}
 \defg
 \begin{cases}
  \su{m}-\su{0}\su{0}^\MP\su{m}\su{0}^\MP\su{0}\incase{n=0}\\
  \diag\brow\Ouu{np}{nq},\su{m}-\su{0}\su{0}^\MP\su{m}\su{0}^\MP\su{0}\erow\incase{n\geq1}
 \end{cases}
\eg
\index{x@$\Thetauuo{n}{m}{s}\), \(\Thetauu{n}{m}$}and write \(\Thetauu{n}{m}\) instead of \(\Thetauuo{n}{m}{s}\) if it is clear which sequence \(\seq{\su{j}}{j}{0}{\kappa}\) of complex matrices is meant.

\bremal{R1433}
 Let \(\kappa\in\NO\cup\set{+\infty}\), let \(\seq{\su{j}}{j}{0}{\kappa}\) be a sequence of complex \tpqa{matrices}, let \(n\in\NO\), and let \(m\in\mn{0}{\kappa}\) be such that \(\Kerna{\su{0}}\subseteq\Kerna{\su{m}}\) and \(\Bilda{\su{m}}\subseteq\Bilda{\su{0}}\). In view of \rpartss{R1631.b}{R1631.c} of \rrema{R1631}, then \(\su{m}\su{0}^\MP\su{0}=\su{m}\) and \(\su{0}\su{0}^\MP\su{m}=\su{m}\). Consequently, \(\Thetauu{n}{m}=\Ouu{(n+1)p}{(n+1)q}\).
\erema

\bremal{R1608}
 Let \(\kappa\in\NO\cup\set{+\infty}\) and let \(\seq{\su{j}}{j}{0}{\kappa}\in\Dpqkappa\). In view of \rdefi{D1658} and \rrema{R1433}, then \(\Thetauu{n}{m}=\Ouu{(n+1)p}{(n+1)q}\) for all \(n\in\NO\) and all \(m\in\mn{0}{\kappa}\).
\erema

Let \(\alpha\in\C\), let \(\kappa\in\NO\cup\set{+\infty}\), and let \(\seq{\su{j}}{j}{0}{\kappa}\) be a sequence of complex \tpqa{matrices}. Furthermore, let the sequence \(\seq{\su{j}^\rezalpha}{j}{0}{\kappa}\) be defined via \eqref{reza}. Then let \(\Hu{n}^\rezalpha\defg\matauuo{\su{j+k}^\rezalpha}{j,k}{0}{n}\)\index{h@$\Hu{n}^\rezalpha$} for all \(n\in\NO\) with \(2n\leq\kappa\), let
\bgl{Kreza}
 \Ku{n}^\rezalpha
 \defg\matauuo{\su{j+k+1}^\rezalpha}{j,k}{0}{n}
\eg
\index{k@$\Ku{n}^\rezalpha$}for all \(n\in\NO\) with \(2n+1\leq\kappa\), and let
\bgl{Greza}
 \Gu{n}^\rezalpha
 \defg\matauuo{\su{j+k+2}^\rezalpha}{j,k}{0}{n}
\eg
\index{g@$\Gu{n}^\rezalpha$}for all \(n\in\NO\) with \(2n+2\leq\kappa\). Further, for all \(l,m\in\NO\) with \(l\leq m\leq\kappa\), let \(\yuu{l}{m}^\rezalpha\defg\col\seq{\su{j}^\rezalpha}{j}{l}{m}\)\index{y@$\yuu{l}{m}^\rezalpha$} and \(\zuu{l}{m}^\rezalpha\defg\brow\su{l}^\rezalpha,\su{l+1}^\rezalpha,\dotsc,\su{m}^\rezalpha\erow\)\index{z@$\zuu{l}{m}^\rezalpha$} and let \(\yuu{l}{m}^\splusalpha\defg\yuuo{l}{m}{t}\)\index{y@$\yuu{l}{m}^\splusalpha$} and \(\zuu{l}{m}^\splusalpha\defg\zuuo{l}{m}{t}\)\index{z@$\zuu{l}{m}^\splusalpha$} where \(\seq{t_j}{j}{0}{\kappa}\) denotes the \tsplusalphata{\seq{\su{j}}{j}{0}{\kappa}}.

\bpropl{L1438}
 Let \(\alpha\in\C\), let \(n\in\NO\), and let \((s_j)_{j=0}^{2n}\in\Dtpqu{2n}\). Then
 \bgl{L1438.A}
  \Hu{n}^\rezalpha+\Ssu{n}^\rezalpha\Hu{n}^\splusalpha\SSu{n}^\rezalpha
  =\yuu{0}{n}^\rezalpha\vpu{n}^\ad+\vqu{n}\zuu{0}{n}^\rezalpha
 \eg
 and
 \bgl{L1438.B}
  \Hu{n}^\splusalpha+\Ssu{n}^\splusalpha\Hu{n}^\rezalpha\SSu{n}^\splusalpha
  =\yuu{0}{n}^\splusalpha\vqu{n}^\ad+\vpu{n}\zuu{0}{n}^\splusalpha+\Thetauu{n}{2n}.
 \eg
 Furthermore, if \(n\geq1\), then
 \begin{multline}\label{L1438.C}
  \Hu{n}^\rezalpha+\Ssu{n}^\MP\lrk\Hu{n}+\alpha\Rpua{n}{\alpha}\OIpu{n}\Hau{n-1}\lek\Rqua{n}{\ko\alpha}\OIqu{n}\rek^\ad\rrk\SSu{n}^\MP\\
  =\Rqua{n}{\alpha}\yuu{0}{n}^\rez\vpu{n}^\ad+\vqu{n}\zuu{0}{n}^\rez\lek\Rpua{n}{\ko\alpha}\rek^\ad
 \end{multline}
 and
 \begin{multline}\label{L1438.D}
  \Hu{n}+\Ssu{n}\Hu{n}^\rezalpha\SSu{n}+\alpha\Rpua{n}{\alpha}\OIpu{n}\Hau{n-1}\lek\Rqua{n}{\ko\alpha}\OIqu{n}\rek^\ad\\
  =\yuu{0}{n}\lek\Rqua{n}{\ko\alpha}\vqu{n}\rek^\ad+\Rpua{n}{\alpha}\vpu{n}\zuu{0}{n}+\Thetauu{n}{2n}.
 \end{multline}
\eprop
\bproof
 Because of \(\seq{\su{j}}{j}{0}{2n}\in\Dtpqu{2n}\) and \rrema{R1446}, we have
 \begin{multline*}
  \su{2n}^\splusalpha-\su{0}^\splusalpha(\su{0}^\splusalpha)^\MP\su{2n}^\splusalpha(\su{0}^\splusalpha)^\MP\su{0}^\splusalpha\\
  =-\alpha\su{2n-1}+\su{2n}+\alpha\su{0}\su{0}^\MP\su{2n-1}\su{0}^\MP\su{0}-\su{0}\su{0}^\MP\su{2n}\su{0}^\MP\su{0}
  =\su{2n}-\su{0}\su{0}^\MP\su{2n}\su{0}^\MP\su{0}
 \end{multline*}
 and, according to \rremp{R0816:D-H}{R0816.e}, the sequence \(\seq{\su{j}^\splusalpha}{j}{0}{2n}\) belongs to \(\Dtpqu{2n}\). Thus, \eqref{L1438.A} and \eqref{L1438.B} immediately follow from~\zitaa{MR3014199}{\ctheo{6.1}}.
 
 Now assume \(n\geq1\). Then \(\seq{\su{j}}{j}{0}{n}\) belongs to \(\Dpqu{n}\) and, thus, \rprop{P1256} yields \(\Ssu{n}^\MP=\Ssu{n}^\rez\) and \(\SSu{n}^\MP=\SSu{n}^\rez\). Consequently, from these equations, from \rrema{R1456}, and from \rlemp{L1606}{L1606.a} we get
 \bspl{L1438.1}
  &\Ssu{n}^\rezalpha\Hu{n}^\splusalpha\SSu{n}^\rezalpha\\
  &=\Ssu{n}^\rez\Rpua{n}{\alpha}\lrk\lek \Rpua{n}{\alpha}\rek^\inv\Hu{n}\lek\Rqua{n}{\ko\alpha}\rek^\invad+\alpha\OIpu{n}\Hau{n-1}\OIqu{n}^\ad\rrk\lek\Rqua{n}{\ko\alpha}\rek^\ad\SSu{n}^\rez\\
  &=\Ssu{n}^\MP\lrk\Hu{n}+\alpha\Rpua{n}{\alpha}\OIpu{n}\Hau{n-1}\lek\Rqua{n}{\ko\alpha}\OIqu{n}\rek^\ad\rrk\SSu{n}^\MP.
 \esp
 Using \eqref{vDN} and \rrema{R1456}, we also see
 \bgl{L1438.2}
  \yuu{0}{n}^\rezalpha
  =\Ssu{n}^\rezalpha\vpu{n}
  =\Rqua{n}{\alpha}\Ssu{n}^\rez\vpu{n}
  =\Rqua{n}{\alpha}\yuu{0}{n}^\rez
 \eg
 and, similarly, \(\zuu{0}{n}^\rezalpha=\zuu{0}{n}^\rez[\Rpua{n}{\ko\alpha}]^\ad\). Thus, \eqref{L1438.A}, \eqref{L1438.1}, and \eqref{L1438.2} imply \eqref{L1438.C}. From \rlemp{L1606}{L1606.a} we know that \eqref{H[+]} holds true. \rremp{R0816:A-C}{R0816.b} shows
 \bgl{L1438.4}
  \Ssu{n}^\splusalpha\Hu{n}^\rezalpha\SSu{n}^\splusalpha
  =\lek\Rpua{n}{\alpha}\rek^\inv\Ssu{n}\Hu{n}^\rezalpha\SSu{n}\lek\Rqua{n}{\ko\alpha}\rek^\invad.
 \eg
 Because of \rremp{R0816:A-C}{R0816.b}, we have
 \bgl{L1438.5}
  \yuu{0}{n}^\splusalpha
  =\Ssu{n}^\splusalpha\vqu{n}
  =\lek\Rpua{n}{\alpha}\rek^\inv\Ssu{n}\vqu{n}
  =\lek\Rpua{n}{\alpha}\rek^\inv\yuu{0}{n}
 \eg
 and, similarly,
 \bgl{L1438.6}
  \zuu{0}{n}^\splusalpha
  =\zuu{0}{n}\lek\Rqua{n}{\ko\alpha}\rek^\invad.
 \eg
 It is readily checked that \([\Rpua{n}{\alpha}]\Thetauu{n}{2n}[\Rqua{n}{\ko\alpha}]^\ad=\Thetauu{n}{2n}\). Thus, multiplying equation~\eqref{L1438.B} from the left by \(\Rpua{n}{\alpha}\) and from the right by \(\lek\Rqua{n}{\ko\alpha}\rek^\ad\), and using \eqref{H[+]}, \eqref{L1438.4}, \eqref{L1438.5}, and \eqref{L1438.6}, we see that \eqref{L1438.D} holds true.
\eproof

\bpropl{T2047}
 Let \(\alpha\in\C\), let \(n\in \NO\), and let \(\seq{\su{j}}{j}{0}{2n}\in\Dtpqu{2n}\). Then:
 \benui
  \il{T2047.a}
  \bgl{T2047.A}
   \Hu{n}^\rezalpha
   =-\Ssu{n}^\rezalpha\lek\Hu{n}^\splusalpha-(\yuu{0}{n}^\splusalpha\vqu{n}^\ad+\vpu{n}\zuu{0}{n}^\splusalpha)\rek\SSu{n}^\rezalpha
  \eg
  and
  \bgl{T2047.B}
   \Hu{n}^\splusalpha
   =\Thetauu{n}{2n}-\Ssu{n}^\splusalpha\lek\Hu{n}^\rezalpha-(\yuu{0}{n}^\rezalpha\vqu{n}^\ad+\vpu{n}\zuu{0}{n}^\rezalpha)\rek\SSu{n}^\splusalpha.
  \eg
  \il{T2047.b} If \(n\geq1\), then
  \begin{multline}\label{T2047.C}
   \Hu{n}^\rezalpha
   =-\Ssu{n}^\MP\bigr[\Hu{n}+\alpha\Rpua{n}{\alpha}\OIpu{n}\Hau{n-1}\lek\Rqua{n}{\ko\alpha}\OIqu{n}\rek^\ad\\-\lrk\yuu{0}{n}\lek\Rqua{n}{\ko\alpha}\vqu{n}\rek^\ad+\Rpua{n}{\alpha}\vpu{n}\zuu{0}{n}\rrk\bigl]\SSu{n}^\MP
  \end{multline}
  and
  \begin{multline}\label{T2047.D}
   \Hu{n}+\alpha\Rpua{n}{\alpha}\OIpu{n}\Hau{n-1}\lek\Rqua{n}{\ko\alpha}\OIqu{n}\rek^\ad.\\
   =\Thetauu{n}{2n}-\Ssu{n}\lek\Hu{n}^\rezalpha-\lrk\Rqua{n}{\alpha}\yuu{0}{n}^\rez\vpu{n}^\ad+\vqu{n}\zuu{0}{n}^\rez\lek\Rpua{n}{\ko\alpha}\rek^\ad\rrk\rek\SSu{n}.
  \end{multline}
  \il{T2047.c} \(\rank(\Hu{n}^\rezalpha)=\rank[\Hu{n}^\splusalpha-(\yuu{0}{n}^\splusalpha\vqu{n}^\ad+\vpu{n}\zuu{0}{n}^\splusalpha)-\Thetauu{n}{2n}]\).%
  
  \il{T2047.d} If \(p=q\), then
  \[
   \det(\Hu{n}^\rezalpha)
   =\lek(\det\su{0})^\MP\rek^{2n+2}\det\lrk-\lek\Hu{n}^\splusalpha-(\yuu{0}{n}^\splusalpha\vqu{n}^\ad+\vpu{n}\zuu{0}{n}^\splusalpha)\rek\rrk.
  \]
 \eenui
\eprop
\bproof
 \eqref{T2047.a} Since \(\seq{\su{j}}{j}{0}{2n}\) belongs to \(\Dtpqu{2n}\), \rprop{L1438} shows that \eqref{L1438.A} and \eqref{L1438.B} hold true. Furthermore, \rremp{R0816:D-H}{R0816.e} yields \(\seq{\su{j}^\splusalpha}{j}{0}{2n}\in\Dtpqu{2n}\) and, consequently, \(\seq{\su{j}^\splusalpha}{j}{0}{n}\in\Dpqu{n}\). From \rprop{P1256} we see then
 \begin{align}\label{T2047.1}
  \Ssu{n}^\rezalpha&=(\Ssu{n}^\splusalpha)^\MP&
  &\text{and}&
  \SSu{n}^\rezalpha&=(\SSu{n}^\splusalpha)^\MP.
 \end{align}
 Thus, \rprop{P1256} yields
 \begin{align*}
  \Ssu{n}^\rezalpha\Ssu{n}^\splusalpha&=\Iu{n+1}\kp(\su{0}^\rezalpha\su{0}^\splusalpha)&
  &\text{and}&
  \SSu{n}^\splusalpha\SSu{n}^\rezalpha&=\Iu{n+1}\kp(\su{0}^\splusalpha\su{0}^\rezalpha).
 \end{align*}
 This implies
 \begin{align*}
  \Ssu{n}^\rezalpha\Ssu{n}^\splusalpha\vqu{n}&=\vqu{n}\su{0}^\rezalpha\su{0}^\splusalpha&
  &\text{and}&
  \vpu{n}^\ad\SSu{n}^\splusalpha\SSu{n}^\rezalpha&=\su{0}^\splusalpha\su{0}^\rezalpha\vpu{n}^\ad.
 \end{align*}
 Consequently, we conclude
 \bgl{T2047.2}
  \Ssu{n}^\rezalpha\yuu{0}{n}^\splusalpha
  =\Ssu{n}^\rezalpha\Ssu{n}^\splusalpha\vqu{n}
  =\vqu{n}\su{0}^\rezalpha\su{0}^\splusalpha
 \eg
 and
 \bgl{T2047.3}
  \zuu{0}{n}^\splusalpha\SSu{n}^\rezalpha
  =\vpu{n}^\ad\SSu{n}^\splusalpha\SSu{n}^\rezalpha
  =\su{0}^\splusalpha\su{0}^\rezalpha\vpu{n}^\ad.
 \eg
 From~\zitaa{MR3014199}{\cprop{4.9(b)}} we see that
 \begin{align}\label{T2047.4}%
  \su{0}^\rezalpha\su{0}^\splusalpha\zuu{0}{n}^\rezalpha
  &=\zuu{0}{n}^\rezalpha&
  &\text{and}&
  \yuu{0}{n}^\rezalpha\su{0}^\splusalpha\su{0}^\rezalpha
  &=\yuu{0}{n}^\rezalpha.
 \end{align}
 Using \eqref{T2047.2}, \eqref{T2047.3}, and \eqref{T2047.4}, we get then
 \begin{multline}\label{T2047.5}
  \Ssu{n}^\rezalpha\lek\yuu{0}{n}^\splusalpha\vqu{n}^\ad+\vpu{n}\zuu{0}{n}^\splusalpha\rek\SSu{n}^\rezalpha
  =\Ssu{n}^\rezalpha\yuu{0}{n}^\splusalpha\zuu{0}{n}^\rezalpha+\yuu{0}{n}^\rezalpha\zuu{0}{n}^\splusalpha\SSu{n}^\rezalpha\\
  =\vqu{n}\su{0}^\rezalpha\su{0}^\splusalpha\zuu{0}{n}^\rezalpha+\yuu{0}{n}^\rezalpha\su{0}^\splusalpha\su{0}^\rezalpha\vpu{n}^\ad
  =\vqu{n}\zuu{0}{n}^\rezalpha+\yuu{0}{n}^\rezalpha\vpu{n}^\ad.
 \end{multline}
 Applying \eqref{T2047.5} to equation~\eqref{L1438.A}, we then obtain \eqref{T2047.A}. Because of \eqref{T2047.1} and \rprop{P1256}, we have
 \begin{align*}
  \Ssu{n}^\splusalpha\Ssu{n}^\rezalpha&=\Iu{n+1}\kp(\su{0}^\splusalpha\su{0}^\rezalpha)&
  &\text{and}&
  \SSu{n}^\rezalpha\SSu{n}^\splusalpha&=\Iu{n+1}\kp(\su{0}^\rezalpha\su{0}^\splusalpha).
 \end{align*}
 This provides us
 \bgl{T2047.6}
  \Ssu{n}^\splusalpha\yuu{0}{n}^\rezalpha
  =\Ssu{n}^\splusalpha\Ssu{n}^\rezalpha\vpu{n}
  =\vpu{n}\su{0}^\splusalpha\su{0}^\rezalpha
 \eg
 and
 \bgl{T2047.7}
  \zuu{0}{n}^\rezalpha\SSu{n}^\splusalpha
  =\vqu{n}^\ad\SSu{n}^\rezalpha\SSu{n}^\splusalpha
  =\su{0}^\rezalpha\su{0}^\splusalpha\vqu{n}^\ad.
 \eg
 Since \(\seq{\su{j}^\splusalpha}{j}{0}{2n}\) belongs to \(\Dtpqu{2n}\), from \rrema{R1631} we see that the equations
 \begin{align}\label{T2047.8}
  \yuu{0}{n}^\splusalpha\su{0}^\rezalpha\su{0}^\splusalpha&=\yuu{0}{n}^\splusalpha&
  &\text{and}&
  \su{0}^\splusalpha\su{0}^\rezalpha\zuu{0}{n}^\splusalpha&=\zuu{0}{n}^\splusalpha
 \end{align}
 hold true. Using \eqref{T2047.6}, \eqref{T2047.7}, and \eqref{T2047.8}, we infer
 \begin{multline}\label{T2047.9}
  \Ssu{n}^\splusalpha(\yuu{0}{n}^\rezalpha\vpu{n}^\ad+\vqu{n}\zuu{0}{n}^\rezalpha)\SSu{n}^\splusalpha
  =\Ssu{n}^\splusalpha\yuu{0}{n}^\rezalpha\zuu{0}{n}^\splusalpha+\yuu{0}{n}^\splusalpha\zuu{0}{n}^\rezalpha\SSu{n}^\splusalpha\\
  =\vpu{n}\su{0}^\splusalpha\su{0}^\rezalpha\zuu{0}{n}^\splusalpha+\yuu{0}{n}^\splusalpha\su{0}^\rezalpha\su{0}^\splusalpha\vqu{n}^\ad
  =\vpu{n}\zuu{0}{n}^\splusalpha+\yuu{0}{n}^\splusalpha\vqu{n}^\ad.
 \end{multline}
 From \eqref{L1438.B} and \eqref{T2047.9} we then get \eqref{T2047.B}.
 
 \eqref{T2047.b} Now let \(n\geq1\). \rprop{L1438} shows that \eqref{L1438.C} and \eqref{L1438.D} are valid. Because of \(\seq{\su{j}}{j}{0}{2n}\in\Dtpqu{2n}\), we have \(\seq{\su{j}}{j}{0}{n}\in\Dpqu{n}\). From \eqref{S}, \eqref{vDN}, and \eqref{yz}, we get \(\Ssu{n}\vqu{n}=\yuu{0}{n}\). \rrema{R1456} yields \([\Rqua{n}{\ko\alpha}]^\ad\SSu{n}^\MP=\SSu{n}^\rezalpha=\SSu{n}^\rez[\Rpua{n}{\ko\alpha}]^\ad\). According to \rprop{P1256}, we have \(\Ssu{n}^\MP\Ssu{n}=\Iu{n+1}\kp(\su{0}^\MP\su{0})\). Using \eqref{vDN}, \eqref{Srez}, \eqref{S}, and \eqref{yzrez}, we obtain \(\vqu{n}^\ad\SSu{n}^\rez=\zuu{0}{n}^\rez\). Keeping in mind \eqref{vDN}, we conclude \([\Iu{n+1}\kp(\su{0}^\MP\su{0})]\vqu{n}=\vqu{n}\su{0}^\MP\su{0}\). In view of \eqref{yzrez} and \rdefi{D1430}, we have \(\su{0}^\MP\su{0}\zuu{0}{n}^\rez=\zuu{0}{n}^\rez\). Thus, we infer
 \bspl{T2047.12}
  \Ssu{n}^\MP\yuu{0}{n}\vqu{n}^\ad\lek\Rqua{n}{\ko\alpha}\rek^\ad\SSu{n}^\MP
  &=\Ssu{n}^\MP\Ssu{n}\vqu{n}\vqu{n}^\ad\SSu{n}^\rez\lek\Rpua{n}{\ko\alpha}\rek^\ad\\
 &=\lek\Iu{n+1}\kp(\su{0}^\MP\su{0})\rek\vqu{n}\zuu{0}{n}^\rez\lek\Rpua{n}{\ko\alpha}\rek^\ad\\
 &=\vqu{n}\su{0}^\MP\su{0}\zuu{0}{n}^\rez\lek\Rpua{n}{\ko\alpha}\rek^\ad
 =\vqu{n}\zuu{0}{n}^\rez\lek\Rpua{n}{\ko\alpha}\rek^\ad.
 \esp
 Similarly, we see that
 \bgl{T2047.13}
  \Ssu{n}^\MP\Rpua{n}{\alpha}\vpu{n}\zuu{0}{n}\SSu{n}^\MP
  =\Rqua{n}{\alpha}\yuu{0}{n}^\rez\vpu{n}^\ad
 \eg
 holds true. According to \rremp{R0816:A-C}{R0816.b}, we have \([\Rpua{n}{\alpha}]^\inv\Ssu{n}=\Ssu{n}[\Rqua{n}{\alpha}]^\inv\). Using \eqref{Srez}, \eqref{S}, \eqref{vDN}, and \eqref{yzrez}, we get \(\Ssu{n}^\rez\vpu{n}=\yuu{0}{n}^\rez\). In view of \eqref{vDN}, \eqref{S}, and \eqref{yz}, we obtain \(\vpu{n}^\ad\SSu{n}=\zuu{0}{n}\). \rprop{P1256} yields \(\Ssu{n}\Ssu{n}^\rez=\Ssu{n}\Ssu{n}^\MP=\Iu{n+1}\kp(\su{0}\su{0}^\MP)\). Keeping in mind \eqref{vDN}, we conclude \([\Iu{n+1}\kp(\su{0}\su{0}^\MP)]\vpu{n}=\vpu{n}\su{0}\su{0}^\MP\). Since the sequence \(\seq{\su{j}}{j}{0}{n}\) belongs to \(\Dpqu{n}\), we have, in view of \eqref{yz}, \rdefi{D1658}, and \rremp{R1631}{R1631.c}, furthermore \(\su{0}\su{0}^\MP\zuu{0}{n}=\zuu{0}{n}\). Thus, we infer
 \bsp
  \lek\Rpua{n}{\alpha}\rek^\inv\Ssu{n}\Rqua{n}{\alpha}\yuu{0}{n}^\rez\vpu{n}^\ad\SSu{n}
  &=\Ssu{n}\lek\Rqua{n}{\alpha}\rek^\inv\Rqua{n}{\alpha}\Ssu{n}^\rez\vpu{n}\zuu{0}{n}
  =\Ssu{n}\Ssu{n}^\rez\vpu{n}\zuu{0}{n}\\
  &=\lek\Iu{n+1}\kp(\su{0}\su{0}^\MP)\rek\vpu{n}\zuu{0}{n}
  =\vpu{n}\su{0}\su{0}^\MP\zuu{0}{n}
  =\vpu{n}\zuu{0}{n}.
 \esp
 Similarly, we get \(\Ssu{n}\vqu{n}\zuu{0}{n}^\rez\lek\Rpua{n}{\ko\alpha}\rek^\ad\SSu{n}\lek\Rqua{n}{\ko\alpha}\rek^\invad=\yuu{0}{n}\vqu{n}^\ad\). Hence,
 \begin{equation}\label{T2047.14}
  \Ssu{n}\lrk\Rqua{n}{\alpha}\yuu{0}{n}^\rez\vpu{n}^\ad+\vqu{n}\zuu{0}{n}^\rez\lek\Rpua{n}{\ko\alpha}\rek^\ad\rrk\SSu{n}
  =\yuu{0}{n}\vqu{n}^\ad\lek\Rqua{n}{\ko\alpha}\rek^\ad+\Rpua{n}{\alpha}\vpu{n}\zuu{0}{n}.
 \end{equation}
 Applying \eqref{T2047.12} and \eqref{T2047.13} to equation~\eqref{L1438.C}, we see that \eqref{T2047.C} holds true, whereas \eqref{L1438.D} and \eqref{T2047.14} yield \eqref{T2047.D}.
 
 \eqref{T2047.c} Because of \(\su{0}^\rezalpha(\su{2n}-\su{0}\su{0}^\MP\su{2n}\su{0}^\MP\su{0})\su{0}^\rezalpha=\su{0}^\MP(\su{2n}-\su{0}\su{0}^\MP\su{2n}\su{0}^\MP\su{0})\su{0}^\MP=\Oqp\), we have \(\Ssu{n}^\rezalpha\Thetauu{n}{2n}\SSu{n}^\rezalpha=\Ouu{(n+1)q}{(n+1)p}\). Thus, \eqref{T2047.A} shows that
 \[
  \rank\Hu{n}^\rezalpha
  \leq\rank\lek\Hu{n}^\splusalpha-(\yuu{0}{n}^\splusalpha\vqu{n}^\ad+\vpu{n}\zuu{0}{n}^\splusalpha)-\Thetauu{n}{2n}\rek
 \]
 holds true. On the other hand, the inequality
 \[
  \rank\lek\Hu{n}^\splusalpha-(\yuu{0}{n}^\splusalpha\vqu{n}^\ad+\vpu{n}\zuu{0}{n}^\splusalpha)-\Thetauu{n}{2n}\rek
  \leq\rank\Hu{n}^\rezalpha
 \]
 immediately follows from \eqref{L1438.B}.
 
 \eqref{T2047.d} \rPart{T2047.d} is a direct consequence of \eqref{T2047.A} and \(\su{0}^\rezalpha=\su{0}^\MP\).
\eproof

\btheol{T1640}
 Let \(\alpha\in\C\), let \(n\in\NO\), and let \(\seq{\su{j}}{j}{0}{2n+1}\in\Dtpqu{2n+1}\). Then:
 \benui
  \il{T1640.a} \(\Ku{n}^\rezalpha=-\Ssu{n}^\rezalpha\Ku{n}^\splusalpha\SSu{n}^\rezalpha\) and \(\Ku{n}^\rezalpha=-\Rqua{n}{\alpha}\Ssu{n}^\MP\Hau{n}\SSu{n}^\MP[\Rpua{n}{\ko\alpha}]^\ad\).
  \il{T1640.b} \(\Ku{n}^\splusalpha=\Thetauu{n}{2n+1}-\Ssu{n}^\splusalpha\Ku{n}^\rezalpha\SSu{n}^\splusalpha\).
  \il{T1640.c} \(\rank\Ku{n}^\rezalpha=\rank(\Ku{n}^\splusalpha-\Thetauu{n}{2n+1})\).
  \il{T1640.d} If \(p=q\), then \(\det(\Ku{n}^\rezalpha)=[(\det\su{0})^\MP]^{2n+2}\det(-\Ku{n}^\splusalpha)\).
 \eenui
\etheo
\bproof
 \rremp{R0816:D-H}{R0816.e} shows that \(\seq{\su{j}^\splusalpha}{j}{0}{2n+1}\) belongs to \(\Dtpqu{2n+1}\). Thus, the first equation in~\eqref{T1640.a} immediately follows from~\zitaa{MR3014199}{\ctheo{6.9(a)}}, whereas the second one is a consequence of this first equation, \rrema{R1456}, and \rlemp{L1606}{L1606.b}. Since \(\seq{\su{j}^\splusalpha}{j}{0}{2n+1}\) belongs to \(\Dtpqu{2n+1}\), from \rrema{R1631} we get
 \bsp
  \su{2n+1}^\splusalpha-\su{0}^\splusalpha(\su{0}^\splusalpha)^\MP\su{2n+1}^\splusalpha(\su{0}^\splusalpha)^\MP\su{0}^\splusalpha
  &=-\alpha\su{2n}+\su{2n+1}-\su{0}\su{0}^\MP(-\alpha\su{2n}+\su{2n+1})\su{0}^\MP\su{0}\\
  &=\su{2n+1}-\su{0}\su{0}^\MP\su{2n+1}\su{0}^\MP\su{0}.
 \esp
 Consequently, using~\zitaa{MR3014199}{\ctheo{6.9}}, we get~\eqref{T1640.b},~\eqref{T1640.c}, and~\eqref{T1640.d}.
\eproof

\bpropl{P1620}
 Let \(n\in\N\) and let \((s_j)_{j=0}^{2n}\in\Dtpqu{2n}\). Then the block Hankel matrix \(G_{n-1}^\rez\) admits the representation
 \begin{equation}\label{P1620.A}
  G_{n-1}^\rez
  =-\OIqu{n}^\ad\swu{n}^\MP  H_n\neu{n}^\MP  \OIpu{n}.
 \end{equation}
\eprop
\begin{proof}
 From~\zitaa{MR3014199}{\ctheo{6.1(a)}} we know that
 \begin{equation}\label{P1620.2}
  H_{n}^\rez
  =-\swu{n}^\MP  H_n \neu{n}^\MP  + y_{0,n}^\rez\vpu{n}^\ad +\vqu{n}z_{0,n}^\rez 
 \end{equation}
 holds true. Combining \(\OIqu{n}^\ad  H_n^\rez  \OIpu{n}=G_{n-1}^\rez\), \eqref{P1620.2}, and \(\vpu{n}^\ad \OIpu{n} =\Ouu{p}{np}\), we obtain then \eqref{P1620.A}.
\end{proof}
 
\btheol{T1701}
 Let \(\alpha\in\C\), let \(n\in \NO\), and let \(\seq{\su{j}}{j}{0}{2n+2}\in\Dtpqu{2n+2}\). Then:
 \benui
  \il{T1701.a} \(\Gu{n}^\rezalpha=-\Ssu{n}^\rezalpha\LLu{n+1}^\splusalpha\SSu{n}^\rezalpha\) and
  \begin{align}
   \Gu{n}^\rezalpha&=-\Rqua{n}{\alpha}\Ssu{n}^\MP(\Kau{n}-\yauu{0}{n}\su{0}^\MP\zauu{0}{n})\SSu{n}^\MP\lek\Rpua{n}{\ko\alpha}\rek^\ad,\nonumber\\
   \label{T1701.A}
   \Gu{n}^\rezalpha&=-\OIqu{n+1}^\ad\Ssu{n+1}^\MP\Hu{n+1}\SSu{n+1}^\MP\OIpu{n+1}-\alpha \Rqua{n}{\alpha}\Ssu{n}^\MP\Hau{n}\SSu{n}^\MP\lek\Rpua{n}{\ko\alpha}\rek^\ad.
  \end{align}
  \il{T1701.b} \(\Gu{n}^\splusalpha=\Thetauu{n}{2n+2}-\Ssu{n}^\splusalpha\LLu{n+1}^\rezalpha\SSu{n}^\splusalpha\), where \(\LLu{n+1}^\rezalpha\) is defined via \eqref{LL} using the \tsraute{} \(\seq{\su{j}^\reza{\alpha}}{j}{0}{2n+2}\) corresponding to the \tsplusalphata{\seq{\su{j}}{j}{0}{2n+2}}.
  \il{T1701.c} \(\rank\Gu{n}^\rezalpha=\rank(\LLu{n+1}^\splusalpha-\Thetauu{n}{2n+2})\) and \(\rank\Gu{n}^\rezalpha=\rank(\Kau{n}-\yauu{0}{n}\su{0}^\MP\zauu{0}{n}-\Thetauu{n}{2n+2})\).
  \il{T1701.d} If \(p=q\), then \(\det(\Gu{n}^\rezalpha)=[(\det\su{0})^\MP]^{2n+2}\det(-\LLu{n+1}^\splusalpha)\) and \(\det(\Gu{n}^\rezalpha)=[(\det\su{0})^\MP]^{2n+2}\det[-(\Kau{n}-\yauu{0}{n}\su{0}^\MP\zauu{0}{n})]\).
 \eenui
\etheo
\bproof
 The strategy of our proof is based on appropriate applications of~\zitaa{MR3014199}{\ctheo{6.13}}. Denote by \(\seq{t_j}{j}{0}{2n+2}\) the \tsplusalphata{\seq{\su{j}}{j}{0}{2n+2}}. \rremp{R0816:D-H}{R0816.e} shows that
 \bgl{T1701.1}
  \seq{t_j}{j}{0}{2n+2}
  \in\Dtpqu{2n+2}
 \eg
 holds true. Using \eqref{LL} and \rlemp{L1606}{L1606.d}, we obtain
 \bgl{T1701.21}
  \Guo{n}{t}-\yuuo{1}{n+1}{t}t_0^\MP\zuuo{1}{n+1}{t}
  =\LLuo{n+1}{t}
  =\LLu{n+1}^\splusalpha
  =\Kau{n}-\yauu{0}{n}\su{0}^\MP\zauu{0}{n}.
 \eg
 Because of \eqref{T1701.1} and \rdefi{D1459}, we have
 \begin{align}\label{T1701.22}
  \seq{\su{j}}{j}{0}{2n+1}&\in\Dpqu{2n+1}&
  &\text{and}&
  \seq{t_j}{j}{0}{2n+1}&\in\Dpqu{2n+1}.
 \end{align}
 Denote by \(\seq{r_j}{j}{0}{2n+2}\) the \tsrautea{\(\seq{t_j}{j}{0}{2n+2}$}. In view of \eqref{G}, \eqref{reza}, \eqref{Greza}, and \eqref{LL}, we get then
 \begin{align}\label{T1701.23}
  \Guo{n}{r}&=\Gu{n}^\rezalpha&
  &\text{and}&
  \Guo{n}{r}-\yuuo{1}{n+1}{r}r_0^\MP\zuuo{1}{n+1}{r}
  &=\LLuo{n+1}{r}
  =\LLu{n+1}^\rezalpha.
 \end{align}
 Because of \eqref{T1701.22} and \rrema{R1456}, we have furthermore
 \begin{align}\label{T1701.24}
  \Ssuo{n}{r}&=\Ssu{n}^\rezalpha=\lek\Rqu{n} (\alpha)\rek\swu{n}^\MP&
  &\text{and}&
  \SSuo{n}{r}&=\SSu{n}^\rezalpha=\neu{n}^\MP\lek\Rpua{n}{\ko\alpha}\rek^\ad.
 \end{align}
 According to \eqref{T1701.1}, we can apply~\zitaa{MR3014199}{\ctheo{6.13(a)}} to the sequence \(\seq{t_j}{j}{0}{2n+2}\) and obtain, by means of \eqref{T1701.23}, \eqref{T1701.24}, and \eqref{T1701.21}, the first two equations stated in~\eqref{T1701.a}.
 
 In view of \eqref{T1701.1}, the application of \rprop{P1620} to the sequence \(\seq{t_j}{j}{0}{2n+2}\) yields
 \bgl{T1701.2}
  \Guo{n}{r}
  =-\OIqu{n+1}^\ad(\Ssuo{n+1}{t})^\MP\Huo{n+1}{t}(\SSuo{n+1}{t})^\MP\OIpu{n+1}.
 \eg
 Because of \eqref{T1701.22}, \rprop{P1256} provides us
 \begin{align}\label{T1701.3}
  \Ssu{m}^\MP&=\Ssu{m}^\rez,&
  \SSu{m}^\MP&=\SSu{m}^\rez,&
  (\Ssuo{m}{t})^\MP&=\Ssuo{m}{r},&
  &\text{and}&
  (\SSuo{m}{t})^\MP&=\SSuo{m}{r}
 \end{align}
 for all \(m\in\mn{0}{2n+1}\). Keeping in mind \eqref{vDN}, \eqref{Srez}, \eqref{S}, and \eqref{yzrez}, this implies
 \begin{align}\label{T1701.4}
  \OIqu{n+1}^\ad\Ssu{n+1}^\MP
  &=\brow\yuu{1}{n+1}^\rez,\Ssu{n}^\MP\erow&
  &\text{and}&
  \SSu{n+1}^\MP\OIpu{n+1}
  &=
  \bMat
   \zuu{1}{n+1}^\rez\\
   \SSu{n}^\MP
  \eMat.
 \end{align}
 According to \eqref{T1701.3} and \eqref{T1701.22}, \rrema{R1456} shows that
 \begin{align}\label{T1701.5}
  (\Ssuo{n+1}{t})^\MP&=\Ssu{n+1}^\MP\lek\Rpua{n+1}{\alpha}\rek,&
  (\SSuo{n+1}{t})^\MP&=\lek\Rqua{n+1}{\ko\alpha}\rek^\ad\SSu{n+1}^\MP
 \end{align}
 and
 \begin{align}\label{T1701.6}
  \Ssu{n}^\MP\lek\Rpua{n}{\alpha}\rek&=\lek\Rqua{n}{\alpha}\rek\Ssu{n}^\MP,&
  \lek\Rqua{n}{\ko\alpha}\rek^\ad\SSu{n}^\MP&=\SSu{n}^\MP\lek\Rpua{n}{\ko\alpha}\rek^\ad
 \end{align}
 hold true. By virtue of \rlemp{L1606}{L1606.a}, the matrix \(\Huo{n+1}{t}\) can be represented via
 \[
  \Huo{n+1}{t}
  =\lek\Rpua{n+1}{\alpha}\rek^\inv\Hu{n+1}\lek\Rqua{n+1}{\ko\alpha}\rek^\invad+\alpha\OIpu{n+1}\Hau{n}\OIqu{n+1}^\ad.
 \]
 Using additionally \eqref{T1701.23}, \eqref{T1701.2}, and \eqref{T1701.5}, we conclude
 \bspl{T1701.7}
  &\Gu{n}^\rezalpha
  =\Guo{n}{r}
  =-\OIqu{n+1}^\ad\Ssu{n+1}^\MP\lek\Rpua{n+1}{\alpha}\rek\Huo{n+1}{t}\lek\Rqua{n+1}{\ko\alpha}\rek^\ad\SSu{n+1}^\MP\OIpu{n+1}\\
  &=-\OIqu{n+1}^\ad\Ssu{n+1}^\MP\Hu{n+1}\SSu{n+1}^\MP\OIpu{n+1}\\
  &\qquad-\alpha\OIqu{n+1}^\ad\Ssu{n+1}^\MP\lek\Rpua{n+1}{\alpha}\rek\OIpu{n+1}\Hau{n}\OIqu{n+1}^\ad\lek\Rqua{n+1}{\ko\alpha}\rek^\ad\SSu{n+1}^\MP\OIpu{n+1}.
 \esp
 From \eqref{Rb} and \eqref{vDN} we get \([\Rpua{n+1}{\alpha}]\OIpu{n+1}=\btmat\Ouu{p}{(n+1)p}\\\Rpua{n}{\alpha}\etmat\) and \( \OIqu{n+1}^\ad[\Rqua{n+1}{\ko\alpha}]^\ad=\brow\Ouu{(n+1)q}{q},[\Rqua{n}{\ko\alpha}]^\ad\erow\). Using additionally \eqref{T1701.4}, and \eqref{T1701.6}, we obtain
 \bgl{T1701.9}
  \OIqu{n+1}^\ad\Ssu{n+1}^\MP\lek\Rpua{n+1}{\alpha}\rek\OIpu{n+1}
  =\Ssu{n}^\MP\lek\Rpua{n}{\alpha}\rek
  =\lek\Rqua{n}{\alpha}\rek\Ssu{n}^\MP
 \eg
 and
 \bgl{T1701.16}
  \OIqu{n+1}^\ad\lek\Rqua{n+1}{\ko\alpha}\rek^\ad\SSu{n+1}^\MP\OIpu{n+1}
  =\lek\Rqua{n}{\ko\alpha}\rek^\ad\SSu{n}^\MP
  =\SSu{n}^\MP\lek\Rpua{n}{\ko\alpha}\rek^\ad.
 \eg
 Thus, equation~\eqref{T1701.A} follows from \eqref{T1701.7}, \eqref{T1701.9}, and \eqref{T1701.16}.
 
 In view of \eqref{[+]_0}, \rlemp{L1606}{L1606.c}, and the notations given in \rremp{R0816:A-C}{R0816.b}, we have
 \begin{align}\label{T1701.20}
  t_0&=s_0,&
  \Guo{n}{t}&=\Gu{n}^\splusalpha,&
  \Ssuo{n}{t}&=\Ssu{n}^\splusalpha,&
  &\text{and}&
  \SSuo{n}{t}&=\SSu{n}^\splusalpha.
 \end{align}
 According to \eqref{T1701.22}, \rpartss{R1631.c}{R1631.b} of \rrema{R1631} yield \(\su{0}\su{0}^\MP\su{2n+1}\su{0}^\MP\su{0}=\su{2n+1}\). By virtue of \rdefi{D1455}, we get hence
 \[
  \begin{split}
   t_{2n+2}-t_{0}t_{0}^\MP t_{2n+2}t_{0}^\MP t_{0}
   &=(-\alpha\su{2n+1}+\su{2n+2})-\su{0}\su{0}^\MP(-\alpha\su{2n+1}+\su{2n+2})\su{0}^\MP\su{0}\\
   &=\su{2n+2}-\su{0}\su{0}^\MP\su{2n+2}\su{0}^\MP\su{0},
  \end{split}  
 \]
 which, in view of \eqref{Xi}, implies
 \bgl{T1701.25}
  \Thetauuo{n}{2n+2}{t}
  =\Thetauu{n}{2n+2}.
 \eg
 Because of \eqref{T1701.1}, we can apply parts~(b)--(d) of~\zitaa{MR3014199}{\ctheo{6.13}} to the sequence \(\seq{t_j}{j}{0}{2n+2}\) and obtain, by means of \eqref{T1701.20}, \eqref{T1701.21}, \eqref{T1701.23}, and \eqref{T1701.25}, consequently~\eqref{T1701.b}--\eqref{T1701.d}.
\eproof

\section{The shortened negative \tsraute{} corres-\\ponding to the \hsplusalphat{} of a sequence from $\Dpqkappa$}\label{S1356}
This section should be compared with~\zitaa{MR3014199}{\cSect{7}}. The main goal of this section is to prepare the elementary step of the later \tSchur{}-type algorithm. As in~\zitaa{MR3014199}{\cSect{7}}, the main tool will be to use appropriately chosen \tsraute{s}. However, in~\zita{MR3014199} a two-step algorithm was applied. In this section, we develope a one-step algorithm where, in addition, the influence of the given real number \(\alpha\) has to be regarded.

Let \(\alpha\in\C\), let \(\kappa\in\N\cup\set{+\infty}\), and let \(\seq{\su{j}}{j}{0}{\kappa}\) be a sequence of complex \tpqa{matrices}. Then, using \rdefi{D1455} and \eqref{reza}, let the sequence \(\seq{\su{j}^\shortalpha}{j}{0}{\kappa-1}\)\index{$\seq{\su{j}^\shortalpha}{j}{0}{\kappa-1}$} be defined by
\bgl{(1)}
 \su{j}^\shortalpha
 \defg-\su{j+1}^\rezalpha
\eg
for all \(j\in\mn{0}{\kappa-1}\). Now we state some observations on the arithmetics of the transform just introduced.

\bremal{R1309}
 Let \(\alpha\in\C\), let \(\kappa\in\N\cup\set{+\infty}\), and let \(\seq{\su{j}}{j}{0}{\kappa}\) be a sequence of complex \tpqa{matrices}. In view of \eqref{(1)}, \eqref{reza}, \rdefi{D1455}, and \rrema{R1022}, one can easily see that, for all \(m\in\mn{0}{\kappa-1}\), the sequence \(\seq{\su{j}^\shortalpha}{j}{0}{m}\) depends only on the matrices \(\su{0},\su{1},\dotsc,\su{m+1}\).
\erema

\bremal{R1256}
 Let \(\alpha\in\C\), let \(\kappa\in\N\cup\set{+\infty}\), and let \(\seq{\su{j}}{j}{0}{\kappa}\) be a sequence of complex \tpqa{matrices}. In view of \eqref{(1)}, \rdefi{D1658},~\zitaa{MR3014197}{\cprop{5.10(a)}}, \eqref{reza}, \rdefi{D1430}, \eqref{[+]_0}, and \rremp{R1631}{R1631.a}, then \(\bigcup_{j=0}^{\kappa-1}\Bilda{\su{j}^\shortalpha}\subseteq\Bilda{\su{0}^\ad}\) and \(\Kerna{\su{0}^\ad}\subseteq\bigcap_{j=0}^{\kappa-1}\Kerna{\su{j}^\shortalpha}\).
\erema

\bremal{R1156}
 Let \(\alpha\in\C\), let \(\kappa\in\N\cup\set{+\infty}\), and let \(\seq{\su{j}}{j}{0}{\kappa}\) be a sequence of complex \tpqa{matrices}. In view of \eqref{(1)} and \rlemm{L1451}, then:
 \benui
  \il{R1156.a} If \(\gamma\in\C\), then \(\seq{(\gamma\su{j})^\shortalpha}{j}{0}{\kappa-1}=\seq{\gamma^\MP\su{j}^\shortalpha}{j}{0}{\kappa-1}\) and \(\Seq{(\gamma^j\su{j})^\shortalpha}{j}{0}{\kappa-1}=\seq{\gamma^{j+1}\su{j}^\shortalpha}{j}{0}{\kappa-1}\).
  \il{R1156.b} If \(m\in\N\) and \(L\in\Coo{m}{p}\) with \(\Bilda{L^\ad}=\Bilda{\su{0}}\), then \(\seq{(L\su{j})^\shortalpha}{j}{0}{\kappa-1}=\seq{\su{j}^\shorta{\alpha}L^\MP}{j}{0}{\kappa-1}\).
  \il{R1156.c} If \(n\in\N\) and \(R\in\Coo{q}{n}\) with \(\Bilda{\su{0}^\ad}=\Bilda{R}\), then \(\seq{(\su{j}R)^\shortalpha}{j}{0}{\kappa-1}=\seq{R^\MP\su{j}^\shorta{\alpha}}{j}{0}{\kappa-1}\).
  \il{R1156.d} If \(m,n\in\N\), \(L\in\Coo{m}{p}\) with \(\Bilda{L^\ad}=\Bilda{\su{0}}\) and \(R\in\Coo{q}{n}\) with \(\Bilda{\su{0}^\ad}=\Bilda{R}\), then \(\seq{(L\su{j}R)^\shortalpha}{j}{0}{\kappa-1}=\seq{R^\MP\su{j}^\shorta{\alpha}L^\MP}{j}{0}{\kappa-1}\).
  \il{R1156.e} If \(m,n\in\N\), \(U\in\Coo{m}{p}\) with \(U^\ad U=\Ip\) and \(V\in\Coo{q}{n}\) with \(VV^\ad=\Iq\), then \(\seq{(U\su{j}V)^\shorta{\alpha}}{j}{0}{\kappa-1}=\seq{V^\ad\su{j}^\shorta{\alpha}U^\ad}{j}{0}{\kappa-1}\).
  \il{R1156.f} If the sequence \(\seq{t_j}{j}{0}{\kappa}\) is given by \(t_j\defg\su{j}^\ad\) for all \(j\in\mn{0}{\kappa}\), then \((\su{j}^\shortalpha)^\ad=t_j^\shorta{\ko\alpha}\) for all \(j\in\mn{0}{\kappa-1}\).
 \eenui
\erema

\bremal{R1223}
 Let \(\alpha\in\C\), let \(\kappa\in\N\cup\set{+\infty}\), and let \(n\in\N\). For all \(m\in\mn{1}{n}\), let \(p_m,q_m\in\N\) and let \(\seq{\su{j}^{(m)}}{j}{0}{\kappa}\) be a sequence of complex \taaa{p_m}{q_m}{matrices}. In view of \eqref{(1)} and \rrema{R1452}, then \(\seq{\diag\matauuo{(\su{j}^{(m)})^\shortalpha}{m}{1}{n}}{j}{0}{\kappa-1}=\seq{(\diag\matauuo{\su{j}^{(m)}}{m}{1}{n})^\shortalpha}{j}{0}{\kappa-1}\).
\erema

\bremal{R1319}
 Let \(\alpha\in\C\), let \(\kappa\in\N\cup\set{+\infty}\), and let \((s_j)_{j=0}^\kappa\) be a sequence of complex \tpqa{matrices}. In view of \eqref{(1)}, \eqref{reza}, \rdefi{D1430}, \eqref{[+]_0}, and \rlemm{L1445}, then it is readily checked that
 \begin{align*}
  \su{j}^\shortalpha
  &=\su{0}^\MP\sum_{l=0}^j\su{j+1-l}^\splusalpha\su{l}^\rezalpha&
  &\text{and}&
  s_j^\shortalpha
  &=-\sum_{l=0}^{j+1} \alpha^{j+1-l} s_l^\rez&\text{for all }j&\in\mn{0}{\kappa-1}.
 \end{align*}
\erema

\bremal{R1347}
 Let \(\alpha\in\C\), let \(\kappa\in\N\cup\set{+\infty}\), and let \((s_j)_{j=0}^\kappa\) be a sequence of complex \tpqa{matrices}. Then, in view of \rrema{R1319}, \eqref{reza}, \rdefi{D1430}, \eqref{[+]_0}, and \eqref{(1)}, one can easily see that
 \begin{align*}
  \su{0}^\shortalpha
  &=\su{0}^\MP\su{1}^\splusalpha\su{0}^\MP&
  &\text{and}&
  s_j^\shortalpha
  &=\su{0}^\MP\su{j+1}^\splusalpha\su{0}^\MP-\su{0}^\MP\sum_{l=0}^{j-1}\su{j-l}^\splusalpha\su{l}^\shortalpha&\text{for all }j&\in\mn{1}{\kappa-1}.
 \end{align*}
\erema

\blemml{L1437}
 Let \(\alpha\in\C\), let \(\kappa\in\N\cup\set{+\infty}\), and let \(\seq{\su{j}}{j}{0}{\kappa}\) be a sequence of complex \tpqa{matrices}. For all \(j\in\mn{1}{\kappa}\), then
 \[
  \su{0}^\MP\su{j}^\sra{\alpha}\su{0}^\MP
  =\su{0}^\MP\sum_{l=0}^{j-1}\su{j-1-l}^\sra{\alpha}\su{l}^\shortalpha.
 \]
\elemm
\bproof
 In view of \eqref{[+]_0} and \rrema{R1347}, we have \(\su{0}^\MP\su{0}^\sra{\alpha}\su{l}^\shortalpha =\su{l}^\shortalpha \) for all \(l\in\mn{0}{j-1}\). Hence, \rrema{R1347} yields \(\su{0}^\MP\su{0}^\sra{\alpha}\su{0}^\shortalpha =\su{0}^\MP\su{1}^\splusalpha\su{0}^\MP\) and, in the case \(\kappa\geq2\), for all \(j\in\mn{2}{\kappa}\), furthermore
 \bsp
  \su{0}^\MP\sum_{l=0}^{j-1}\su{j-1-l}^\sra{\alpha}\su{l}^\shortalpha 
  &=\su{0}^\MP\sum_{l=0}^{j-2}\su{j-1-l}^\sra{\alpha}\su{l}^\shortalpha +\su{0}^\MP\su{0}^\sra{\alpha}\su{j-1}^\shortalpha \\
  &=\su{0}^\MP\su{j}^\sra{\alpha}\su{0}^\MP-\su{j-1}^\shortalpha +\su{j-1}^\shortalpha 
  =\su{0}^\MP\su{j}^\sra{\alpha}\su{0}^\MP.\qedhere
 \esp
\eproof

\blemml{L1601}
 Let \(\alpha\in\C\), let \(\kappa\in\N\cup\set{+\infty}\), and let \(\seq{\su{j}}{j}{0}{\kappa}\) and \(\seq{t_j}{j}{0}{\kappa}\) be sequences of complex \tpqa{matrices}. Then the following statements are equivalent:
 \baeqi{0}
  \il{L1601.i} \(\su{j}^\shortalpha=t_j^\shortalpha\) for all \(j\in\mn{0}{\kappa-1}\) and \(\su{0}=t_0\).
  \il{L1601.ii} \(\su{0}\su{0}^\MP\su{j}\su{0}^\MP\su{0}=t_0t_0^\MP t_jt_0^\MP t_0\) for all \(j\in\mn{0}{\kappa}\).
 \eaeqi
\elemm
\bproof
 According to \eqref{[+]_0}, \rdefi{D1430}, \eqref{reza}, and \rremp{R1631}{R1631.a}, statement~\rstat{L1601.i} is equivalent to:
 \baeqi{2}
  \il{L1601.iii} \(\su{j}^\shortalpha=t_j^\shortalpha\) for all \(j\in\mn{0}{\kappa-1}\) and \(\su{0}^\rezalpha=t_0^\rezalpha\).
 \eaeqi
 In view of \eqref{(1)} and \rlemm{L1454}, statement~\rstat{L1601.iii} is equivalent to~\rstat{L1601.ii}.
\eproof

\bremal{R1005}
 Let \(\alpha\in\C\), let \(\kappa\in\N\cup\set{+\infty}\), let \(\seq{\su{j}}{j}{0}{\kappa}\) be a sequence from \(\Cpq\), and let \(n\in\mn{0}{\kappa-1}\). In view of \eqref{S}, \eqref{(1)}, \rrema{R1456}, \eqref{vDN} and \eqref{T}, then
 \bsp
  \Ssuo{n}{s^\shortalpha}
  &=-\OIqu{n+1}^\ad\lrk\lek\Rqua{n+1}{\alpha}\rek\Ssu{n+1}^\rez-\lek\Iu{n+2}\kp(\su{0}^\MP)\rek\rrk\IOpu{n+1}\\
  &=\lek\Tuu{1}{n}^\ad\kp(\su{0}^\MP)\rek-\OIqu{n+1}^\ad\lek\Rqua{n+1}{\alpha}\rek\Ssu{n+1}^\rez\IOpu{n+1}
 \esp
 and
 \bsp
  \SSuo{n}{s^\shortalpha}
  &=-\IOqu{n+1}^\ad\lrk\SSu{n+1}^\rez\lek\Rpua{n+1}{\ko\alpha}\rek^\ad-\lek\Iu{n+2}\kp(\su{0}^\MP)\rek\rrk\OIpu{n+1}\\
  &=\lek\Tuu{1}{n}\kp(\su{0}^\MP)\rek-\IOqu{n+1}^\ad\SSu{n+1}^\rez\lek\Rpua{n+1}{\ko\alpha}\rek^\ad\OIpu{n+1}.
 \esp
\erema

Let \(\alpha\in\C\), let \(\kappa\in\N\cup\set{+\infty}\), and let \(\seq{\su{j}}{j}{0}{\kappa}\) be a sequence from \(\Cpq\). Then, let
\bgl{H(1)}
 \Hu{n}^\shortalpha
 \defg\matauuo{\su{j+k}^\shortalpha}{j,k}{0}{n}
\eg
\index{h@$\Hu{n}^\shortalpha$}for all \(n\in\NO\) with \(2n\leq\kappa-1\) and let
\bgl{K(1)}
 \Ku{n}^\shortalpha
 \defg\matauuo{\su{j+k+1}^\shortalpha}{j,k}{0}{n}
\eg
\index{k@$\Ku{n}^\shortalpha$}for all \(n\in\NO\) with \(2n+1\leq\kappa-1\). Now we derive some formulas for several block \tHankel{} matrices associated with the sequence introduced in \eqref{(1)}.
\blemml{L1513}
 Let \(\alpha\in\C\), let \(n\in\NO\), and let \(\seq{\su{j}}{j}{0}{2n+1}\in\Dtpqu{2n+1}\). Then:
 \benui
  \il{L1513.a} \(\Hu{n}^\shortalpha=-\Ku{n}^\rezalpha=[\Rqua{n}{\alpha}]\Ssu{n}^\MP\Hau{n}\SSu{n}^\MP[\Rpua{n}{\ko\alpha}]^\ad\).
  \il{L1513.b} \(\rank(\Hu{n}^\shortalpha)=\rank(\Hau{n}-\Thetauu{n}{2n+1})\).
  \il{L1513.c} If \(p=q\), then \(\det(\Hu{n}^\shortalpha)=[(\det\su{0})^\MP]^{2n+2}\det\Hau{n}\).
 \eenui
\elemm
\bproof
 \eqref{L1513.a} Use \eqref{H(1)}, \eqref{(1)}, \eqref{Kreza} and \rthmp{T1640}{T1640.a}.
 
 \eqref{L1513.b} Use the first equation in~\eqref{L1513.a}, \rthmp{T1640}{T1640.c} and \rlemp{L1606}{L1606.b}.
 
 \eqref{L1513.c} Use the first equation in~\eqref{L1513.a}, \rthmp{T1640}{T1640.d} and \rlemp{L1606}{L1606.b}.
\eproof

\blemml{L1104}
 Let \(\alpha\in\C\), let \(n\in\N\), and let \(\seq{\su{j}}{j}{0}{2n}\in\Dtpqu{2n}\). Then:
 \benui
  \il{L1104.a} \(-\alpha\Hu{n-1}^\shortalpha+\Ku{n-1}^\shortalpha=\OIqu{n}^\ad\Ssu{n}^\MP\Hu{n}\SSu{n}^\MP\OIpu{n}\) and \(-\alpha\Hu{n-1}^\shortalpha+\Ku{n-1}^\shortalpha=\Ssu{n-1}^\MP\LLu{n}\SSu{n-1}^\MP\).
  \il{L1104.b} \(\rank(-\alpha\Hu{n-1}^\shortalpha+\Ku{n-1}^\shortalpha)=\rank(\Hu{n}-\Thetauu{n}{2n})-\rank\su{0}\) and \(\rank(-\alpha\Hu{n-1}^\shortalpha+\Ku{n-1}^\shortalpha)=\rank(\LLu{n}-\Thetauu{n-1}{2n})\).
  \il{L1104.c} \(\det(-\alpha\Hu{n-1}^\shortalpha+\Ku{n-1}^\shortalpha)=[(\det\su{0})^\MP]^{2n+1}\det\Hu{n}\) and \(\det(-\alpha\Hu{n-1}^\shortalpha+\Ku{n-1}^\shortalpha)=[(\det\su{0})^\MP]^{2n}\det\LLu{n}\).
 \eenui
\elemm
\bproof
 \eqref{L1104.a} In view of \eqref{K(1)}, \eqref{(1)} and \eqref{Greza}, we have \(\Ku{n-1}^\shortalpha=-\Gu{n-1}^\rezalpha\). Taking additionally into account \rlemp{L1513}{L1513.a}, we obtain \(-\alpha\Hu{n-1}^\shortalpha+\Ku{n-1}^\shortalpha=\alpha\Ku{n-1}^\rezalpha-\Gu{n-1}^\rezalpha\). Hence, the first equation in~\eqref{L1104.a} is an immediate consequence of \rthmp{T1640}{T1640.a} and \rthmp{T1701}{T1701.a}. From~\zitaa{MR3014199}{\ctheo{6.1(a)}} we know that \(
  \Hu{n}^\rez+\Ssu{n}^\MP\Hu{n}\SSu{n}^\MP
  =\yuu{0}{n}^\rez\vpu{n}^\ad+\vqu{n}\zuu{0}{n}^\rez
 \). In view of \eqref{vDN}, we have \(\OIqu{n}^\ad(\yuu{0}{n}^\rez\vpu{n}^\ad+\vqu{n}\zuu{0}{n}^\rez)\OIpu{n}=0\) and, consequently,
 \bgl{L1104.0}
  \Gu{n-1}^\rez
  =\OIqu{n}^\ad\Hu{n}^\rez\OIpu{n}
  =-\OIqu{n}^\ad\Ssu{n}^\MP\Hu{n}\SSu{n}^\MP\OIpu{n}.
 \eg
 Furthermore,~\zitaa{MR3014199}{\ctheo{6.13(a)}} and \eqref{LL} yield \(\Gu{n-1}^\rez=-\Ssu{n-1}^\MP\LLu{n}\SSu{n-1}^\MP\). Thus, \eqref{L1104.0} and the first equation in~\eqref{L1104.a} yield the second one.
 
 \eqref{L1104.b} Obviously, \(\su{0}^\MP(\su{2n}-\su{0}\su{0}^\MP\su{2n}\su{0}^\MP\su{0})\su{0}^\MP=0\). Since \rprop{P1256} yields
 \begin{align}\label{L1104.1}
  \Ssu{k}^\MP&=\Ssu{k}^\rez&
  &\text{and}&
  \SSu{k}^\MP&=\SSu{k}^\rez&
  \text{for all }k&\in\mn{0}{n},
 \end{align}
 we then have \(\Ssu{n-1}^\MP\Thetauu{n-1}{2n}\SSu{n-1}^\MP=0\). In view of \rpart{L1104.a}, this implies
 \bgl{L1104.2}
  -\alpha\Hu{n-1}^\shortalpha+\Ku{n-1}^\shortalpha
  =\Ssu{n-1}^\MP(\LLu{n}-\Thetauu{n-1}{2n})\SSu{n-1}^\MP.
 \eg
 In particular,
 \bgl{L1104.3}
  \rank(-\alpha\Hu{n-1}^\shortalpha+\Ku{n-1}^\shortalpha)
  \leq\rank(\LLu{n}-\Thetauu{n-1}{2n}).
 \eg
 Multiplying equation~\eqref{L1104.2} from the left by \(\Ssu{n-1}\) and from the right by \(\SSu{n-1}\) and using \rprop{P1256}, we conclude
 \bgl{L1104.4}
  \Ssu{n-1}(-\alpha\Hu{n-1}^\shortalpha+\Ku{n-1}^\shortalpha)\SSu{n-1}
  =\lek\Iu{n}\kp(\su{0}\su{0}^\MP)\rek(\LLu{n}-\Thetauu{n-1}{2n})\lek\Iu{n}\kp(\su{0}^\MP\su{0})\rek.
 \eg
 Since \(\seq{\su{j}}{j}{0}{2n}\) belongs to \(\Dtpqu{2n}\), from \rpartss{R1631.c}{R1631.b} of \rrema{R1631} we get \(\su{0}\su{0}^\MP\su{j}=\su{j}\) and \(\su{j}\su{0}^\MP\su{0}=\su{j}\) hold for all \(j\in\mn{0}{2n-1}\). Because the \tqqa{block} in the right lower corner of \(\LLu{n}-\Thetauu{n-1}{2n}\) is exactly \(\su{0}\su{0}^\MP\su{2n}\su{0}^\MP\su{0}-\su{n}\su{0}^\MP\su{n}\), then from \eqref{LL} it follows that
 \[
  \lek\Iu{n}\kp(\su{0}\su{0}^\MP)\rek(\LLu{n}-\Thetauu{n-1}{2n})\lek\Iu{n}\kp(\su{0}^\MP\su{0})\rek
  =\LLu{n}-\Thetauu{n-1}{2n}.
 \]
 Combining this with \eqref{L1104.4}, we get
 \bgl{L1104.5}
  \rank(-\alpha\Hu{n-1}^\shortalpha+\Ku{n-1}^\shortalpha)
  \geq\rank(\LLu{n}-\Thetauu{n-1}{2n}).
 \eg
 From \eqref{L1104.3} and \eqref{L1104.5} we infer the second equation in~\eqref{L1104.b}. Taking into account that
 \[
  \Hu{n}-\Thetauu{n}{2n}
  =
  \bMat
   \su{0}&\zuu{1}{n}\\
   \yuu{1}{n}&\Gu{n-1}-\Thetauu{n-1}{2n}
  \eMat
 \]
 holds true and that \(\seq{\su{j}}{j}{0}{2n}\in\Dtpqu{2n}\) implies \(\Kerna{\su{0}}\subseteq\Kerna{\yuu{1}{n}}\) and \(\Bilda{\zuu{1}{n}}\subseteq\Bilda{\su{0}}\), we get from~\zitaa{MR1152328}{\clemm{1.1.7(a)}} that
 \bgl{L1104.6}
  \rank(\Hu{n}-\Thetauu{n}{2n})
  =\rank\su{0}+\rank(\LLu{n}-\Thetauu{n-1}{2n}).
 \eg
 Similarly, we see that
 \bgl{L1104.7}
  \det\Hu{n}
  =\det(\su{0})\det(\LLu{n})
 \eg
 is valid. The second equation in~\eqref{L1104.b} and \eqref{L1104.6} imply the first equation in~\eqref{L1104.b}.
 
 \eqref{L1104.c} Use the second equation in~\eqref{L1104.a} and \eqref{L1104.1} to get the second equation. The first equation follows in the case \(\det\su{0}\neq0\) from the second one and \eqref{L1104.7} and in the case \(\det\su{0}=0\) from the second one and \(n\geq1\).
\eproof

\blemml{L1611}
 Let \(\alpha\in\C\), let \(n\in\N\), and let \(\seq{\su{j}}{j}{0}{2n}\in\Dtpqu{2n}\). Then
 \bgl{L1611.A}
  \Ssu{n}^\MP\Hu{n}\SSu{n}^\MP
  =\diaga{\su{0}^\rez,-\alpha\Hu{n-1}^\shortalpha+\Ku{n-1}^\shortalpha}.
 \eg
\elemm
\bproof
 From~\zitaa{MR3014199}{\ctheo{6.1(a)}} we get \(\Hu{n}^\rez+\Ssu{n}^\MP\Hu{n}\SSu{n}^\MP=\yuu{0}{n}^\rez\vpu{n}^\ad+\vqu{n}\zuu{0}{n}^\rez\). In view of \eqref{Hrez} ,\eqref{Grez}, \eqref{yzrez}, and \eqref{vDN}, this implies \(
  \Ssu{n}^\MP\Hu{n}\SSu{n}^\MP
  =\diaga{\su{0}^\rez,-\Gu{n-1}^\rez}
 \). Taking additionally into account \rlemp{L1104}{L1104.a} and \eqref{vDN}, we obtain in particular
 \[
  -\Gu{n-1}^\rez
  =\OIqu{n}^\ad\Ssu{n}^\MP\Hu{n}\SSu{n}^\MP\OIpu{n}
  =-\alpha\Hu{n-1}^\shortalpha+\Ku{n-1}^\shortalpha.\qedhere
 \]
\eproof

The following proposition plays a key role in our further considerations. In its proof we will make essential use of \rlemmss{L1513}{L1104}.
\bpropl{P1542}
 Let \(\alpha\in\R\), let \(\kappa\in\N\cup\set{+\infty}\), and let \(\seq{\su{j}}{j}{0}{\kappa}\) be a sequence of complex \tqqa{matrices}. Then:
 \benui
  \il{P1542.a} If \(\seq{\su{j}}{j}{0}{\kappa}\in\Kggqualpha{\kappa}\), then \(\seq{\su{j}^\shortalpha}{j}{0}{\kappa-1}\in\Kggqualpha{\kappa-1}\).
  \il{P1542.b} If \(\seq{\su{j}}{j}{0}{\kappa}\in\Kggequalpha{\kappa}\), then \(\seq{\su{j}^\shortalpha}{j}{0}{\kappa-1}\in\Kggequalpha{\kappa-1}\).
  \il{P1542.c} If \(\seq{\su{j}}{j}{0}{\kappa}\in\Kgqualpha{\kappa}\), then \(\seq{\su{j}^\shortalpha}{j}{0}{\kappa-1}\in\Kgqualpha{\kappa-1}\).
  \il{P1542.d} If \(m\in\mn{0}{\kappa}\) and \(\seq{\su{j}}{j}{0}{\kappa}\in\Kggdoqualpha{m}{\kappa}\), then \(\seq{\su{j}^\shortalpha}{j}{0}{\kappa-1}\in\Kggdoqualpha{\max\set{0,m-1}}{\kappa-1}\).
  \il{P1542.e} If \(\seq{\su{j}}{j}{0}{\kappa}\in\Kggdqualpha{\kappa}\), then \(\seq{\su{j}^\shortalpha}{j}{0}{\kappa-1}\in\Kggdqualpha{\kappa-1}\).
 \eenui
\eprop
\bproof
 \eqref{P1542.a} Let \(\seq{\su{j}}{j}{0}{\kappa}\in\Kggqkappaalpha\). From \rpropp{P1442}{P1442.c} we then know that \(\seq{\su{j}}{j}{0}{\kappa}\in\Dtqqkappa\). Consequently, for all \(n\in\NO\) with \(2n+1\leq\kappa\), \rlemp{L1513}{L1513.a} yields
 \bgl{P1542.1}
  \Hu{n}^\shortalpha
  =\lek\Rqua{n}{\alpha}\rek\Ssu{n}^\MP\Hau{n}\SSu{n}^\MP\lek\Rqua{n}{\alpha}\rek^\ad.
 \eg
 \rlemp{L1104}{L1104.a} shows that, for all \(n\in\N\) with \(2n\leq\kappa\) the equation 
 \bgl{P1542.2}
  -\alpha\Hu{n-1}^\shortalpha+\Ku{n-1}^\shortalpha
  =\OIqu{n}^\ad\Ssu{n}^\MP\Hu{n}\SSu{n}^\MP\OIqu{n}
 \eg
 holds true. Since \(\seq{\su{j}}{j}{0}{\kappa}\) belongs to \(\Kggqkappaalpha\), we see from \rlemp{L1738}{L1738.a}, that \(\su{j}^\ad=\su{j}\) holds for all \(j\in\mn{0}{\kappa}\), which implies \(\Sbbu{n}^\MP=(\Sbu{n}^\MP)^\ad\) for all \(n\in\mn{0}{\kappa}\). Because of \(\seq{\su{j}}{j}{0}{\kappa}\in\Kggqkappaalpha\), we also have \(\Hu{n}\in\Cggo{(n+1)q}\) for all \(n\in\NO\) with \(2n\leq\kappa\) and \(\Hau{n}\in\Cggo{(n+1)q}\) for all \(n\in\NO\) with \(2n+1\leq\kappa\). Thus, we conclude that, for all \(n\in\NO\) with \(2n+1\leq\kappa\), the right-hand side of \eqref{P1542.1} is \tnnH{} and that, for all \(n\in\N\) with \(2n\leq\kappa\), the right-hand side of \eqref{P1542.2} is \tnnH{}. Consequently, \(\Hu{n}^\shortalpha\in\Cggo{(n+1)q}\) for all \(n\in\NO\) with \(2n+1\leq\kappa\) and \(-\alpha\Hu{n-1}^\shortalpha+\Ku{n-1}^\shortalpha\in\Cggo{nq}\) for all \(n\in\N\) with \(2n\leq\kappa\). Hence \(\seq{\su{j}^\shortalpha}{j}{0}{\kappa-1}\) belongs to \(\Kggqualpha{\kappa-1}\).
 
 \eqref{P1542.b} In the case \(\kappa=+\infty\), \rpart{P1542.b} is already proved in \rpart{P1542.a}. Now let \(m\in\NO\) and let \(\seq{\su{j}}{j}{0}{m}\in\Kggequalpha{m}\). Then there is an \(\su{m+1}\in\Cqq\) such that \(\seq{\su{j}}{j}{0}{m+1}\in\Kggqualpha{m+1}\). According to \rpart{P1542.a}, then we see that \(\seq{\su{j}^\shortalpha}{j}{0}{m}\) belongs to \(\Kggqualpha{m}\), which implies \(\seq{\su{j}^\shortalpha}{j}{0}{m-1}\in\Kggequalpha{m-1}\). Thus, in view of \rrema{R1309}, \rpart{P1542.b} is proved.
 
 \eqref{P1542.c} Let \(\seq{\su{j}}{j}{0}{\kappa} \in \Kgqkappaalpha\). \rPartss{P1442.d}{P1442.a} of \rprop{P1442} then yields \(\seq{\su{j}}{j}{0}{\kappa} \in \Dqqkappa\), which, in view of the \rdefiss{D1658}{D1459}, implies \(\seq{\su{j}}{j}{0}{m}\in\Dtqqu{m}\) for all \(m\in\mn{0}{\kappa}\). \rPartss{P1442.d}{P1442.c} of \rprop{P1442} and~\eqref{P1542.a} show that \(\seq{\su{j}^\shortalpha}{j}{0}{\kappa-1}\) belongs to \(\Kggqualpha{\kappa-1}\). In other words, we have:
 \bAeqi{0}
  \il{P1542.I} For all \(n\in \NO\) with \(2n+ 1 \leq \kappa\), the matrix \(\Hu{n}^\shortalpha\) is \tnnH{}.
 \eAeqi
 and
 \bAeqi{1}
  \il{P1542.II} For all \(n\in \N\) with \(2n\leq \kappa\), the matrix \(-\alpha\Hu{n-1}^\shortalpha+\Ku{n-1}^\shortalpha\) is \tnnH{}.
 \eAeqi
 Since \(\seq{\su{j}}{j}{0}{\kappa}\) belongs to \(\Kgqkappaalpha\), the following two statements hold true:
 \bAeqi{2}
  \il{P1542.III} For all \(n\in \NO\) with \(2n\leq \kappa\), the matrix \(\Hu{n}\) is \tpH{}.
  \il{P1542.IV} For all \(n\in \NO\) with \(2n+ 1 \leq \kappa\), the matrix \(\Hau{n}\) is \tpH{}.
 \eAeqi
 Using \rlemp{L1513}{L1513.c} and~\rstat{P1542.IV}, we see that:
 \bAeqi{4}
  \il{P1542.V} For all \(n\in \NO\) with \(2n+ 1 \leq \kappa\), the matrix \(\Hu{n}^\shortalpha\) is \tns{}.
 \eAeqi
 From~\rstat{P1542.I} and~\rstat{P1542.V} we see that:
 \bAeqi{5}
  \il{P1542.VI} For all \(n\in \NO\) with \(2n+ 1 \leq \kappa\), the matrix \(\Hu{n}^\shortalpha\) is \tpH{}.
 \eAeqi
 Using \rlemp{L1104}{L1104.c} and~\rstat{P1542.III}, we get that:
 \bAeqi{6}
  \il{P1542.VII} For all \(n\in \N\) with \(2n\leq \kappa\), the matrix \(-\alpha\Hu{n-1}^\shortalpha+\Ku{n-1}^\shortalpha\) is \tns{}.
 \eAeqi
 From~\rstat{P1542.II} and~\rstat{P1542.VII} we obtain that:
 \bAeqi{7}
  \il{P1542.VIII} For all \(n\in \N\) with \(2n\leq \kappa\), the matrix \(-\alpha\Hu{n-1}^\shortalpha+\Ku{n-1}^\shortalpha\) is \tpH{}.
 \eAeqi
 Because of~\rstat{P1542.VI} and~\rstat{P1542.VIII}, the sequence \(\seq{\su{j}^\shortalpha}{j}{0}{\kappa-1}\) belongs to \(\cKpd{q}{\kappa - 1}{\alpha}\).
 
 \eqref{P1542.d} Let \(m\in\mn{0}{\kappa}\) and let \(\seq{\su{j}}{j}{0}{\kappa}\in\Kggdoqkappaalpha{m}\). Because of \eqref{Kggcdm}, then \(\seq{\su{j}}{j}{0}{\kappa}\in\Kggqkappaalpha\), which, in view of~\eqref{P1542.a}, implies \(\seq{\su{j}^\shortalpha}{j}{0}{\kappa-1}\in\Kggqualpha{\kappa-1}\).
 Furthermore, \eqref{Kggcdm} yields \(\seq{\su{j}}{j}{0}{m}\in\Kggdqualpha{m}\). According to \rpartss{P1442.d}{P1442.a} of \rprop{P1442}, we have then \(\seq{\su{j}}{j}{0}{m}\in\Dqqu{m}\), which, in view of the \rdefiss{D1658}{D1459}, implies \(\seq{\su{j}}{j}{0}{l}\in\Dtqqu{l}\) for all \(l\in\mn{0}{m}\) and \(\su{0}\su{0}^\MP\su{j}=\su{j}\) and \(\su{j}\su{0}^\MP\su{0}=\su{j}\) for all \(j\in\mn{0}{m}\). In view of \eqref{Xi}, this implies
 \begin{align}\label{P1542.3}
  \Thetauu{n}{l}&=0&\text{for all }n&\in\NO\text{ and all }l\in\mn{0}{m}.
 \end{align}
 
 If \(m=0\), then \(\su{0}=\Oqq\). Thus, in view of \rrema{R1256}, the assertion holds true in the case \(m=0\).
 
 We now consider the case that \(m=2n+1\) with some \(n\in\NO\). Because of \eqref{Kggcd2n+1}, then \(\seq{\sau{j}}{j}{0}{2n}\in\Hggdqu{2n}\), which, in view of \eqref{Hggcd}, implies \(\Lau{n}=\Oqq\) and that the matrix \(\Hau{n}\) is \tnnH{}. If \(n\geq1\), then \(\Hau{n}\) admits the block representation
 \[
  \Hau{n}
  =
  \bMat
   \Hau{n - 1}   & \yauu{n}{2n-1} \\
   \zauu{n}{2n-1}   & \sau{2n}
  \eMat,
 \]
 which, in view of~\zitaa{MR1152328}{\clemmss{1.1.9}{1.1.7}}, implies \(\rank\Hau{n}=\rank\Hau{n-1}+\rank\Lau{n}\). Since \(\Hau{0}=\sau{0}=\Lau{0}\) holds true, we have thus \(\rank\Hau{n}=0\) in the case \(n=0\) and, in the case \(n\geq1\), furthermore \(\rank\Hau{n}=\rank\Hau{n-1}\). Because of \rlemp{L1513}{L1513.b}, \eqref{P1542.3}, and \rrema{R1309}, this implies \(\rank\Hu{n}^\shortalpha=0\) in the case \(n=0\) and, in the case \(n\geq1\), furthermore \(\rank\Hu{n}^\shortalpha=\rank\Hu{n-1}^\shortalpha\). Because of \rrema{R1440} and \eqref{Kgg2n}, we have \(\seq{\su{j}^\shortalpha}{j}{0}{2n}\in\Kggqualpha{2n}\subseteq\Hggqu{2n}\). Thus, similar to the considerations above, from~\zitaa{MR1152328}{\clemmss{1.1.9}{1.1.7}}, we conclude that \(\rank{\Luo{n}{s^\shortalpha}}=0\) and hence \(\Luo{n}{s^\shortalpha}=\Oqq\). According to \eqref{Hggcd}, we obtain then \(\seq{\su{j}^\shortalpha}{j}{0}{2n}\in\Hggdqu{2n}\) and, in view of \eqref{Kggcd2n}, consequently \(\seq{\su{j}^\shortalpha}{j}{0}{2n}\in\Kggdqualpha{2n}\).
 
 Now we consider the case that \(m=2n\) with some \(n\in\N\). Because of \eqref{Kggcd2n}, then \(\seq{\su{j}}{j}{0}{2n}\in\Hggdqu{2n}\), which, in view of \eqref{Hggcd}, implies \(\Lu{n}=\Oqq\) and that the block \tHankel{} matrices \(\Hu{n}\) and \(\Hu{n-1}\) are \tnnH{}. Obviously,
 \begin{align*}
  \Hu{n}
  &=
  \bMat
   \Hu{n-1}&\yuu{n}{2n-1}\\
   \zuu{n}{2n-1}&\su{2n}
  \eMat&
  &\text{and}&
  \Hu{n}
  =
  \bMat
   \su{0}&\zuu{1}{n}\\
   \yuu{1}{n}&\Gu{n-1}
  \eMat.
 \end{align*}
 In the case \(n\geq2\), the matrix \(\Hu{n-1}\) admits the block representation
 \[
  \Hu{n-1}
  =
  \bMat
   \su{0}&\zuu{1}{n-1}\\
   \yuu{1}{n-1}&\Gu{n-2}
  \eMat.
 \]
 Consequently, from~\zitaa{MR1152328}{\clemmss{1.1.9}{1.1.7}} we get \(\rank\Hu{n}=\rank\Hu{n-1}+\rank\Lu{n}=\rank\Hu{n-1}\) and \(\rank\Hu{n}=\rank\su{0}+\rank\LLu{n}\). Furthermore, \(\rank\Hu{n-1}=\rank\su{0}\) in the case \(n=1\) and, in the case \(n\geq2\), moreover \(\rank\Hu{n-1}=\rank\su{0}+\rank\LLu{n-1}\). In view of \eqref{P1542.3}, the last equations imply \(\rank(\LLu{n}-\Thetauu{n-1}{2n})=0\) in the case \(n=1\) and, in the case \(n\geq2\), furthermore \(\rank(\LLu{n}-\Thetauu{n-1}{2n})=\rank(\LLu{n-1}-\Thetauu{n-2}{2n})\). Using \rlemp{L1104}{L1104.b}, we then conclude  that \(\rank(-\alpha\Hu{n-1}^\shortalpha+\Ku{n-1}^\shortalpha)=0\) in the case \(n=1\) and, in the case \(n\geq2\), that \(\rank(-\alpha\Hu{n-1}^\shortalpha+\Ku{n-1}^\shortalpha)=\rank(-\alpha\Hu{n-2}^\shortalpha+\Ku{n-2}^\shortalpha)\). Because of \rrema{R1440}, we have \(\seq{\su{j}^\shortalpha}{j}{0}{2n-1}\in\Kggqualpha{2n-1}\), which, in view of \eqref{s_a} and \eqref{Kgg2n+1}, implies, that the sequence \(\seq{t_j}{j}{0}{2n-2}\) defined by \(t_j\defg-\alpha\su{j}^\shortalpha+\su{j+1}^\shortalpha\) for all \(j\in\mn{0}{2n-2}\) belongs to \(\Hggqu{2n-2}\). Then \(\rank\Luo{n-1}{t}=0\) follows immediately in the case of \(n=1\) and in the case \(n\geq2\) from~\zitaa{MR1152328}{\clemmss{1.1.9}{1.1.7}}. Thus, we have \(\Luo{n-1}{t}=\Oqq\). According to \eqref{Hggcd}, we obtain then \(\seq{t_j}{j}{0}{2n-2}\in\Hggdqu{2n-2}\) and, in view of \eqref{Kggcd2n+1}, consequently \(\seq{\su{j}^\shortalpha}{j}{0}{2n-1}\in\Kggdqualpha{2n-1}\). Thus, we have proved that \(\seq{\su{j}^\shortalpha}{j}{0}{\max\set{0,m-1}}\) belongs to \(\Kggqualpha{\max\set{0,m-1}}\), which, in view of \eqref{Kggcdm}, then implies \(\seq{\su{j}^\shortalpha}{j}{0}{\kappa-1}\in\Kggdoqualpha{\max\set{0,m-1}}{\kappa-1}\).
 
 \eqref{P1542.e} Let \(\seq{\su{j}}{j}{0}{\kappa}\in\Kggdqualpha{\kappa}\). We first consider the case \(\kappa\in\N\). In view of \eqref{Kggcd2n}, \eqref{Kggcd2n+1}, and \eqref{Kggcdm}, we have then \(\seq{\su{j}}{j}{0}{\kappa}\in\Kggdoqkappaalpha{\kappa}\), which, according to~\eqref{P1542.d}, implies \(\seq{\su{j}^\shortalpha}{j}{0}{\kappa-1}\in\Kggdoqualpha{\kappa-1}{\kappa-1}\). Because of \eqref{Kggcdm}, hence \(\seq{\su{j}^\shortalpha}{j}{0}{\kappa-1}\in\Kggdqualpha{\kappa-1}\). Now we consider the case \(\kappa=+\infty\).  In view of \eqref{G1023}, then \(\seq{\su{j}}{j}{0}{\kappa}\in\Kggdoqinfalpha{m}\) for some \(m\in\NO\), which, according to~\eqref{P1542.d}, implies \(\seq{\su{j}^\shortalpha}{j}{0}{\infty}\in\Kggdoqinfalpha{\max\set{0,m-1}}\). Because of \eqref{G1023}, hence \(\seq{\su{j}^\shortalpha}{j}{0}{\infty}\in\Kggdqinfalpha\).
\eproof

\section{The first \hascht{\alpha} of a sequence of complex matrices}\label{S1649}
This section plays a similar role as~\zitaa{MR3014199}{\cSect{8}}. Guided by our experiences from~\zita{MR3014199}, we will choose a convenient two-sided normalization of the sequence introduced in \eqref{(1)}. This construction gives us the tool to realize the elementary step in the \tSchur{}-type algorithm we are striving for. The following notion is one of the central objects in this paper. It describes the basic step of the \tSchur{}-type algorithm which will be developed in \rsect{S1401}.
\begin{defn}\label{D1059}
 Let \(\alpha\in\C\), let \(\kappa\in\N\cup\set{+\infty}\), and let \(\seq{\su{j}}{j}{0}{\kappa}\) be a sequence of complex \tpqa{matrices}. Let the sequence \(\seq{\su{j}^\shortalpha}{j}{0}{\kappa-1}\) be given by \eqref{(1)}. Then the sequence \(\seq{\su{j}^\sntaa{1}{\alpha}}{j}{0}{\kappa-1}\)\index{$\seq{\su{j}^\sntaa{1}{\alpha}}{j}{0}{\kappa-1}$} defined by
 \begin{align*}
  \su{j}^\sntaa{1}{\alpha}&\defg\su{0}\su{j}^\shortalpha\su{0}&\text{for all }j&\in\mn{0}{\kappa-1}
 \end{align*}
 is called the \emph{first \tascht{\alpha}} (or short \emph{\tlasnt{\alpha}}) \emph{of \(\seq{\su{j}}{j}{0}{\kappa}$}.
\end{defn}

 Our next considerations are aimed at studying the arithmetics of the first \tlasnt{\alpha}.
\bremal{R1013}
 Let \(\alpha\in\C\), let \(\kappa\in\N\cup\set{+\infty}\), and let \(\seq{\su{j}}{j}{0}{\kappa}\) be a sequence from \(\Cpq\). Denote by \(\seq{\su{j}^\sntaa{1}{\alpha}}{j}{0}{\kappa-1}\) the first \tlasnta{\alpha}{\seq{\su{j}}{j}{0}{\kappa}}.
 In view of \rdefi{D1059} and \rremass{R1256}{R1631}, then \(\su{j}^\shortalpha=\su{0}^\MP\su{j}^\sntaa{1}{\alpha}\su{0}^\MP\) for all \(j\in\mn{0}{\kappa-1}\).
\erema

\bremal{R1012}
 Let \(\alpha\in\C\), let \(\kappa\in\N\cup\set{+\infty}\), and let \(\seq{\su{j}}{j}{0}{\kappa}\) be a sequence of complex \tpqa{matrices}. Denote by \(\seq{\su{j}^\sntaa{1}{\alpha}}{j}{0}{\kappa-1}\) the first \tlasnt{\alpha} of \(\seq{\su{j}}{j}{0}{\kappa}\). In view of \rrema{R1309}, one can easily see that, for all \(m\in\mn{1}{\kappa}\), the sequence \(\seq{\su{j}^\sntaa{1}{\alpha}}{j}{0}{m-1}\) depends only on the matrices \(\su{0},\su{1},\dotsc,\su{m}\) and, hence, it coincides with the \tlasnt{\alpha} of \(\seq{\su{j}}{j}{0}{m}\).
\erema

\bremal{R1040}
 Let \(\alpha\in\C\), let \(\kappa\in\N\cup\set{+\infty}\), and let \(\seq{\su{j}}{j}{0}{\kappa}\) be a sequence of complex \tpqa{matrices}. In view of \rdefi{D1059}, for all \(j\in\mn{0}{\kappa-1}\), then \(\bigcup_{j=0}^{\kappa-1}\Bilda{\su{j}^\sntaa{1}{\alpha}}\subseteq\Bilda{\su{0}}\) and \(\Kerna{\su{0}}\subseteq\bigcap_{j=0}^{\kappa-1}\Kerna{\su{j}^\sntaa{1}{\alpha}}\).
\erema

\blemml{L1228}
 Let \(\alpha\in\C\), let \(\kappa\in\N\cup\set{+\infty}\), and let \(\seq{\su{j}}{j}{0}{\kappa}\) be a sequence of complex \tpqa{matrices}. Then:
 \benui
  \il{L1228.a} If \(\gamma\in\C\), then \(\seq{(\gamma\su{j})^\seinsalpha}{j}{0}{\kappa-1}=\seq{\gamma\su{j}^\seinsalpha}{j}{0}{\kappa-1}\) and \(\seq{(\gamma^j\su{j})^\seinsalpha}{j}{0}{\kappa-1}=\seq{\gamma^{j+1}\su{j}^\seinsalpha}{j}{0}{\kappa-1}\)
  \il{L1228.b} If \(m\in\N\) and \(L\in\Coo{m}{p}\) with \(\Bilda{L^\ad}=\Bilda{\su{0}}\), then \(\seq{(L\su{j})^\seinsalpha}{j}{0}{\kappa-1}=\seq{L(\su{j}^\seinsalpha)}{j}{0}{\kappa-1}\).
  \il{L1228.c} If \(n\in\N\) and \(R\in\Coo{q}{n}\) with \(\Bilda{\su{0}^\ad}=\Bilda{R}\), then \(\seq{(\su{j}R)^\seinsalpha}{j}{0}{\kappa-1}=\seq{\su{j}^\seinsalpha R}{j}{0}{\kappa-1}\).
  \il{L1228.d} If \(m,n\in\N\), \(L\in\Coo{m}{p}\) with \(\Bilda{L^\ad}=\Bilda{\su{0}}\) and \(R\in\Coo{q}{n}\) with \(\Bilda{\su{0}^\ad}=\Bilda{R}\), then \(\seq{(L\su{j}R)^\seinsalpha}{j}{0}{\kappa-1}=\seq{L\su{j}^\seinsalpha R}{j}{0}{\kappa-1}\).
  \il{L1228.e} If \(m,n\in\N\), \(U\in\Coo{m}{p}\) with \(U^\ad U=\Ip\), and \(V\in\Coo{q}{n}\) with \(VV^\ad=\Iq\), then \(\seq{(U\su{j}V)^\seinsalpha}{j}{0}{\kappa-1}=\seq{U\su{j}^\seinsalpha V}{j}{0}{\kappa-1}\).
  \il{L1228.f} If the sequence \(\seq{t_j}{j}{0}{\kappa}\) is given by \(t_j\defg\su{j}^\ad\) for all \(j\in\mn{0}{\kappa}\), then \((\su{j}^\seinsalpha)^\ad=t_j^\saa{1}{\ko\alpha}\) for all \(j\in\mn{0}{\kappa-1}\).
 \eenui
\elemm
\bproof
 \eqref{L1228.a} Use \rremp{R1156}{R1156.a}.
 
 \eqref{L1228.b} Because of \(\Bilda{L^\ad}=\Bilda{\su{0}}\) we have \(L^\MP L\su{0}=\su{0}\). Using \rremp{R1156}{R1156.b} for all \(j\in\mn{0}{\kappa-1}\), we get then
 \[
  (L\su{j})^\seinsalpha
  =(L\su{0})(L\su{j})^\shortalpha(L\su{0})
  =L\su{0}\su{j}^\shortalpha L^\MP L\su{0}
  =L\su{0}\su{j}^\shortalpha\su{0}
  =L\su{j}^\seinsalpha.
 \]
 
 \eqref{L1228.c} and~\eqref{L1228.d} can be proved like~\eqref{L1228.b} using \rpartss{R1156.c}{R1156.d} of \rrema{R1156}.
 
 \eqref{L1228.e} and~\eqref{L1228.f} can be proved using \rpartss{R1156.e}{R1156.f} of \rrema{R1156}.
\eproof

\bremal{R1229}
 Let \(\alpha\in\C\), let \(\kappa\in\N\cup\set{+\infty}\), and let \(n\in\N\). For all \(m\in\mn{1}{n}\), let \(p_m,q_m\in\N\) and let \(\seq{\su{j}^{(m)}}{j}{0}{\kappa}\) be a sequence of complex \taaa{p_m}{q_m}{matrices} with \tseinsalpha{} \(\seq{t_j^{(m)}}{j}{0}{\kappa-1}\). Then \rrema{R1223} shows that \(\seq{\diag\matauuo{t_j^{(m)}}{m}{1}{n}}{j}{0}{\kappa-1}\) is exactly the \tseinsalphaa{\seq{\diag\matauuo{\su{j}^{(m)}}{m}{1}{n}}{j}{0}{\kappa}}.
\erema

\bremal{R1645}
 Let \(\alpha\in\R\), let \(\kappa\in\N\cup \set{+\infty}\), and let \((s_j)_{j=0}^\kappa\) be a sequence of complex \tpqa{matrices}. Further, let \((s_j^\seinsalpha )_{j=0}^{\kappa-1}\) be the first \tlasnt{\alpha} of \((s_j)_{j=0}^\kappa\). In view of \rrema{R1319}, for all \(j\in\mn{0}{\kappa-1}\), then
 \begin{align*}
  \su{j}^\seinsalpha
  &=\su{0}\su{0}^\MP\sum_{l=0}^j\su{j+1-l}^\splusalpha\su{l}^\rezalpha\su{0}&
  &\text{and}&
  s_j^\seinsalpha
  &=-\sum_{l=0}^{j+1} \alpha^{j+1-l}\su{0}\su{l}^\rez\su{0}.
 \end{align*}
\erema

\begin{lem}\label{L1613}
 Let \(\alpha\in\C\), let \(\kappa\in\N\cup\set{+\infty}\), and let \((s_j)_{j=0}^{\kappa}\) be a sequence of complex \tpqa{matrices}. For all \(j\in\mn{1}{\kappa-1}\),
 \begin{align*}
  \su{0}^\seinsalpha
  &=\su{0}\su{0}^\MP\su{1}^\splusalpha\su{0}^\MP\su{0}&
  &\text{and}&
  s_j^\seinsalpha
  &=\su{0}\su{0}^\MP\lrk\su{j+1}^\splusalpha\su{0}^\MP\su{0}-\sum_{l=0}^{j-1}\su{j-l}^\splusalpha\su{0}^\MP\su{l}^\seinsalpha\rrk.
 \end{align*}
 
\end{lem}
\begin{proof}
 In view of \rremass{R1013}{R1347}, we get \(
  s_0^\seinsalpha 
  =\su{0}\su{0}^\shortalpha\su{0}
  =\su{0}\su{0}^\MP\su{1}^\splusalpha\su{0}^\MP\su{0}
 \) and, in the case \(\kappa\geq2\), for all \(j\in\mn{1}{\kappa-1}\), furthermore
 \bsp
  &s_j^\seinsalpha 
  =\su{0}\su{j}^\shortalpha\su{0}
  =\su{0}\lrk\su{0}^\MP\su{j+1}^\splusalpha\su{0}^\MP-\su{0}^\MP\sum_{l=0}^{j-1}\su{j-l}^\splusalpha\su{l}^\shortalpha\rrk\su{0}\\
  &=\su{0}\su{0}^\MP\lrk\su{j+1}^\splusalpha\su{0}^\MP\su{0}-\sum_{l=0}^{j-1}\su{j-l}^\splusalpha\su{l}^\shortalpha\su{0}\rrk
  =\su{0}\su{0}^\MP\lrk\su{j+1}^\splusalpha\su{0}^\MP\su{0}-\sum_{l=0}^{j-1}\su{j-l}^\splusalpha\su{0}^\MP\su{0}\su{l}^\shortalpha\su{0}\rrk\\
  &=\su{0}\su{0}^\MP\lrk\su{j+1}^\splusalpha\su{0}^\MP\su{0}-\sum_{l=0}^{j-1}\su{j-l}^\splusalpha\su{0}^\MP\su{l}^\seinsalpha\rrk.\qedhere
 \esp
\end{proof}

\begin{lem}\label{L1447}
 Let \(\alpha\in\C\), let \(\kappa\in\N\cup\set{+\infty}\), and let \((s_j)_{j=0}^{\kappa}\in\Dpqu{\kappa}\). Then \(\su{0}^\seinsalpha=\su{1}^\splusalpha\) and, for all \(j\in\mn{1}{\kappa-1}\), furthermore \(s_j^\seinsalpha=\su{j+1}^\splusalpha-\sum_{l=0}^{j-1}\su{j-l}^\splusalpha\su{0}^\MP\su{l}^\seinsalpha\).
\end{lem}
\begin{proof}
 \rremp{R0816:D-H}{R0816.d} yields \(\seq{\su{j}^\sra{\alpha}}{j}{0}{\kappa}\in\Dpqu{\kappa}\). Hence, in view of \rlemm{L1613}, \eqref{[+]_0}, \rdefi{D1658}, and \rpartss{R1631.c}{R1631.b} of \rrema{R1631}, we get
 \[
  s_0^\seinsalpha 
  =\su{0}\su{0}^\MP\su{1}^\splusalpha\su{0}^\MP\su{0}
  =\su{0}^\sra{\alpha}(\su{0}^\sra{\alpha})^\MP\su{1}^\sra{\alpha}(\su{0}^\sra{\alpha})^\MP\su{0}^\sra{\alpha}
  =s_1^\sra{\alpha}
 \]
 and, in the case \(\kappa\geq2\), for all \(j\in\mn{1}{\kappa-1}\), furthermore
 \bsp
  s_j^\seinsalpha
  &=\su{0}\su{0}^\MP\su{j+1}^\splusalpha\su{0}^\MP\su{0}-\sum_{l=0}^{j-1}\su{0}\su{0}^\MP\su{j-l}^\splusalpha\su{0}^\MP\su{l}^\seinsalpha\\
  &=\su{0}^\sra{\alpha}(\su{0}^\sra{\alpha})^\MP\su{j+1}^\splusalpha(\su{0}^\sra{\alpha})^\MP\su{0}^\sra{\alpha}-\sum_{l=0}^{j-1}\su{0}^\sra{\alpha}(\su{0}^\sra{\alpha})^\MP\su{j-l}^\splusalpha\su{0}^\MP\su{l}^\seinsalpha\\
  &=\su{j+1}^\sra{\alpha}-\sum_{l=0}^{j-1}\su{j-l}^\sra{\alpha}\su{0}^\MP\su{l}^\sntaa{1}{\alpha}.\qedhere
 \esp
\end{proof}

\bremal{R1315}
 Let \(\alpha\in\C\), let \(\kappa\in\N\cup\set{+\infty}\), and let \(\seq{\su{j}}{j}{0}{\kappa}\) be a sequence of complex \tpqa{matrices}. In view of \rlemm{L1437} and \rremass{R1347}{R1013}, for all \(j\in\mn{1}{\kappa}\), then
 \[
  \su{0}^\MP\su{j}^\sra{\alpha}\su{0}^\MP\su{0}
  =\su{0}^\MP\sum_{l=0}^{j-1}\su{j-1-l}^\sra{\alpha}\su{0}^\MP\su{l}^\seinsalpha.
 \]
\erema

\blemml{L1330}
 Let \(\alpha\in\C\), let \(\kappa\in\N\cup\set{+\infty}\), and let \(\seq{\su{j}}{j}{0}{\kappa}\in\Dpqu{\kappa}\). For all \(j\in\mn{1}{\kappa}\), then
 \[
  \su{j}^\sra{\alpha}
  =\sum_{l=0}^{j-1}\su{j-1-l}^\sra{\alpha}\su{0}^\MP\su{l}^\sntaa{1}{\alpha}.
 \]
\elemm
\bproof
 \rremp{R0816:D-H}{R0816.d} yields \(\seq{\su{j}^\sra{\alpha}}{j}{0}{\kappa}\in\Dpqu{\kappa}\). Hence, in view of \rdefi{D1658}, \rpartss{R1631.c}{R1631.b} of \rrema{R1631}, \eqref{[+]_0}, and \rrema{R1315}, for all \(j\in\mn{1}{\kappa}\), we get
 \bsp
  \su{j}^\sra{\alpha}
  &=\su{0}^\sra{\alpha}(\su{0}^\sra{\alpha})^\MP\su{j}^\sra{\alpha}(\su{0}^\sra{\alpha})^\MP\su{0}^\sra{\alpha}
  =\su{0}^\sra{\alpha}\su{0}^\MP\su{j}^\sra{\alpha}\su{0}^\MP\su{0}\\
  &=\su{0}^\sra{\alpha}\su{0}^\MP\sum_{l=0}^{j-1}\su{j-1-l}^\sra{\alpha}\su{0}^\MP\su{l}^\seinsalpha
  =\sum_{l=0}^{j-1}\su{0}^\sra{\alpha}(\su{0}^\sra{\alpha})^\MP\su{j-1-l}^\sra{\alpha}\su{0}^\MP\su{l}^\seinsalpha
  =\sum_{l=0}^{j-1}\su{j-1-l}^\sra{\alpha}\su{0}^\MP\su{l}^\seinsalpha.\qedhere
 \esp
\eproof

\blemml{L1411}
 Let \(\alpha\in\C\), let \(\kappa\in\N\cup\set{+\infty}\), and let \(\seq{\su{j}}{j}{0}{\kappa}\) and \(\seq{t_j}{j}{0}{\kappa}\) be sequences of complex \tpqa{matrices}. Then the following statements are equivalent:
 \baeqi{0}
  \il{L1411.i} \(\su{j}^\seinsalpha=t_j^\seinsalpha\) for all \(j\in\mn{0}{\kappa-1}\) and \(\su{0}=t_0\).
  \il{L1411.ii} \(\su{0}\su{0}^\MP\su{j}\su{0}^\MP\su{0}=t_0t_0^\MP t_jt_0^\MP t_0\) for all \(j\in\mn{0}{\kappa}\).
 \eaeqi
\elemm
\bproof
 According to \rrema{R1013}, statement~\rstat{L1411.i} is equivalent to
 \baeqi{2}
  \il{L1411.iii} \(\su{j}^\shortalpha=t_j^\shortalpha\) for all \(j\in\mn{0}{\kappa-1}\) and \(\su{0}=t_0\).
 \eaeqi
 which, in view of \rlemm{L1601}, is equivalent to~\rstat{L1411.ii}.
\eproof

For all \(m\in\NO\) and all \(w\in\C\), we easily see, in view of the block structure given in \eqref{Rb}, that
\bgl{sR-Rs}
 (\Iu{m+1}\kp\su{0})\lek\Rqua{m}{w}\rek
 =\lek\Rpua{m}{w}\rek(\Iu{m+1}\kp\su{0})
\eg
and
\bgl{R*s-sR*}
 \lek\Rpua{m}{w}\rek^\ad(\Iu{m+1}\kp\su{0})
 =(\Iu{m+1}\kp\su{0})\lek\Rqua{m}{w}\rek^\ad.
\eg

\bremal{R1020}
 Let \(\alpha\in\C\), let \(\kappa\in\N\cup\set{+\infty}\), let \(\seq{\su{j}}{j}{0}{\kappa}\) be a sequence from \(\Cpq\), and let \(n\in\mn{0}{\kappa-1}\). In view of \eqref{S}, \rrema{R1005}, \eqref{vDN}, \eqref{sR-Rs}, and \eqref{R*s-sR*}, then
 \bsp
  \Ssuo{n}{s^\seinsalpha}
  &=-\OIpu{n+1}^\ad\lrk\lek\Rpua{n+1}{\alpha}\rek(\Iu{n+2}\kp\su{0})\Ssu{n+1}^\rez(\Iu{n+2}\kp\su{0})-(\Iu{n+2}\kp\su{0})\rrk\IOqu{n+1}\\
  &=(\Tuu{1}{n}^\ad\kp\su{0})-\OIpu{n+1}^\ad\lek\Rpua{n+1}{\alpha}\rek(\Iu{n+2}\kp\su{0})\Ssu{n+1}^\rez(\Iu{n+2}\kp\su{0})\IOqu{n+1}
 \esp
 and
 \bsp
  \SSuo{n}{s^\seinsalpha}
  &=-\IOpu{n+1}^\ad\lrk(\Iu{n+2}\kp\su{0})\SSu{n+1}^\rez(\Iu{n+2}\kp\su{0})\lek\Rqua{n+1}{\ko\alpha}\rek^\ad-(\Iu{n+2}\kp\su{0})\rrk\OIqu{n+1}\\
  &=(\Tuu{1}{n}\kp\su{0})-\IOpu{n+1}^\ad(\Iu{n+2}\kp\su{0})\SSu{n+1}^\rez(\Iu{n+2}\kp\su{0})\lek\Rqua{n+1}{\ko\alpha}\rek^\ad\OIqu{n+1}.
 \esp
\erema

\bnotal{N1039}
 For all \(n\in\NO\), let \(\nudqu{n}\)\index{l@$\nudqu{n}$} (resp.\ \(\nodqu{n}\)\index{u@$\nodqu{n}$}) be the set of all complex \taaa{(n+1)q}{(n+1)q}{matrices} \(A\) which satisfy the following condition:
 \begin{itemize}
  \item If \(A=[A_{jk}]_{j,k=0}^n\) is the \tqqa{block} representation of \(A\), then \(A_{jj}=\Iq \) for all \(j\in\mn{0}{n}\) and \(A_{jk}=\Oqq \) for all \(j,k\in\mn{0}{n}\) with \(j<k\) (resp.\ \(k<j\)).
 \end{itemize}
\enota

\bremal{R1247}
 Let \(\kappa\in\NO\cup\set{+\infty}\) and let \(\seq{\su{j}}{j}{0}{\kappa}\) be a sequence of complex \tpqa{matrices}. In view of \eqref{Srez}, \rdefi{D1430}, and \rnota{N1039}, then, for all \(m\in\mn{0}{\kappa}\), the matrix
 \bgl{Dl}
  \Dluo{m}{s}
  \defg(\Iu{m+1}\kp\su{0})\Ssu{m}^\rez+\lek\Iu{m+1}\kp(\Ip-\su{0}\su{0}^\MP)\rek
 \eg
 \index{d@$\Dluo{m}{s}\), \(\Dlu{m}$}belongs to \(\nuduu{p}{m}\), the matrix
 \bgl{Dr}
  \Druo{m}{s}
  \defg\SSu{m}^\rez(\Iu{m+1}\kp\su{0})+\lek\Iu{m+1}\kp(\Iq-\su{0}^\MP\su{0})\rek
 \eg
 \index{d@$\Druo{m}{s}\), \(\Dru{m}$}belongs to \(\noduu{q}{m}\), and, according to \rrema{R1745}, in particular, \(\det\Dluo{m}{s}=1\) and \(\det\Druo{m}{s}=1\).
\erema
For short, we will also write \(\Dlu{m}\) and \(\Dru{m}\) for \(\Dluo{m}{s}\) and \(\Druo{m}{s}\), respectively.

\blemml{L1320}
 Let \(\kappa\in\NO\cup\set{+\infty}\) and let \(\seq{\su{j}}{j}{0}{\kappa}\) be a sequence of \tH{} complex \tqqa{matrices}. For all \(m\in\mn{0}{\kappa}\), then \(\Dlu{m}^\ad=\Dru{m}\).
\elemm
\bproof
 Let \(m\in\mn{0}{\kappa}\). From~\zitaa{MR3014197}{\ccoro{5.17}} we obtain then \((\su{j}^\rez)^\ad=\su{j}^\rez\) for all \(j\in\mn{0}{m}\) which, in view of \eqref{Srez} and \eqref{S}, implies \((\Ssu{m}^\rez)^\ad=\SSu{m}^\rez\). Using \rremp{R1631}{R1631.a}, we get furthermore \((\su{0}\su{0}^\MP)^\ad=(\su{0}^\MP)^\ad \su{0}^\ad=(\su{0}^\ad)^\MP \su{0}^\ad=\su{0}^\MP \su{0}\). Taking additionally into account \eqref{Dl} and \eqref{Dr}, we finally obtain the asserted equation.
\eproof

\bremal{R1339}
 Let \(\alpha\in\C\), let \(\kappa\in\NO\cup\set{+\infty}\), and let \(\seq{\su{j}}{j}{0}{\kappa}\) be a sequence of complex \tpqa{matrices} with \tsplusalphat{} \(\seq{t_j}{j}{0}{\kappa}\). In view of \eqref{Dl}, \eqref{Dr}, \eqref{[+]_0}, \rrema{R1456}, \eqref{sR-Rs}, and \eqref{R*s-sR*}, then one can easily see that, for all \(m\in\mn{0}{\kappa}\), the matrices
 \begin{align}\label{D[+]}
  \Dlu{m}^\splusalpha&\defg\Dluo{m}{t}&
  &\text{and}&
  \Dru{m}^\splusalpha&\defg\Druo{m}{t}
 \end{align}
 \index{d@$\Dlu{m}^\splusalpha$}\index{d@$\Dru{m}^\splusalpha$}can be represented via
 \begin{align*}
  \Dlu{m}^\splusalpha
  &=\lek\Rpua{m}{\alpha}\rek(\Iu{m+1}\kp\su{0})\Ssu{m}^\rez+\lek\Iu{m+1}\kp(\Ip-\su{0}\su{0}^\MP)\rek\\
 \shortintertext{and}
  \Dru{m}^\splusalpha
  &=\SSu{m}^\rez(\Iu{m+1}\kp\su{0})\lek\Rqua{m}{\ko\alpha}\rek^\ad+\lek\Iu{m+1}\kp(\Iq-\su{0}^\MP\su{0})\rek.
 \end{align*}
\erema

We now turn our attention to several block \tHankel{} matrices built from the \tseinsalpha. Let \(\alpha\in\C\), let \(\kappa\in\N\cup\set{+\infty}\), and let \(\seq{\su{j}}{j}{0}{\kappa}\) be a sequence of complex \tpqa{matrices}. Then, let
\bgl{H[1]}
 \Hu{n}^\seinsalpha
 \defg\matauuo{\su{j+k}^\seinsalpha}{j,k}{0}{n}
\eg
\index{h@$\Hu{n}^\seinsalpha$}for all \(n\in\NO\) with \(2n\leq\kappa-1\) and let
\bgl{K[1]}
 \Ku{n}^\seinsalpha
 \defg\matauuo{\su{j+k+1}^\seinsalpha}{j,k}{0}{n}
\eg
\index{k@$\Ku{n}^\seinsalpha$}for all \(n\in\NO\) with \(2n+1\leq\kappa-1\).

\bremal{R0937}
 Let \(\alpha\in\C\), let \(n\in\NO\), and let \(\seq{\su{j}}{j}{0}{2n+1}\) be a sequence of complex \tpqa{matrices}. Then because of \eqref{H[1]}, \rdefi{D1059}, and \eqref{H(1)}, we have \(\Hu{n}^\seinsalpha=(\Iu{n+1}\kp\su{0})\Hu{n}^\shortalpha(\Iu{n+1}\kp\su{0})\).
\erema

The following result contains interesting links between the block \tHankel{} matrices \(\Hu{n}^\seinsalpha\) and \(\Hau{n}\) introduced via \eqref{H[1]} and \eqref{Ha}, respectively.
\blemml{L1548}
 Let \(\alpha\in\C\), let \(n\in\NO\), and let \(\seq{\su{j}}{j}{0}{2n+1}\in\Dtpqu{2n+1}\). Let the matrices \(\Thetauu{n}{2n+1}\), \(\Dlu{n}^\splusalpha\), and \(\Dru{n}^\splusalpha\) be given via \eqref{Xi}, \eqref{D[+]}, \eqref{Dl}, and \eqref{Dr}, respectively. Then:
 \benui
  \il{L1548.a}
  \begin{align}
  \label{L1548.A} 
   \Hu{n}^\seinsalpha
   &=\lek\Rpua{n}{\alpha}\rek(\Iu{n+1}\kp\su{0})\Ssu{n}^\MP\Hau{n}\SSu{n}^\MP(\Iu{n+1}\kp\su{0})\lek\Rqua{n}{\ko\alpha}\rek^\ad\\
  \shortintertext{and}
  \label{L1548.B}
   \Hu{n}^\seinsalpha
   &=\Dlu{n}^\splusalpha(\Hau{n}-\Thetauu{n}{2n+1})\Dru{n}^\splusalpha.
  \end{align}
  \il{L1548.b} \(\rank(\Hu{n}^\seinsalpha)=\rank(\Hau{n}-\Thetauu{n}{2n+1})\).
  \il{L1548.c} If \(p=q\), then \(\det(\Hu{n}^\seinsalpha)=(\det\su{0})(\det\su{0})^\MP\det\Hau{n}\) and \(\det(\Hu{n}^\seinsalpha)=\det(\Hau{n}-\Thetauu{n}{2n+1})\).
 \eenui
\elemm
\bproof
 \eqref{L1548.a} Using \rrema{R0937}, \rlemp{L1513}{L1513.a}, \eqref{sR-Rs}, and \eqref{R*s-sR*}, then \eqref{L1548.A} follows. Since
 \(\seq{\su{j}}{j}{0}{2n+1}\) belongs to \(\Dtpqu{2n+1}\), \rdefi{D1459} yields \(\seq{\su{j}}{j}{0}{2n}\in\Dpqu{2n}\). Hence, according to \rprop{P1256}, we have
 \begin{align}\label{L1548.2}
  \Ssu{k}^\rez&=\MoPen{ \Ssu{k} }&
  &\text{and}&
  \Rezip{ \SSu{k} }&=\MoPen{ \SSu{k} }
 \end{align}
 for all \(k\in\mn{0}{2n}\). From \rdefi{D1658} and \rrema{R1631} we also see that
 \begin{align}\label{L1548.1}
  \su{0}\su{0}^\MP\su{j}&=\su{j}&
  &\text{and}&
  \su{j}\su{0}^\MP\su{0}&=\su{j}
 \end{align}
 hold true for all \(j\in\mn{0}{2n}\). In view of \eqref{Ha}, \eqref{Hs}, \eqref{s_a}, and \eqref{Xi}, the equations in \eqref{L1548.1} immediately imply
 \[
  \lek\Iu{n+1}\kp(\Ip-\su{0}\su{0}^\MP)\rek(\Hau{n}-\Thetauu{n}{2n+1})
  =\Ouu{(n+1)p}{(n+1)q}
 \]
 and
 \[
  (\Hau{n}-\Thetauu{n}{2n+1})\lek\Iu{n+1}\kp(\Iq-\su{0}^\MP\su{0})\rek
  =\Ouu{(n+1)p}{(n+1)q}.
 \]
 Consequently, \rrema{R1339} yields
 \begin{multline}\label{L1548.4}
  \Dlu{n}^\splusalpha(\Hau{n}-\Thetauu{n}{2n+1})\Dru{n}^\splusalpha\\
  =\lek\Rpua{n}{\alpha}\rek(\Iu{n+1}\kp\su{0})\Rezip{\Ssu{n}}(\Hau{n}-\Thetauu{n}{2n+1})\Rezip{\SSu{n}}(\Iu{n+1}\kp\su{0})\lek  \Rqua{n}{\ko\alpha}\rek^\ad.
 \end{multline}
 Taking \eqref{Srez}, \eqref{S}, \rdefi{D1430}, and \eqref{Xi} into account, we get \(\Rezip{\Ssu{n}}\Thetauu{n}{2n+1}\Rezip{\SSu{n}}=\Ouu{(n+1)q}{(n+1)p}\), whereas \eqref{L1548.2} for \(k=n\) shows that \(\Rezip{\Ssu{n}}=\MoPen{\Ssu{n}}\) and \(\Rezip{\SSu{n}}=\MoPen{\SSu{n}}\) hold true. Thus, the expression on the right-hand side of \eqref{L1548.4} coincides with the expression on the right-hand side of \eqref{L1548.A}. Hence, \eqref{L1548.B} is proved as well.

 \eqref{L1548.b} This follows from \eqref{L1548.B}, \eqref{D[+]}, and \rrema{R1247}.

 \eqref{L1548.c} From \eqref{Rb} we know that \(\det[\Rpua{n}{\alpha}]=1=\det([\Rqua{n}{\ko\alpha}]^\ad)\). Equations~\eqref{L1548.2} and \eqref{Srez} and \rdefi{D1430} yield
 \[
  \det\lek(\Iu{n+1}\kp\su{0})\MoPen{\Ssu{n}}\rek
  =\det(\Iu{n+1}\kp\su{0})\det(\Rezip{\Ssu{n}})
  =(\det\su{0})^{n+1}\lek\det(\su{0}^\MP )\rek^{n+1}
 \]
 and, similarly
 \[
  \det\lek\MoPen{\SSu{n}}(\Iu{n+1}\kp\su{0})\rek
  =\lek\det(\su{0}^\MP)\rek^{n+1}(\det\su{0})^{n+1}.
 \]
 Since
 \[
  (\det\su{0})^{2n + 2}\lek\det(\su{0}^\MP)\rek^{2n+2}
  =(\det\su{0})\MoPen{(\det\su{0})}
 \]
 is true, equation~\eqref{L1548.A} then implies the first equation stated in~\eqref{L1548.c}. \rrema{R1247} and \eqref{D[+]} show that \(\det\Dlu{n}^\splusalpha=1=\det\Dru{n}^\splusalpha\). Thus, \eqref{L1548.B} shows that the second equation stated in \rpart{L1548.c} is also true.
\eproof

\blemml{L1545}
 Let \(\alpha\in\C\), let \(n\in\N\), and let \(\seq{\su{j}}{j}{0}{2n}\in\Dtpqu{2n}\). Let the matrices \(\LLu{n}\), \(\Thetauu{n-1}{2n}\), \(\Dlu{n-1}\), and \(\Dru{n-1}\) be given via \eqref{LL}, \eqref{Xi}, \eqref{Dl}, and \eqref{Dr}, respectively. Then:
 \benui
  \il{L1545.a}
  \begin{align}
  \label{L1545.A}
   -\alpha\Hu{n-1}^\seinsalpha+\Ku{n-1}^\seinsalpha
   &=(\Iu{n}\kp\su{0})\Ssu{n-1}^\MP\LLu{n}\SSu{n-1}^\MP(\Iu{n}\kp\su{0})\\
  \shortintertext{and}
  \label{L1545.B}
   -\alpha\Hu{n-1}^\seinsalpha+\Ku{n-1}^\seinsalpha
   &=\Dlu{n-1}(\LLu{n}-\Thetauu{n-1}{2n})\Dru{n-1}.
  \end{align}
  \il{L1545.b} \(\rank(-\alpha\Hu{n-1}^\seinsalpha+\Ku{n-1}^\seinsalpha)=\rank(\LLu{n}-\Thetauu{n-1}{2n})\).
  \il{L1545.c} If \(p=q\), then
  \begin{align}
  \label{L1545.C}
   \det(-\alpha\Hu{n-1}^\seinsalpha+\Ku{n-1}^\seinsalpha)
   &=(\det\su{0})(\det\su{0})^\MP\det\LLu{n}\\
  \shortintertext{and}
  \label{L1545.D}
   \det(-\alpha\Hu{n-1}^\seinsalpha+\Ku{n-1}^\seinsalpha)
   &=\det(\LLu{n}-\Thetauu{n-1}{2n}).
  \end{align}
 \eenui
\elemm
\bproof
 \eqref{L1545.a} Equation~\eqref{L1545.A} immediately follows from \eqref{H[1]}, \eqref{K[1]}, \rdefi{D1059}, \eqref{H(1)}, \eqref{K(1)}, and \rlemp{L1104}{L1104.a}. Since \(n \in \N\) and \(\rFolge{s}{j}{0}{2n} \in \cDwtdd{p}{q}{2n}\) hold, we have \eqref{L1548.1} for all \(j\in\mn{0}{2n-1}\) and \(\rFolge{s}{j}{0}{n}\in\cDnwtdd{p}{q}{n}\). Thus \rprop{P1256} yields \eqref{L1548.2} for all \(k\in\mn{0}{n}\). From \eqref{Srez}, \eqref{S}, \rdefi{D1430}, and \eqref{Xi}  we also see that \(\Ssu{n-1}^\rez\Thetauu{n-1}{2n}\SSu{n}^\rez=0\) is true. Using \eqref{LL}, \eqref{G}, \eqref{yz}, \eqref{Xi}, and \eqref{L1548.1}, we get \([\Iu{n}\kp(\Ip-\su{0}\su{0}^\MP)](\LLu{n}-\Thetauu{n-1}{2n})=\Ouu{np}{nq}\) and \((\LLu{n}-\Thetauu{n-1}{2n})[\Iu{n}\kp(\Iq-\su{0}^\MP\su{0})]=\Ouu{np}{nq}\). Thus, because of \eqref{L1548.2}, from \eqref{Dl}, \eqref{Dr}, and \eqref{L1545.A} then \eqref{L1545.B} follows.

 \eqref{L1545.b}  This follows from \eqref{L1545.B} and \rrema{R1247}.

 \eqref{L1545.c} Because of \eqref{L1548.2}, \eqref{Srez}, and \rdefi{D1430}, we have
 \bsp
  \det(\Iu{n}\kp\su{0})\det(\MoPen{\Ssu{n-1}})\det(\MoPen{\SSu{n-1}})\det(\Iu{n}\kp\su{0})
  &=(\det\su{0})^n\lek\det(\su{0}^\MP)\rek^n\lek\det(\su{0}^\MP)\rek^n(\det\su{0})^n\\
  &=(\det\su{0})\MoPen{(\det\su{0})}.
 \esp
 Using this, from \eqref{L1545.A}, we get \eqref{L1545.C}. Equation~\eqref{L1545.D} follows from \eqref{L1545.B} and \rrema{R1247}.
\eproof

 In~\zitas{MR1807884,MR2038751} Chen and Hu treat the truncated matricial Stieltjes moment problem (\(\alpha=0\) in our setting). They introduce a transformation \(\Gamma_k\) which maps a sequence of length \(k+1\) of complex square matrices to a sequence of length \(k\) of complex square matrices. This transformation is defined via~\zitaa{MR1807884}{formula~(9)} (see also~\zitaa{MR2038751}{formula~(3.2)}). A closer look on~\zitaa{MR1807884}{formula~(9)} shows that this identity is essentially \eqref{L1545.A} for \(\alpha=0\) and \(p=q\). Therefore, the transformation \(\Gamma_m\) coincides for sequences from \(\Dtqqu{m}\) with the first \tascht{0}ation. To describe the respective solution sets, Chen and Hu reduce the length of the given sequence of prescribed moments in each step by 2, using the transformation \(\Gamma_{m-1}\Gamma_m\) (see~\zitaa{MR1807884}{formula~(12)} and~\zitaa{MR2038751}{formula~(3.7)}).

 Now we state the main result of this section.
\btheol{P1546}
 Let \(\alpha\in\R\), let \(\kappa\in\N\cup\set{+\infty}\), and let \(\seq{\su{j}}{j}{0}{\kappa}\) be a sequence of complex \tqqa{matrices}. Then:
 \benui
  \il{P1546.a} If \(\seq{\su{j}}{j}{0}{\kappa}\in\Kggqualpha{\kappa}\), then \(\seq{\su{j}^\seinsalpha}{j}{0}{\kappa-1}\in\Kggqualpha{\kappa-1}\).
  \il{P1546.b} If \(\seq{\su{j}}{j}{0}{\kappa}\in\Kggequalpha{\kappa}\), then \(\seq{\su{j}^\seinsalpha}{j}{0}{\kappa-1}\in\Kggequalpha{\kappa-1}\).
  \il{P1546.c} If \(\seq{\su{j}}{j}{0}{\kappa}\in\Kgqualpha{\kappa}\), then \(\seq{\su{j}^\seinsalpha}{j}{0}{\kappa-1}\in\Kgqualpha{\kappa-1}\).
  \il{P1546.d} If \(m\in\mn{0}{\kappa}\) and \(\seq{\su{j}}{j}{0}{\kappa}\in\Kggdoqualpha{m}{\kappa}\), then \(\seq{\su{j}^\seinsalpha}{j}{0}{\kappa-1}\in\Kggdoqualpha{\max\set{0,m-1}}{\kappa-1}\).
  \il{P1546.e} If \(\seq{\su{j}}{j}{0}{\kappa}\in\Kggdqualpha{\kappa}\), then \(\seq{\su{j}^\seinsalpha}{j}{0}{\kappa-1}\in\Kggdqualpha{\kappa-1}\).
 \eenui
\etheo
\bproof
 \eqref{P1546.a} Let \(\seq{\su{j}}{j}{0}{\kappa}\in\Kggqkappaalpha\). In view of \rlemp{L1738}{L1738.a}, then \(\su{0}^\ad=\su{0}\). Thus, \rdefi{D1059}, \rpropp{P1542}{P1542.a}, and \rrema{R1303} yield \(\seq{\su{j}^\seinsalpha}{j}{0}{\kappa-1}\in\Kggqualpha{\kappa-1}\).

 \eqref{P1546.b} Let \(\seq{\su{j}}{j}{0}{\kappa}\in\Kggeqkappaalpha\). Then \(\seq{\su{j}}{j}{0}{\kappa}\in\Kggqkappaalpha\). \rlemp{L1738}{L1738.a} shows that \(\su{0}^\ad=\su{0}\). Thus, \rdefi{D1059}, \rpropp{P1542}{P1542.b}, and \rrema{R1303} yield \(\seq{\su{j}^\seinsalpha}{j}{0}{\kappa-1}\in\Kggequalpha{\kappa-1}\).

 \eqref{P1546.c} Let \(\seq{\su{j}}{j}{0}{\kappa}\in\Kgqkappaalpha\). Then, in view of \eqref{H}, the matrix \(\su{0}\) is \tpH{} and hence \tH{} and \tns{}. Thus, \rdefi{D1059}, \rpropp{P1542}{P1542.c}, and~\zitaa{MR3014201}{\crema{2.15}} yield \(\seq{\su{j}^\seinsalpha}{j}{0}{\kappa-1}\in\Kgqualpha{\kappa-1}\).

 \eqref{P1546.d} Let \(m\in\mn{0}{\kappa}\) and let \(\seq{\su{j}}{j}{0}{\kappa}\in\Kggdoqualpha{m}{\kappa}\). Because of \eqref{Kggcdm}, then \(\seq{\su{j}}{j}{0}{\kappa}\in\Kggqkappaalpha\). Hence, in view of~\eqref{P1546.a}, we have \(\seq{\su{j}^\seinsalpha}{j}{0}{\kappa-1}\in\Kggqualpha{\kappa-1}\) and, because of \rlemp{L1738}{L1738.a}, furthermore \(\su{0}^\ad=\su{0}\). The application of \rpropp{P1542}{P1542.d} yields \(\seq{\su{j}^\shortalpha}{j}{0}{\kappa-1}\in\Kggdoqualpha{\max\set{0,m-1}}{\kappa-1}\) which, in view of \eqref{Kggcdm}, implies \(\seq{\su{j}^\shortalpha}{j}{0}{\max\set{0,m-1}}\in\Kggdqualpha{\max\set{0,m-1}}\). Using \rrema{R1256} and \(\su{0}^\ad=\su{0}\) we obtain \(\Kerna{\su{0}}\subseteq\bigcap_{j=0}^{\max\set{0,m-1}-1}\Kerna{\su{j}^\shortalpha}\) in the case \(\max\set{0,m-1}\geq1\). Thus, \rdefi{D1059} and~\zitaa{MR3014201}{\clemm{5.7}} yield \(\seq{\su{j}^\seinsalpha}{j}{0}{\max\set{0,m-1}}\in\Kggdqualpha{\max\set{0,m-1}}\). Hence, in view of \eqref{Kggcdm}, we obtain \(\seq{\su{j}^\seinsalpha}{j}{0}{\kappa-1}\in\Kggdoqualpha{\max\set{0,m-1}}{\kappa-1}\).
 
 \eqref{P1546.e} Let \(\seq{\su{j}}{j}{0}{\kappa}\in\Kggdqualpha{\kappa}\). Then \(\seq{\su{j}}{j}{0}{\kappa}\in\Kggqkappaalpha\). Hence, in view of \rlemp{L1738}{L1738.a}, we have \(\su{0}^\ad=\su{0}\). The application of \rpropp{P1542}{P1542.e} yields \(\seq{\su{j}^\shortalpha}{j}{0}{\kappa-1}\in\Kggdqualpha{\kappa-1}\). Using \rrema{R1256} and \(\su{0}^\ad=\su{0}\), we obtain \(\Kerna{\su{0}}\subseteq\bigcap_{j=0}^{\kappa-2}\Kerna{\su{j}^\shortalpha}\) in the case \(\kappa-1\geq1\). Thus, \rdefi{D1059} and~\zitaa{MR3014201}{\clemm{5.7}} yield \(\seq{\su{j}^\seinsalpha}{j}{0}{\kappa-1}\in\Kggdqualpha{\kappa-1}\).
\eproof

\section{A \tSchur{}-type algorithm for sequences of complex matrices}\label{S1401}
The \tascht{\alpha} introduced in \rsect{S1649} generates in a natural way a corresponding algorithm for (finite or infinite) sequences of complex \tpqa{matrices}. The investigation of this algorithm is the central point of this section. First we are going to extend \rdefi{D1059}.
\bdefil{D1632}
 Let \(\alpha\in\C\), let \(\kappa\in\NO\cup\set{+\infty}\), and let \(\seq{\su{j}}{j}{0}{\kappa}\) be a sequence of complex \tpqa{matrices}. The sequence \(\seq{\su{j}^\sntaa{0}{\alpha}}{j}{0}{\kappa}\) given by \(\su{j}^\sntaa{0}{\alpha}\defg\su{j}\) for all \(j\in\mn{0}{\kappa}\) is called the \emph{\(0\)\nobreakdash-th \tlasnta{\alpha}{\seq{\su{j}}{j}{0}{\kappa}}}. In the case \(\kappa\geq1\), for all \(k\in\mn{1}{\kappa}\), the \emph{\(k\)\nobreakdash-th \tlasnt{\alpha} \(\seq{\su{j}^\sntaa{k}{\alpha}}{j}{0}{\kappa-k}\)\index{$\seq{\su{j}^\sntaa{k}{\alpha}}{j}{0}{\kappa-k}$} of \(\seq{\su{j}}{j}{0}{\kappa}$} is recursively defined by \(\su{j}^\sntaa{k}{\alpha}\defg t_j^\seinsalpha\) for all \(j\in\mn{0}{\kappa-k}\), where \(\seq{t_j}{j}{0}{\kappa-(k-1)}\) denotes the \((k-1)\)\nobreakdash-th \tlasnta{\alpha}{\seq{\su{j}}{j}{0}{\kappa}}.
\edefi

\bremal{R1501}
 Let \(\alpha\in\C\), let \(\kappa\in\NO\cup\set{+\infty}\), let \(\seq{\su{j}}{j}{0}{\kappa}\) be a sequence of complex \tpqa{matrices}, and let \(k\in\mn{0}{\kappa}\). Denote by \(\seq{\su{j}^\sntaa{k}{\alpha}}{j}{0}{\kappa-k}\) the \(k\)\nobreakdash-th \tlasnta{\alpha}{\seq{\su{j}}{j}{0}{\kappa}}. From \rdefi{D1632} and \rrema{R1012} one can see then that, for all \(m\in\mn{k}{\kappa}\), the sequence \(\seq{\su{j}^\sntaa{k}{\alpha}}{j}{0}{m-k}\) depends only on the matrices \(\su{0},\su{1},\dotsc,\su{m}\) and, hence, it coincides with the \(k\)\nobreakdash-th \tlasnta{\alpha}{\seq{\su{j}}{j}{0}{m}}. In particular, the sequence \(\seq{\su{j}^\saalpha{k}}{j}{0}{0}\) depends only on the matrices \(\su{0},\su{1},\dotsc,\su{k}\) and is the \tsaalphaa{k}{\seq{\su{j}}{j}{0}{k}}.
\erema

\bremal{R1510}
 Let \(\alpha\in\C\), let \(\kappa\in\NO\cup\set{+\infty}\), and let \(\seq{\su{j}}{j}{0}{\kappa}\) be a sequence of complex \tpqa{matrices}. From \rdefi{D1632} then it is immediately obvious that, for all \(k\in\mn{0}{\kappa}\) and all \(l\in\mn{0}{\kappa-k}\), the \((k+l)\)\nobreakdash-th \tlasnt{\alpha} \(\seq{\su{j}^\sntaa{k+l}{\alpha}}{j}{0}{\kappa-(k+l)}\) of \(\seq{\su{j}}{j}{0}{\kappa}\) is exactly the \(l\)\nobreakdash-th \tlasnt{\alpha} of the \(k\)\nobreakdash-th \tlasnt{\alpha} \(\seq{\su{j}^\sntaa{k}{\alpha}}{j}{0}{\kappa-k}\) of \(\seq{\su{j}}{j}{0}{\kappa}\).
\erema

\blemml{L1627}
 Let \(\alpha\in\C\), let \(\kappa\in\N\cup\set{+\infty}\), let \(\seq{\su{j}}{j}{0}{\kappa}\) be a sequence of complex \tpqa{matrices}, and let \(k\in\mn{0}{\kappa-1}\). Then \(\bigcup_{l=1}^{\kappa-k}[\bigcup_{j=0}^{\kappa-(k+l)}\Bilda{\su{j}^\saalpha{k+l}}]\subseteq\Bilda{\su{0}^\saalpha{k}}\) and \(\Kerna{\su{0}^\saalpha{k}}\subseteq\bigcap_{l=1}^{\kappa-k}[\bigcap_{j=0}^{\kappa-(k+l)}\Kerna{\su{j}^\saalpha{k+l}}]\).
\elemm
\bproof
 Let \(l\in\mn{1}{\kappa-k}\). Then \(\kappa-(k+l-1)\geq1\). Hence, having in mind \rdefi{D1632}, the application of \rrema{R1040} to the sequence \(\seq{\su{j}^\sntaalpha{k+l-1}}{j}{0}{\kappa-(k+l-1)}\) yields
 \begin{align*}
  \bigcup_{j=0}^{\kappa-(k+l)}\Bilda{\su{j}^\sntaalpha{k+l}}&\subseteq\Bilda{\su{0}^\sntaalpha{k+l-1}}&
  &\text{and}&
  \Kerna{\su{0}^\sntaalpha{k+l-1}}&\subseteq\bigcap_{j=0}^{\kappa-(k+l)}\Kerna{\su{j}^\sntaalpha{k+l}}.
 \end{align*}
 Thus, it is sufficient to show \(\Bilda{\su{0}^\sntaalpha{k+l-1}}\subseteq\Bilda{\su{0}^\sntaalpha{k}}\) and \(\Kerna{\su{0}^\sntaalpha{k}}\subseteq\Kerna{\su{0}^\sntaalpha{k+l-1}}\). However, keeping in mind that the case \(l=1\) is trivial, these inclusions follow by induction using \rdefi{D1632} and \rrema{R1040}.
\eproof

\bremal{R1116}
 Let \(\alpha\in\C\), let \(\kappa\in\N\cup\set{+\infty}\), and let \((s_j)_{j=0}^\kappa\) be a sequence of complex \tpqa{matrices}. In view of \rlemm{L1627}, for all \(k\in\mn{0}{\kappa-1}\), all \(l\in\mn{1}{\kappa-k}\), and all \(j\in\mn{0}{\kappa-(k+l)}\), then \(\Bilda{s_j^\sntaa{k+l}{\alpha}}\subseteq\Bilda{s_0^\sntaa{k}{\alpha}}\) and \(\Kerna{s_0^\sntaa{k}{\alpha}}\subseteq\Kerna{s_j^\sntaa{k+l}{\alpha}}\) and, in particular, \(\rank(s_j^\sntaa{k+l}{\alpha})\leq\rank(s_0^\sntaa{k}{\alpha})\).
\erema

\blemml{L1643}
 Let \(\alpha\in\C\), let \(\kappa\in\NO\cup\set{+\infty}\), let \(\seq{\su{j}}{j}{0}{\kappa}\) be a sequence of complex \tpqa{matrices}, and let \(k\in\mn{0}{\kappa}\). Then:
 \benui
  \il{L1643.a} If \(\gamma\in\C\), then \(\seq{(\gamma\su{j})^\saalpha{k}}{j}{0}{\kappa-k}=\seq{\gamma\su{j}^\saalpha{k}}{j}{0}{\kappa-k}\) and \(\seq{(\gamma^j\su{j})^\saalpha{k}}{j}{0}{\kappa-k}=\seq{\gamma^{j+k}\su{j}^\saalpha{k}}{j}{0}{\kappa-k}\).
  \il{L1643.b} If \(m,n\in\N\), \(U\in\Coo{m}{p}\) with \(U^\ad U=\Ip\), and \(V\in\Coo{q}{n}\) with \(VV^\ad=\Iq\), then \(\seq{(U\su{j}V)^\saalpha{k}}{j}{0}{\kappa-k}=\seq{U\su{j}^\saalpha{k}V}{j}{0}{\kappa-k}\).
  \il{L1643.c} If the sequence \(\seq{t_j}{j}{0}{\kappa}\) is given by \(t_j\defg\su{j}^\ad\) for all \(j\in\mn{0}{\kappa}\), then \((\su{j}^\saalpha{k})^\ad=t_j^\saa{k}{\ko\alpha}\) for all \(j\in\mn{0}{\kappa-k}\).
 \eenui
\elemm
\bproof
 Use \rdefi{D1632} and \rlemm{L1228}.
\eproof

\bremal{R1320}
 Let \(\alpha\in\R\), let \(\kappa\in\NO\cup\set{+\infty}\), let \((s_j)_{j=0}^{\kappa}\) be a sequence of \tH{} complex \tqqa{matrices}, and let \(k\in\mn{0}{\kappa}\). In view of \rlemp{L1643}{L1643.c}, then \((\su{j}^\saalpha{k})^\ad=\su{j}^\saalpha{k}\) for all \(j\in\mn{0}{\kappa-k}\).
\erema

\bremal{R1642}
 Let \(\alpha\in\C\), let \(\kappa\in\NO\cup\set{+\infty}\), let \(k\in\mn{0}{\kappa}\), and let \(n\in\N\). For all \(m\in\mn{1}{n}\), let \(p_m,q_m\in\N\), and let \(\seq{\su{j}^{(m)}}{j}{0}{\kappa}\) be a sequence of complex \taaa{p_m}{q_m}{matrices} with \tsaalpha{k} \(\seq{t_j^{(m)}}{j}{0}{\kappa-k}\). In view of \rdefi{D1632} and \rrema{R1229}, then \(\seq{\diag\matauuo{t_j^{(m)}}{m}{1}{n}}{j}{0}{\kappa-k}\) is exactly the \tsaalphaa{k}{\seq{\diag\matauuo{\su{j}^{(m)}}{m}{1}{n}}{j}{0}{\kappa}}.
\erema

\bremal{R1323}
 Let \(\alpha\in\C\), let \(\kappa\in\N\cup\set{+\infty}\), let \(\seq{\su{j}}{j}{0}{\kappa}\) and \(\seq{t_j}{j}{0}{\kappa}\) be sequences of complex \tpqa{matrices} such that \(\su{0}\su{0}^\MP\su{j}\su{0}^\MP\su{0}=t_0t_0^\MP t_jt_0^\MP t_0\) for all \(j\in\mn{0}{\kappa}\), and let \(k\in\mn{1}{\kappa}\). In view of \rdefi{D1632} and \rlemm{L1411}, then \(\su{j}^\saalpha{k}=t_j^\saalpha{k}\) for all \(j\in\mn{0}{\kappa-k}\).
\erema

 Now we are going to study the \tSchur{}-type algorithm introduced in \rdefi{D1632} for sequences belonging to the class \(\Kggqkappaalpha\) and its distinguished subclasses.
 
 Now we state the main result of this section.
\btheol{P1410}
 Let \(\alpha\in\R\), let \(\kappa\in\NO\cup\set{+\infty}\), let \(\seq{\su{j}}{j}{0}{\kappa}\) be a sequence of complex \tqqa{matrices}, and let \(k\in\mn{0}{\kappa}\). Then:
 \benui
  \il{P1410.a} If \(\seq{\su{j}}{j}{0}{\kappa}\in\Kggqualpha{\kappa}\), then \(\seq{\su{j}^\sntaa{k}{\alpha}}{j}{0}{\kappa-k}\in\Kggqualpha{\kappa-k}\).
  \il{P1410.b} If \(\seq{\su{j}}{j}{0}{\kappa}\in\Kggequalpha{\kappa}\), then \(\seq{\su{j}^\sntaa{k}{\alpha}}{j}{0}{\kappa-k}\in\Kggequalpha{\kappa-k}\).
  \il{P1410.c} If \(\seq{\su{j}}{j}{0}{\kappa}\in\Kgqualpha{\kappa}\), then \(\seq{\su{j}^\sntaa{k}{\alpha}}{j}{0}{\kappa-k}\in\Kgqualpha{\kappa-k}\).
  \il{P1410.d} If \(m\in\mn{0}{\kappa}\) and \(\seq{\su{j}}{j}{0}{\kappa}\in\Kggdoqualpha{m}{\kappa}\), then \(\seq{\su{j}^\sntaa{k}{\alpha}}{j}{0}{\kappa-k}\in\Kggdoqualpha{\max\set{0,m-k}}{\kappa-k}\).
  \il{P1410.e} If \(\seq{\su{j}}{j}{0}{\kappa}\in\Kggdqualpha{\kappa}\), then \(\seq{\su{j}^\sntaa{k}{\alpha}}{j}{0}{\kappa-k}\in\Kggdqualpha{\kappa-k}\).
 \eenui
\etheo
\bproof
 In view of \rdefi{D1632}, the case \(k=0\) is trivial. Thus, there is an \(l\in\mn{0}{\kappa}\) such that the stated implications hold true for all \(k\in\mn{0}{l}\). If \(l=\kappa\), then the proof is complete. Assume that \(l<\kappa\). Then \rdefi{D1632} and \rtheo{P1546} show that the stated implications also hold true for \(k=l+1\). Thus, the assertion inductively follows.
\eproof

\bremal{R1515}
 Let \(\alpha\in\R\), let \(\kappa\in\NO\cup\set{+\infty}\), and let \(\seq{\su{j}}{j}{0}{\kappa}\in\Kggqualpha{\kappa}\). In view of \rthmp{P1410}{P1410.a} and \rlemp{L1738}{L1738.b}, then \(\su{0}^\sntaa{k}{\alpha}\in\Cggq\) for all \(k\in\mn{0}{\kappa}\).
\erema

\section{On the \hrasp{\alpha} against to the background of the \tSchur{}-type algorithm}\label{S1402}
In this section, we mainly concentrate our considerations to the class \(\Kggqkappaalpha\) or one of its distinguished subclasses \(\Kggeqkappaalpha\), \(\Kgqkappaalpha\), and \(\Kggdqkappaalpha\). In this situation, we will recognize that the \traspa{\alpha}{\seq{\su{j}}{j}{0}{\kappa}} introduced in \rdefi{D1021} is generated by the \tSchur{}-type algorithm applied to \(\seq{\su{j}}{j}{0}{\kappa}\). The following considerations are aimed at establishing block \(LDU\)~decompositions of several block Hankel matrices. In order to realize this goal, we first derive some matrix identities.

\blemml{L1705}
 Let \(\alpha\in\C\), let \(n\in\N\), and let \(\seq{\su{j}}{j}{0}{2n}\in\Dtpqu{2n}\). Let \(\Dlu{n}\) and \(\Dru{n}\) be given via \eqref{Dl} and \eqref{Dr}, respectively, and let \(\Thetauu{n}{2n}\) be given via \eqref{Xi}. Then
 \bgl{L1705.A}
  \Dlu{n}\Hu{n}\Dru{n}
  =\diaga{\su{0},-\alpha\Hu{n-1}^\seinsalpha+\Ku{n-1}^\seinsalpha}+\Thetauu{n}{2n}.
 \eg
\elemm
\bproof
 From \rlemm{L1611} we get \eqref{L1611.A}. In view of \rdefi{D1459}, we have \(\seq{\su{j}}{j}{0}{2n-1}\in\Dpqu{2n-1}\). According to \rprop{P1256}, thus \(\Ssu{n}^\MP=\Ssu{n}^\rez\) and \(\SSu{n}^\MP=\SSu{n}^\rez\). Furthermore, \(\su{0}^\rez=\su{0}^\MP\) by \rdefi{D1430}. Taking additionally into account \eqref{L1611.A}, \eqref{H(1)}, \eqref{K(1)}, \rdefi{D1059}, \eqref{H[1]}, and \eqref{K[1]} we obtain
 \[
  (\Iu{n+1}\kp\su{0})\Ssu{n}^\rez\Hu{n}\SSu{n}^\rez(\Iu{n+1}\kp\su{0})
  =\diaga{\su{0},-\alpha\Hu{n-1}^\seinsalpha+\Ku{n-1}^\seinsalpha}.
 \]
 In view of \(\seq{\su{j}}{j}{0}{2n-1}\in\Dpqu{2n-1}\), we get from \rdefi{D1658} and \rpartss{R1631.c}{R1631.b} of \rrema{R1631} furthermore \(\su{0}\su{0}^\MP\su{j}=\su{j}\) and \(\su{j}\su{0}^\MP\su{0}=\su{j}\) for all \(j\in\mn{0}{2n-1}\). Because of \eqref{H}, it follows \([\Iu{n+1}\kp(\Ip-\su{0}\su{0}^\MP)]\Hu{n}=\diag\brow\Ouu{np}{nq},\su{2n}-\su{0}\su{0}^\MP\su{2n}\erow\) and \(\Hu{n}[\Iu{n+1}\kp(\Iq-\su{0}^\MP\su{0})]=\diag\brow\Ouu{np}{nq},\su{2n}-\su{2n}\su{0}^\MP\su{0}\erow\).
 Hence,
 \begin{multline*}
  \lek\Iu{n+1}\kp(\Ip-\su{0}\su{0}^\MP)\rek\Hu{n}\lek\Iu{n+1}\kp(\Iq-\su{0}^\MP\su{0})\rek\\
  =\diag\brow\Ouu{np}{nq},\su{2n}-\su{2n}\su{0}^\MP\su{0}-\su{0}\su{0}^\MP\su{2n}+\su{0}\su{0}^\MP\su{2n}\su{0}^\MP\su{0}\erow
 \end{multline*}
 and, taking into account \eqref{Srez}, \eqref{S}, and \(\su{0}^\rez=\su{0}^\MP\), we obtain furthermore
 \[
  \lek\Iu{n+1}\kp(\Ip-\su{0}\su{0}^\MP)\rek\Hu{n}\SSu{n}^\rez(\Iu{n+1}\kp\su{0})
  =\diag\brow\Ouu{np}{nq},\su{2n}\su{0}^\MP\su{0}-\su{0}\su{0}^\MP\su{2n}\su{0}^\MP\su{0}\erow
 \]
 and
 \[
  (\Iu{n+1}\kp\su{0})\Ssu{n}^\rez\Hu{n}\lek\Iu{n+1}\kp(\Iq-\su{0}^\MP\su{0})\rek
  =\diag\brow\Ouu{np}{nq},\su{0}\su{0}^\MP\su{2n}-\su{0}\su{0}^\MP\su{2n}\su{0}^\MP\su{0}\erow.
 \]
 Using \eqref{Dl}, \eqref{Dr}, and \eqref{Xi}, we get consequently \eqref{L1705.A}.
\eproof

We now turn our attention to block \tHankel{} matrices built from the \tsaalpha{k}. Let \(\alpha\in\C\), let \(\kappa\in\NO\cup\set{+\infty}\), let \(\seq{\su{j}}{j}{0}{\kappa}\) be a sequence of complex \tpqa{matrices}, and let \(k\in\mn{0}{\kappa}\). Then, let
\bgl{H[k]}
 \Hu{n}^\sntaa{k}{\alpha}
 \defg\matauuo{\su{l+m}^\sntaa{k}{\alpha}}{l,m}{0}{n}
\eg
\index{h@$\Hu{n}^\sntaa{k}{\alpha}$}for all \(n\in\NO\) with \(2n\leq\kappa-k\) and let
\bgl{K[k]}
 \Ku{n}^\sntaa{k}{\alpha}
 \defg\matauuo{\su{l+m+1}^\sntaa{k}{\alpha}}{l,m}{0}{n}
\eg
\index{k@$\Ku{n}^\sntaa{k}{\alpha}$}for all \(n\in\NO\) with \(2n+1\leq\kappa-k\).

\blemml{L1441}%
 Let \(\alpha\in\R\), let \(n\in\N\), and let \(\seq{\su{j}}{j}{0}{2n}\in\Kggquu{2n}{\alpha}\). Denote by \(\seq{t_j}{j}{0}{2n-1}\) the first \tlasnta{\alpha}{\seq{\su{j}}{j}{0}{2n}} and by \(\seq{u_j}{j}{0}{2n-1}\) the \tsplusalphata{\seq{t_j}{j}{0}{2n-1}}. Then
 \[
  \lrk\diaga{\Iq,\Dluo{n-1}{u}}\rrk\Dlu{n}\Hu{n}\Dru{n}\lrk\diaga{\Iq,\Druo{n-1}{u}}\rrk
  =\diaga{\su{0},\Hu{n-1}^\sntaa{2}{\alpha}}+\Thetauu{n}{2n}+\Thetauuo{n}{2n-1}{t}.
 \]
\elemm
\bproof
 Because of \rpropp{P1442}{P1442.b}, the sequence \(\seq{\su{j}}{j}{0}{2n}\) belongs to \(\Dtqqu{2n}\). By virtue of \rlemm{L1705}, \requa{L1705.A} holds true. In view of \eqref{Xi}, this implies
 \bgl{L1441.1}
  \Dlu{n}\Hu{n}\Dru{n}
  =\diaga{\su{0},-\alpha\Huo{n-1}{t}+\Kuo{n-1}{t}+\Thetauu{n-1}{2n}}.
 \eg
 Taking into account \rthmp{P1546}{P1546.a} and \rpropp{P1442}{P1442.c}, we see that the sequence \(\seq{t_j}{j}{0}{2n-1}\) belongs to \(\Dtqqu{2n-1}\). Thus, the application of \rlemp{L1548}{L1548.a} (, more precisely, the application of equation~\eqref{L1548.B} to the sequence \(\seq{t_j}{j}{0}{2n-1}\),) provides us then
 \[
  \Dluo{n-1}{u}(-\alpha\Huo{n-1}{t}+\Kuo{n-1}{t}-\Thetauuo{n-1}{2n-1}{t})\Druo{n-1}{u}
  =\Huo{n-1}{v},
 \]
 where the sequence \(\seq{v_j}{j}{0}{2n-2}\) is given by \(v_j\defg t_j^\seinsalpha\) for all \(j\in\mn{0}{2n-2}\). In view of \rrema{R1247} and \eqref{Xi}, we have
 \[
  \Dluo{n-1}{u}(\Thetauuo{n-1}{2n-1}{t}+\Thetauu{n-1}{2n})\Druo{n-1}{u}
  =\Thetauuo{n-1}{2n-1}{t}+\Thetauu{n-1}{2n}.
 \]
 From \rdefi{D1632} we see that \(\seq{v_j}{j}{0}{2n-2}\) is the second \tlasnta{\alpha}{\seq{\su{j}}{j}{0}{2n}} which, in view of \eqref{H[k]}, implies \(\Huo{n-1}{v}=\Hu{n-1}^\sntaa{2}{\alpha}\). Hence, we obtain
 \bgl{L1441.2}
  \begin{split}
   &\Dluo{n-1}{u}(-\alpha\Huo{n-1}{t}+\Kuo{n-1}{t}+\Thetauu{n-1}{2n})\Druo{n-1}{u}\\
   &=\Dluo{n-1}{u}(-\alpha\Huo{n-1}{t}+\Kuo{n-1}{t}-\Thetauuo{n-1}{2n-1}{t})\Druo{n-1}{u}+\Dluo{n-1}{u}(\Thetauuo{n-1}{2n-1}{t}+\Thetauu{n-1}{2n})\Druo{n-1}{u}\\
   &=\Huo{n-1}{v}+\Thetauuo{n-1}{2n-1}{t}+\Thetauu{n-1}{2n}
   =\Hu{n-1}^\sntaa{2}{\alpha}+\Thetauu{n-1}{2n}+\Thetauuo{n-1}{2n-1}{t}
  \end{split}
 \eg
 and, taking \eqref{L1441.1}, \eqref{L1441.2}, and \eqref{Xi} into account, finally
 \[
  \begin{split}
   &\lrk\diaga{\Iq,\Dluo{n-1}{u}}\rrk\Dlu{n}\Hu{n}\Dru{n}\lrk\diaga{\Iq,\Druo{n-1}{u}}\rrk\\
   &=\diagA{\su{0},\Dluo{n-1}{u}(-\alpha\Huo{n-1}{t}+\Kuo{n-1}{t}+\Thetauu{n-1}{2n})\Druo{n-1}{u}}\\
   &=\diaga{\su{0},\Hu{n-1}^\sntaa{2}{\alpha}+\Thetauu{n-1}{2n}+\Thetauuo{n-1}{2n-1}{t}}
   =\diaga{\su{0},\Hu{n-1}^\sntaa{2}{\alpha}}+\Thetauu{n}{2n}+\Thetauuo{n}{2n-1}{t}.\qedhere
  \end{split}
 \]
\eproof

\blemml{L1007}%
 Let \(\alpha\in\R\), let \(n\in\N\), and let \(\seq{\su{j}}{j}{0}{2n+1}\in\Kggquu{2n+1}{\alpha}\).  Denote by \(\seq{r_j}{j}{0}{2n+1}\) the \tsplusalphata{\seq{\su{j}}{j}{0}{2n+1}} and by \(\seq{t_j}{j}{0}{2n}\) the first \tlasnta{\alpha}{\seq{\su{j}}{j}{0}{2n+1}}. Then
 \[
  \Dluo{n}{t}\Dluo{n}{r}(-\alpha\Hu{n}+\Ku{n})\Druo{n}{r}\Druo{n}{t}
  =\diaga{\su{0}^\sntaa{1}{\alpha},-\alpha\Hu{n-1}^\sntaa{2}{\alpha}+\Ku{n-1}^\sntaa{2}{\alpha}}+\Thetauu{n}{2n+1}+\Thetauuo{n}{2n}{t}.
 \]
\elemm
\bproof
 Because of \rpropp{P1442}{P1442.c} the sequence \(\seq{\su{j}}{j}{0}{2n+1}\) belongs to \(\Dtqqu{2n+1}\). Thus, by virtue of \rlemp{L1548}{L1548.a}, we obtain then the equation \(\Hu{n}^\sntaa{1}{\alpha}=\Dlu{n}^\splusalpha(\Hau{n}-\Thetauu{n}{2n+1})\Dru{n}^\splusalpha\) which, in view of \eqref{D[+]} and \eqref{-aH+K=Ha}, implies
 \bgl{NCM47}
  \Dluo{n}{r}(-\alpha\Hu{n}+\Ku{n})\Druo{n}{r}
  =\Hu{n}^\sntaa{1}{\alpha}+\Dluo{n}{r}\Thetauu{n}{2n+1}\Druo{n}{r}.
 \eg
 According to  \rthmp{P1546}{P1546.a} and \rpropp{P1442}{P1442.b}, we have \(\seq{t_j}{j}{0}{2n}\in\Dtqqu{2n}\). Using \eqref{H[1]} and \rlemm{L1705}, we then obtain
 \bgl{NCM48}
  \Dluo{n}{t}\Hu{n}^\sntaa{1}{\alpha}\Druo{n}{t}
  =\Dluo{n}{t}\Huo{n}{t}\Druo{n}{t}
  =\diaga{t_0,-\alpha\Huo{n-1}{v}+\Kuo{n-1}{v}}+\Thetauuo{n}{2n}{t},
 \eg
 where the sequence \(\seq{v_j}{j}{0}{2n-1}\) is given by \(v_j\defg t_j^\seinsalpha\) for all \(j\in\mn{0}{2n-1}\). In view of \rrema{R1247} and \eqref{Xi}, we have
 \bgl{JSB}
  \Dluo{n}{t}\Dluo{n}{r}\Thetauu{n}{2n+1}\Druo{n}{r}\Druo{n}{t}
  =\Dluo{n}{t}\Thetauu{n}{2n+1}\Druo{n}{t}
  =\Thetauu{n}{2n+1}.
 \eg
 From \rdefi{D1632} we see that \(\seq{v_j}{j}{0}{2n-1}\) is the second \tlasnta{\alpha}{\seq{\su{j}}{j}{0}{2n+1}} which, in view of \eqref{H[k]} and \eqref{K[k]}, implies
 \bgl{LVB}
  -\alpha\Huo{n-1}{v}+\Kuo{n-1}{v}
  =-\alpha\Hu{n-1}^\sntaa{2}{\alpha}+\Ku{n-1}^\sntaa{2}{\alpha}.
 \eg
 Finally, using \eqref{NCM47}, \eqref{NCM48}, \eqref{JSB}, and \eqref{LVB} we obtain then
 \[
  \begin{split}
   \Dluo{n}{t}\Dluo{n}{r}(-\alpha\Hu{n}+\Ku{n})\Druo{n}{r}\Druo{n}{t}
   &=\Dluo{n}{t}\Hu{n}^\sntaa{1}{\alpha}\Druo{n}{t}+\Dluo{n}{t}\Dluo{n}{r}\Thetauu{n}{2n+1}\Druo{n}{r}\Druo{n}{t}\\
   &=\diaga{t_0,-\alpha\Huo{n-1}{v}+\Kuo{n-1}{v}}+\Thetauuo{n}{2n}{t}+\Thetauu{n}{2n+1}\\
   &=\diaga{\su{0}^\sntaa{1}{\alpha},-\alpha\Hu{n-1}^\sntaa{2}{\alpha}+\Ku{n-1}^\sntaa{2}{\alpha}}+\Thetauu{n}{2n+1}+\Thetauuo{n}{2n}{t}.\qedhere
  \end{split}
 \]
\eproof

 Now we introduce further particular matrices, which play an essential role in our further considerations. We will see that these matrices indicate in some sense how far a sequence \(\seq{\su{j}}{j}{0}{\kappa}\in\Kggqkappaalpha\) is away from the class \(\Kggeqkappaalpha\). Let \(\alpha\in\C\), let \(\kappa\in\NO\cup\set{+\infty}\), and let \(\seq{\su{j}}{j}{0}{\kappa}\) be a sequence of complex \tpqa{matrices}. Then let
\bgl{P[k]}
 \Puu{j}{k}
 \defg\su{j}^\sntaa{k}{\alpha}-\su{0}^\sntaa{k}{\alpha}(\su{0}^\sntaa{k}{\alpha})^\MP \su{j}^\sntaa{k}{\alpha}(\su{0}^\sntaa{k}{\alpha})^\MP \su{0}^\sntaa{k}{\alpha}
\eg
\index{e@$\Puu{j}{k}$}for all \(k\in\mn{0}{\kappa}\) and all \(j\in\mn{0}{\kappa-k}\). Furthermore, for all \(m\in\mn{0}{\kappa}\) and all \(l\in\mn{m}{\kappa}\), let\index{r@$\Zuu{l}{m}$}
\bgl{Z}
  \Zuu{l}{m}
  \defg
  \begin{cases}
   \Oqq\incase{m=0}\\
   \sum_{k=0}^{m-1}\Puu{l-k}{k}\incase{m\geq1}
  \end{cases}.
\eg

\blemml{L1143}
 Let \(\alpha\in\R\), let \(\kappa\in\NO\cup\set{+\infty}\), and let \(\seq{\su{j}}{j}{0}{\kappa}\in\Kggequu{\kappa}{\alpha}\). Then \(\Puu{j}{k}=\Oqq\) for all \(k\in\mn{0}{\kappa}\) and all \(j\in\mn{0}{\kappa-k}\). Furthermore, \(\Zuu{l}{m}=\Oqq\) for all \(m\in\mn{0}{\kappa}\) and all \(l\in\mn{m}{\kappa}\).
\elemm
\bproof
 Let \(k\in\mn{0}{\kappa}\). According to \rthmp{P1410}{P1410.b}, we have \(\seq{\su{j}^\sntaa{k}{\alpha}}{j}{0}{\kappa-k}\in\Kggequu{\kappa-k}{\alpha}\), which, in view of \rpropp{P1442}{P1442.a}, implies \(\seq{\su{j}^\sntaa{k}{\alpha}}{j}{0}{\kappa-k}\in\Dqqu{\kappa-k}\). From \eqref{P[k]}, \rdefi{D1658}, and \rpartss{R1631.c}{R1631.b} of \rrema{R1631}, we obtain then \(\Puu{j}{k}=\Oqq\) for all \(j\in\mn{0}{\kappa-k}\). In view of \eqref{Z}, this implies \(\Zuu{l}{m}=\Oqq\) for all \(m\in\mn{0}{\kappa}\) and all \(l\in\mn{m}{\kappa}\).
\eproof

\bremal{R1622}
 Let \(\alpha\in\R\), let \(\kappa\in\N\cup\set{+\infty}\), and let \(\seq{\su{j}}{j}{0}{\kappa}\in\Kggquu{\kappa}{\alpha}\). Because of \eqref{P[k]}, \eqref{Z}, \rrema{R1501}, and \rlemm{L1143}, then \(\Puu{j}{k}=\Oqq\) for all \(k\in\mn{0}{\kappa-1}\) and all \(j\in\mn{0}{\kappa-1-k}\) and, furthermore, \(\Zuu{l}{m}=\Oqq\) for all \(m\in\mn{0}{\kappa-1}\) and all \(l\in\mn{m}{\kappa-1}\).
\erema

\blemml{L1602}
 Let \(\alpha\in\C\), let \(\kappa\in\N\cup\set{+\infty}\), and let \(\seq{\su{j}}{j}{0}{\kappa}\) be a sequence of complex \tpqa{matrices}. For all \(k\in\mn{0}{\kappa-1}\), then \(-\alpha\su{0}^\sntaa{k}{\alpha}+\su{1}^\sntaa{k}{\alpha}=\su{0}^\sntaa{k+1}{\alpha}+\Puu{1}{k}\).
\elemm
\bproof
 Let \(k\in\mn{0}{\kappa-1}\). Denote by \(\seq{t_j}{j}{0}{\kappa-k}\) the \tsaalphaa{k}{\seq{\su{j}}{j}{0}{\kappa}}. The application of \rlemm{L1613} to the sequence \(\seq{t_j}{j}{0}{\kappa-k}\) yields then \(t_0^\sntaa{1}{\alpha}=t_0t_0^\MP t_1^\splusalpha t_0^\MP t_0\) which, in view of \rdefi{D1455}, implies
 \[
  t_0^\sntaa{1}{\alpha}
  =t_0t_0^\MP(-\alpha t_0 +t_1 )t_0^\MP t_0 
  =-\alpha t_0 t_0^\MP t_0 t_0^\MP t_0 +t_0 t_0^\MP t_1 t_0^\MP t_0 
  =-\alpha t_0 +t_0 t_0^\MP t_1 t_0^\MP t_0 .
 \]
 Taking additionally \rdefi{D1632} and \eqref{P[k]} into account, this implies
 \[
  -\alpha\su{0}^\sntaa{k}{\alpha}+\su{1}^\sntaa{k}{\alpha}
  =-\alpha t_0+t_1
  =t_0^\sntaa{1}{\alpha}-t_0t_0^\MP t_1t_0^\MP t_0+t_1
  =\su{0}^\sntaa{k+1}{\alpha}+\Puu{1}{k}.\qedhere
 \]
\eproof

\bremal{R1610}
 Let \(\alpha\in\C\) and let \(\seq{\su{j}}{j}{0}{0}\) be a sequence of complex \tpqa{matrices}. In view of \eqref{H}, \rdefi{D1632}, and \eqref{Z}, then \(\Hu{0}=\su{0}^\sntaa{0}{\alpha}+\Zuu{0}{0}\).
\erema

\bremal{R1613}
 Let \(\alpha\in\C\) and let \(\seq{\su{j}}{j}{0}{1}\) be a sequence from \(\Cpq\). In view of \eqref{H}, \eqref{K}, \rdefi{D1632}, \rlemm{L1602}, and \eqref{Z}, then \(-\alpha\Hu{0}+\Ku{0}=\su{0}^\sntaa{1}{\alpha}+\Zuu{1}{1}\).
\erema

The following two lemmas play a key role in the proof of central results of this paper, because they provide the block \(LDU\)~decompositions we are striving for. Note that the sets \(\nudqu{n}\) and \(\nodqu{n}\) were introduced in \rnota{N1039}. 
\blemml{L1440}
 Let \(\alpha\in\R\), let \(n\in\N\), and let \(\seq{\su{j}}{j}{0}{2n}\in\Kggquu{2n}{\alpha}\). For all \(k\in\mn{0}{2n}\), denote by \(\seq{t_j^{(k)}}{j}{0}{2n-k}\) the \tsaalphaa{k}{\seq{\su{j}}{j}{0}{2n}} and by \(\seq{u_j^{(k)}}{j}{0}{2n-k}\) the \tsplusalphata{\seq{t_j^{(k)}}{j}{0}{2n-k}}. For all \(l\in\mn{0}{n-1}\), let
 \[
  \Vlu{l}
  \defg
  \begin{cases}
   \diaga{\Iq,\Dluo{n-1}{u^{(2l+1)}}}\cdot\Dluo{n}{t^{(2l)}}\ifa{l=0}\\
   \diaga{I_{lq},\diaga{\Iq,\Dluo{n-l-1}{u^{(2l+1)}}}\cdot\Dluo{n-l}{t^{(2l)}}}\ifa{l\geq1}
  \end{cases}
 \]
 and
 \[
  \Vru{l}
  \defg
  \begin{cases}
   \Druo{n}{t^{(2l)}}\cdot\diaga{\Iq,\Druo{n-1}{u^{(2l+1)}}}\ifa{l=0}\\
   \diaga{I_{lq},\Druo{n-l}{t^{(2l)}}\cdot\diaga{\Iq,\Druo{n-l-1}{u^{(2l+1)}}}}\ifa{l\geq1}
  \end{cases}.
 \]
 \benui
  \il{L1440.a} For all \(m\in\mn{1}{n}\), the matrix \(\VVlu{m}\defg\Vlu{m-1}\Vlu{m-2} \dotsm \Vlu{0}\) belongs to \(\nudqu{n}\) and the matrix \(\VVru{m}\defg\Vru{0}\Vru{1} \dotsm \Vru{m-1}\) belongs to \(\nodqu{n}\).
  \il{L1440.b} \(\VVlu{n}H_n\VVru{n}=\diaga{\su{0}^\sntaa{0}{\alpha},\su{0}^\sntaa{2}{\alpha},\dotsc,\su{0}^\sntaa{2(n-1)}{\alpha},\su{0}^\sntaa{2n}{\alpha}+\Zuu{2n}{2n}}\).
  \il{L1440.c} If \(\seq{\su{j}}{j}{0}{2n}\in\Kggequu{2n}{\alpha}\), then \(\VVlu{n}H_n\VVru{n}=\diaga{\su{0}^\sntaa{0}{\alpha},\su{0}^\sntaa{2}{\alpha},\dotsc,\su{0}^\sntaa{2n}{\alpha}}\).
 \eenui
\elemm
\bproof
 \eqref{L1440.a} Using \rremasss{R1247}{R1745}{R1042}, we easily see that \(\Vlu{l}\in\nudqu{n}\) and that \(\Vru{l}\in\nodqu{n}\) hold true for all \(l\in\mn{0}{n-1}\) which, in view of \rrema{R1745}, implies \(\VVlu{n}\in\nudqu{n}\) and \(\VVru{n}\in\nodqu{n}\).
 
 \eqref{L1440.b} We prove~\eqref{L1440.b} by induction. In view of \eqref{Xi}, \eqref{P[k]}, and \eqref{Z}, there is, according to \rlemm{L1441} and \rdefi{D1632}, some \(m\in\mn{1}{n}\) such that for all \(k\in\mn{1}{m}\) the following equation holds:
 \begin{itemize}
  \item[(I\(_k\))] \(\VVlu{k}H_n\VVru{k}=\diaga{\su{0}^\sntaa{0}{\alpha},\su{0}^\sntaa{2}{\alpha}, \dotsc ,\su{0}^\sntaa{2(k-1)}{\alpha},\Hu{n-k}^\sntaa{2k}{\alpha}}+\diaga{\Ouu{nq}{nq},\Zuu{2n}{2k}}\).
 \end{itemize}
 According to \eqref{H[k]}, we have furthermore
 \bgl{L1440.0}
  \Hu{0}^\sntaa{2n}{\alpha}
  =\su{0}^\sntaa{2n}{\alpha}.
 \eg
 
 If \(m=n\), then~\eqref{L1440.b} immediately follows from~(I\(_m\)) and \eqref{L1440.0}.
 
 Now we consider the case \(m<n\). Obviously, using~(I\(_m\)), we easily see that
 \bspl{L1440.1}
  &\VVlu{m+1}H_n\VVru{m+1}
  =\Vlu{m}\VVlu{m}H_n\VVru{m}\Vru{m}\\
  &=\diag\Bigl[\su{0}^\sntaa{0}{\alpha},\su{0}^\sntaa{2}{\alpha}, \dotsc ,\su{0}^\sntaa{2(m-1)}{\alpha},\\
  &\hspace{90pt}\diaga{\Iq,\Dluo{n-m-1}{u^{(2m+1)}}}\Dluo{n-m}{t^{(2m)}}\Hu{n-m}^\sntaa{2m}{\alpha}\Druo{n-m}{t^{(2m)}}\diaga{\Iq,\Druo{n-m-1}{u^{(2m+1)}}}\Bigr]\\
  &\qquad+\Vlu{m}\diaga{\Ouu{nq}{nq},\Zuu{2n}{2m}}\Vru{m}.
 \esp
 Due to \rthmp{P1410}{P1410.a}, the sequence \(\seq{\su{j}^\sntaa{2m}{\alpha}}{j}{0}{2(n-m)}\) belongs to \(\Kggquu{2(n-m)}{\alpha}\). The application of \rlemm{L1441} to the sequence \(\seq{\su{j}^\sntaa{2m}{\alpha}}{j}{0}{2(n-m)}\) yields, in view of \eqref{H[k]} and \rdefi{D1632}, then
 \begin{multline}\label{L1440.4}
  \diaga{\Iq,\Dluo{n-m-1}{u^{(2m+1)}}}\Dluo{n-m}{t^{(2m)}}\Hu{n-m}^\sntaa{2m}{\alpha}\Druo{n-m}{t^{(2m)}}\diaga{\Iq,\Druo{n-m-1}{u^{(2m+1)}}}\\
  =\diaga{\su{0}^\sntaa{2m}{\alpha},\Hu{n-m-1}^\sntaa{2m+2}{\alpha}}+\Thetauuo{n-m}{2(n-m)}{t^{(2m)}}+\Thetauuo{n-m}{2(n-m)-1}{t^{(2m+1)}}.
 \end{multline}
 Because of \(\Vlu{m}\in\nudqu{n}\) and \(\Vru{m}\in\nodqu{n}\), we have furthermore
 \bgl{L1440.5}
  \Vlu{m}\diaga{\Ouu{nq}{nq},\Zuu{2n}{2m}}\Vru{m}
  =\diaga{\Ouu{nq}{nq},\Zuu{2n}{2m}}.
 \eg
 Taking into account \eqref{L1440.1}, \eqref{L1440.4}, \eqref{L1440.5}, and \eqref{Xi}, thus
 \begin{multline}\label{L1440.2}
  \VVlu{m+1}H_n\VVru{m+1}
  =\diaga{\su{0}^\sntaa{0}{\alpha},\su{0}^\sntaa{2}{\alpha}, \dotsc ,\su{0}^\sntaa{2(m-1)}{\alpha},\su{0}^\sntaa{2m}{\alpha},\Hu{n-(m+1)}^\sntaa{2(m+1)}{\alpha}}\\
  +\Thetauuo{n}{2(n-m)}{t^{(2m)}}+\Thetauuo{n}{2(n-m)-1}{t^{(2m+1)}}+\diaga{\Ouu{nq}{nq},\Zuu{2n}{2m}}.
 \end{multline}
 In view of \eqref{Xi}, \eqref{P[k]}, and \eqref{Z}, we have moreover
 \bgl{L1440.3}
  \Thetauuo{n}{2(n-m)}{t^{(2m)}}+\Thetauuo{n}{2(n-m)-1}{t^{(2m+1)}}+\diaga{\Ouu{nq}{nq},\Zuu{2n}{2m}}
  =\diaga{\Ouu{nq}{nq},\Zuu{2n}{2(m+1)}}.
 \eg
 From \eqref{L1440.2} and \eqref{L1440.3} it follows~(I\(_{m+1}\)). Hence,~(I\(_n\)) is inductively proved. In view of \eqref{L1440.0}, the proof is of~\eqref{L1440.b} finished.
 
 \eqref{L1440.c} Combine~\eqref{L1440.b} and \rlemm{L1143}.
\eproof

\blemml{L1306}
 Let \(\alpha\in\R\), let \(n\in\N\), and let \(\seq{\su{j}}{j}{0}{2n+1}\in\Kggquu{2n+1}{\alpha}\). For all \(k\in\mn{0}{2n+1}\), denote by \(\seq{t_j^{(k)}}{j}{0}{2n+1-k}\) the \tsaalphaa{k}{\seq{\su{j}}{j}{0}{2n+1}} and by \(\seq{u_j^{(k)}}{j}{0}{2n+1-k}\) the \tsplusalphata{\seq{t_j^{(k)}}{j}{0}{2n+1-k}}. For all \(l\in\mn{0}{n-1}\), let
 \[
  \Wlu{l}
  \defg
  \begin{cases}
   \Dluo{n}{t^{(2l+1)}}\Dluo{n}{u^{(2l)}}\incase{l=0}\\
   \diaga{\Iu{lq},\Dluo{n-l}{t^{(2l+1)}}\Dluo{n-l}{u^{(2l)}}}\incase{l\geq1}
  \end{cases}
 \]
 and
 \[
  \Wru{l}
  \defg
  \begin{cases}
   \Druo{n}{u^{(2l)}}\Druo{n}{t^{(2l+1)}}\incase{l=0}\\
   \diaga{I_{lq},\Druo{n-l}{u^{(2l)}}\Druo{n-l}{t^{(2l+1)}}}\incase{l\geq1}
  \end{cases}.
 \]
 \benui
  \il{L1306.a} For all \(m\in\mn{1}{n}\), the matrices \(\WWlu{m}\defg\Wlu{m-1}\Wlu{m-2}\dotsm\Wlu{0}\) and \(\WWru{m}\defg\Wru{0}\Wru{1}\dotsm\Wru{m-1}\) belong to \(\nudqu{n}\) and to \(\nodqu{n}\), respectively.
  \il{L1306.b} \(\WWlu{n}(-\alpha\Hu{n}+\Ku{n})\WWru{n}=\diaga{\su{0}^\sntaa{1}{\alpha},\su{0}^\sntaa{3}{\alpha}, \dotsc,\su{0}^\sntaa{2n-1}{\alpha} ,\su{0}^\sntaa{2n+1}{\alpha}+\Zuu{2n+1}{2n+1}}\).
  \il{L1306.c} If \(\seq{\su{j}}{j}{0}{2n+1}\in\Kggequu{2n+1}{\alpha}\), then \(\WWlu{n}(-\alpha\Hu{n}+\Ku{n})\WWru{n}=\diaga{\su{0}^\sntaa{1}{\alpha},\su{0}^\sntaa{3}{\alpha}, \dotsc,\su{0}^\sntaa{2n+1}{\alpha}}\).
 \eenui
\elemm
\bproof
 \eqref{L1306.a} Using \rremasss{R1247}{R1745}{R1042}, we easily see that \(\Wlu{l}\in\nudqu{n}\) and that \(\Wru{l}\in\nodqu{n}\) hold true for all \(l\in\mn{0}{n-1}\) which, in view of \rrema{R1745}, implies \(\WWlu{n}\in\nudqu{n}\) and \(\WWru{n}\in\nodqu{n}\).
 
 \eqref{L1306.b} We prove~\eqref{L1306.b} by induction. According to \rlemm{L1007}, there is, in view of \eqref{Xi}, \eqref{P[k]}, and \eqref{Z}, some \(m\in\mn{1}{n}\) such that for all \(k\in\mn{1}{m}\) the following equation holds:
 \begin{itemize}
  \item[(I\(_k\))] \(\WWlu{k}H_n\WWru{k}=\diaga{\su{0}^\sntaa{1}{\alpha},\dotsc,\su{0}^\sntaa{2k-1}{\alpha},-\alpha\Hu{n-k}^\sntaa{2k}{\alpha}+\Ku{n-k}^\sntaa{2k}{\alpha}}+\diaga{\Ouu{nq}{nq},\Zuu{2n+1}{2k}}\).
 \end{itemize}
 The application of \rlemm{L1602} yields \(-\alpha\su{0}^\sntaa{2n}{\alpha}+\su{1}^\sntaa{2n}{\alpha}=\su{0}^\sntaa{2n+1}{\alpha}+\Puu{1}{2n}\). Taking additionally into account \eqref{H[k]} and \eqref{K[k]}, we obtain then
 \begin{equation}\label{L1306.2}
  -\alpha\Hu{0}^\sntaa{2n}{\alpha}+\Ku{0}^\sntaa{2n}{\alpha}
  =-\alpha\su{0}^\sntaa{2n}{\alpha}+\su{1}^\sntaa{2n}{\alpha}
  =\su{0}^\sntaa{2n+1}{\alpha}+\Puu{1}{2n}.
 \end{equation}
 In view of \eqref{Z}, we have furthermore
 \bgl{L1306.1}
  \Puu{1}{2n}+\Zuu{2n+1}{2n}
  =\Zuu{2n+1}{2n+1}.
 \eg
 
 If \(m=n\), then~\eqref{L1306.b} immediately follows from~(I\(_m\)), \eqref{L1306.2}, and \eqref{L1306.1}.
 
 Now we consider the case \(m<n\). Obviously, from~(I\(_m\)) we have
 \bspl{L1306.3}
  &\WWlu{m+1}H_n\WWru{m+1}
  =\Wlu{m}\WWlu{m}H_n\WWru{m}\Wru{m}\\
  &=\diag\Bigl[\su{0}^\sntaa{1}{\alpha},\su{0}^\sntaa{3}{\alpha}, \dotsc ,\su{0}^\sntaa{2m-1}{\alpha},\\
  &\hspace{103pt}\Dluo{n-m}{t^{(2m+1)}}\Dluo{n-m}{u^{(2m)}}(-\alpha\Hu{n-m}^\sntaa{2m}{\alpha}+\Ku{n-m}^\sntaa{2m}{\alpha})\Druo{n-m}{u^{(2m)}}\Druo{n-m}{t^{(2m+1)}}\Bigr]\\
  &\qquad+\Wlu{m}\diaga{\Ouu{nq}{nq},\Zuu{2n+1}{2m}}\Wru{m}.
 \esp
 According to \rthmp{P1410}{P1410.a}, the sequence \(\seq{\su{j}^\sntaa{2m}{\alpha}}{j}{0}{2(n-m)+1}\) belongs to \(\Kggquu{2(n-m)+1}{\alpha}\). The application of \rlemm{L1007} to the sequence \(\seq{\su{j}^\sntaa{2m}{\alpha}}{j}{0}{2(n-m)+1}\) yields, in view of \eqref{H[k]}, \eqref{K[k]}, and  \rdefi{D1632}, then
 \begin{multline*}
  \Dluo{n-m}{t^{(2m+1)}}\Dluo{n-m}{u^{(2m)}}(-\alpha\Hu{n-m}^\sntaa{2m}{\alpha}+\Ku{n-m}^\sntaa{2m}{\alpha})\Druo{n-m}{u^{(2m)}}\Druo{n-m}{t^{(2m+1)}}\\
  =\diaga{\su{0}^\sntaa{2m+1}{\alpha},-\alpha\Hu{n-m-1}^\sntaa{2m+2}{\alpha}+\Ku{n-m-1}^\sntaa{2m+2}{\alpha}}+\Thetauuo{n-m}{2(n-m)+1}{t^{(2m)}}+\Thetauuo{n-m}{2(n-m)}{t^{(2m+1)}}.
 \end{multline*}
 Because of \(\Wlu{m}\in\nudqu{n}\) and \(\Wru{m}\in\nodqu{n}\), furthermore we get \(\Wlu{m}\cdot\diaga{\Ouu{nq}{nq},\Zuu{2n+1}{2m}}\Wru{m}=\diaga{\Ouu{nq}{nq},\Zuu{2n+1}{2m}}\). Taking into account \eqref{L1306.3} and \eqref{Xi}, thus we have
 \begin{multline}\label{L1306.4}
  \WWlu{m+1}H_n\WWru{m+1}\\
  =\diaga{\su{0}^\sntaa{1}{\alpha},\su{0}^\sntaa{3}{\alpha}, \dotsc ,\su{0}^\sntaa{2m-1}{\alpha},\su{0}^\sntaa{2m+1}{\alpha},-\alpha\Hu{n-(m+1)}^\sntaa{2(m+1)}{\alpha}+\Ku{n-(m+1)}^\sntaa{2(m+1)}{\alpha}}\\
  +\Thetauuo{n}{2(n-m)+1}{t^{(2m)}}+\Thetauuo{n}{2(n-m)}{t^{(2m+1)}}+\diaga{\Ouu{nq}{nq},\Zuu{2n+1}{2m}}.
 \end{multline}
 In view of \eqref{Xi}, \eqref{P[k]}, and \eqref{Z}, we have moreover
 \bgl{L1306.5}
  \Thetauuo{n}{2(n-m)+1}{t^{(2m)}}+\Thetauuo{n}{2(n-m)}{t^{(2m+1)}}+\diaga{\Ouu{nq}{nq},\Zuu{2n+1}{2m}}
  =\diaga{\Ouu{nq}{nq},\Zuu{2n+1}{2(m+1)}}.
 \eg
 From \eqref{L1306.4} and \eqref{L1306.5} it follows~(I\(_{m+1}\)). Hence,~(I\(_n\)) is inductively proved. In view of \eqref{L1306.2} and \eqref{L1306.1}, the proof of~\eqref{L1306.b} is finished.
 
 \eqref{L1306.c} Combine~\eqref{L1306.b} and \rlemm{L1143}.
\eproof

\bremal{R1500}
 Let \(n\in\N\) and let \((s_j)_{j=0}^{2n}\in\Hggqu{2n}\). Then, according to~\zitaa{MR2805417}{\cprop{4.17}}, there exists a matrix \(\FFru{n}\in\noduu{q}{n}\) such that
 \bgl{R1500.A}
  \FFru{n}^\ad H_n \FFru{n}
  =\diaga{\Lu{0},\Lu{1},\dotsc,\Lu{n}}.
 \eg
\erema

The following \rthmss{T1016}{T1615} indicate an explicit connection between the \tlasnt{\alpha}s and the \trasp{\alpha}. These results can be considered as the first two main results of this paper.
\btheol{T1016}
 Let \(\alpha\in\R\), let \(m\in\NO\), and let \(\seq{\su{j}}{j}{0}{m}\in\Kggquu{m}{\alpha}\). Then \(\Spu{m}=\su{0}^\sntaa{m}{\alpha}+\Zuu{m}{m}\) and, in the case \(m\geq1\), furthermore \(\Spu{j}=\su{0}^\sntaa{j}{\alpha}\) for all \(j\in\mn{0}{m-1}\).
\etheo
\bproof
 First observe that, in view of \rdefi{D1021}, \eqref{L}, \eqref{H}, and \rrema{R1610}, we have
 \bgl{T1016.1}
  \Spu{0}
  =\Lu{0}
  =\su{0}
  =\Hu{0}
  =\su{0}^\sntaa{0}{\alpha}+\Zuu{0}{0}.
 \eg
 In the case \(m=0\), the proof is complete. In the case \(m\geq1\), from \rdefi{D1021}, \eqref{La}, \eqref{L}, \eqref{s_a}, \eqref{H}, \eqref{K}, and \rrema{R1613}, we see
 \bgl{T1016.2}
  \Spu{1}
  =\Lau{0}
  =\sau{0}
  =-\alpha\su{0}+\su{1}
  =-\alpha\Hu{0}+\Ku{0}
  =\su{0}^\sntaa{1}{\alpha}+\Zuu{1}{1}
 \eg
 and, in view of \rrema{R1622}, furthermore
 \bgl{T1016.3}
  \Zuu{m-1}{m-1}
  =\Oqq.
 \eg
 
 In the case \(m=1\), the assertion follows from \eqref{T1016.2}, \eqref{T1016.1}, and \eqref{T1016.3}.
 
 Now we consider the case \(m=2n\) with some \(n\in\N\). Then, from \eqref{Kgg2n} we obtain \(\seq{\su{j}}{j}{0}{2n}\in\Hggqu{2n}\). Thus, in view of \rrema{R1500}, there exists a matrix \(\FFru{n}\in\noduu{q}{n}\) such that \eqref{R1500.A} holds true. Taking into account \rdefi{D1021}, we get then
 \bgl{T1016.4}
  \FFru{n}^\ad H_n \FFru{n}
  =\diaga{\Spu{0},\Spu{2},\dotsc,\Spu{2n}}.
 \eg
 Furthermore, according to \rlemm{L1440}, there exist matrices \(\VVlu{n}\in\nudqu{n}\) and \(\VVru{n}\in\nodqu{n}\) such that
 \bgl{T1016.5}
  \VVlu{n}H_n\VVru{n}
  =\diaga{\su{0}^\sntaa{0}{\alpha},\su{0}^\sntaa{2}{\alpha},\dotsc,\su{0}^\sntaa{2(n-1)}{\alpha},\su{0}^\sntaa{2n}{\alpha}+\Zuu{2n}{2n}}.
 \eg
 Since because of \rrema{R1745} the matrices \(\FFru{n}^\ad\), \(\VVlu{n}\), \(\FFru{n}\), and \(\VVru{n}\) are \tns{} with \(\set{\FFru{n}^\invad,\VVlu{n}^\inv}\subseteq\nuduu{q}{n}\) and \(\set{\FFru{n}^\inv,\VVru{n}^\inv}\subseteq\noduu{q}{n}\), we obtain from \eqref{T1016.4}, \eqref{T1016.5}, and \rrema{R1427} that \(\Spu{2k}=\su{0}^\sntaa{2k}{\alpha}\) holds true for all \(k\in\mn{0}{n-1}\) and that \(\Spu{2n}=\su{0}^\sntaa{2n}{\alpha}+\Zuu{2n}{2n}\). In the case \(n=1\), the assertion hence follows from \eqref{T1016.2} and \eqref{T1016.3}. Now we consider the case \(n\geq2\). From \eqref{Kgg2n} we obtain \(\seq{\sau{j}}{j}{0}{2(n-1)}\in\Hggqu{2(n-1)}\). Thus, in view of \rrema{R1500}, \eqref{Ha}, and \eqref{La}, there exists a matrix \(\GGru{n-1}\in\noduu{q}{n-1}\) such that \(\GGru{n-1}^\ad\Hau{n-1}\GGru{n-1}=\diaga{\Lau{0},\Lau{1},\dotsc,\Lau{n-1}}\). From \eqref{-aH+K=Ha} and \rdefi{D1021}, we get then
 \bgl{T1016.6}
  \GGru{n-1}^\ad(-\alpha\Hu{n-1}+\Ku{n})\GGru{n-1}
  =\diaga{\Spu{1},\Spu{3},\dotsc,\Spu{2n-1}}.
 \eg
 In view of \rrema{R1440}, we have \(\seq{\su{j}}{j}{0}{2n-1}\in\Kggquu{2n-1}{\alpha}\). Thus, according to \rlemm{L1306}, there exist matrices  \(\WWlu{n-1}\in\nudqu{n-1}\) and \(\WWru{n-1}\in\nodqu{n-1}\) such that
 \begin{multline}\label{T1016.7}
  \WWlu{n-1}(-\alpha\Hu{n-1}+\Ku{n-1})\WWru{n-1}\\
  =\diaga{\su{0}^\sntaa{1}{\alpha},\su{0}^\sntaa{3}{\alpha}, \dotsc,\su{0}^\sntaa{2n-3}{\alpha} ,\su{0}^\sntaa{2n-1}{\alpha}+\Zuu{2n-1}{2n-1}}.
 \end{multline}
 Since because of \rrema{R1745} the matrices \(\GGru{n-1}^\ad\), \(\WWlu{n-1}\), \(\GGru{n-1}\), and \(\WWru{n-1}\) are \tns{} with \(\set{\GGru{n-1}^\invad,\WWlu{n-1}^\inv}\subseteq\nuduu{q}{n-1}\) and \(\set{\GGru{n-1}^\inv,\WWru{n-1}^\inv}\subseteq\noduu{q}{n-1}\), we obtain from \eqref{T1016.6}, \eqref{T1016.7}, and \rrema{R1427} that \(\Spu{2k+1}=\su{0}^\sntaa{2k+1}{\alpha}\) holds true for all \(k\in\mn{0}{n-2}\) and that \(\Spu{2n-1}=\su{0}^\sntaa{2n-1}{\alpha}+\Zuu{2n-1}{2n-1}\) which, in view of \eqref{T1016.3}, completes the proof in this case.
 
 Finally, we consider the case where \(m=2n+1\) with some \(n\in\N\). Then, from \eqref{Kgg2n+1} we obtain \(\seq{\su{j}}{j}{0}{2n}\in\Hggqu{2n}\), whereas \rrema{R1440} yields \(\seq{\su{j}}{j}{0}{2n}\in\Kggqualpha{2n}\). Similar to the above mentioned case, we conclude that \(\Spu{2k}=\su{0}^\sntaa{2k}{\alpha}\) holds true for all \(k\in\mn{0}{n-1}\) and that \(\Spu{2n}=\su{0}^\sntaa{2n}{\alpha}+\Zuu{2n}{2n}\) which, in view of \eqref{T1016.3}, implies \(\Spu{2k}=\su{0}^\sntaa{2k}{\alpha}\) for all \(k\in\mn{0}{n}\). Moreover, from \eqref{Kgg2n+1} we obtain \(\seq{\sau{j}}{j}{0}{2n}\in\Hggqu{2n}\). Again similar to the above mentioned case, we conclude that \(\Spu{2k+1}=\su{0}^\sntaa{2k+1}{\alpha}\) holds true for all \(k\in\mn{0}{n-1}\) and that \(\Spu{2n+1}=\su{0}^\sntaa{2n+1}{\alpha}+\Zuu{2n+1}{2n+1}\). This completes the proof.
\eproof

\bcorol{C1500}
 Let \(\alpha\in\R\), let \(n\in\N\), and let \(\seq{\su{j}}{j}{0}{2n}\in\Kggquu{2n}{\alpha}\). For all \(k\in\mn{0}{2n}\) denote by \(\seq{\su{j}^\sntaa{k}{\alpha}}{j}{0}{2n-k}\) the \tsaalphaa{k}{\seq{\su{j}}{j}{0}{2n}}. Then
 \[
  \rank\Hu{n}
  =\lrk\sum_{k=0}^{n-1}\rank\su{0}^\sntaa{2k}{\alpha}\rrk+\rank(\su{0}^\sntaa{2n}{\alpha}+\Zuu{2n}{2n})
 \]
 and
 \[
  \det\Hu{n}
  =\lrk\prod_{k=0}^{n-1}\det\su{0}^\sntaa{2k}{\alpha}\rrk\det(\su{0}^\sntaa{2n}{\alpha}+\Zuu{2n}{2n}).
 \]
 Furthermore,
 \begin{align*}
  \rank\Hau{n-1}
  &=\sum_{k=0}^{n-1}\rank\su{0}^\sntaa{2k+1}{\alpha}&
  &\text{and}&
  \det\Hau{n-1}
  &=\prod_{k=0}^{n-1}\det\su{0}^\sntaa{2k+1}{\alpha}.
 \end{align*}
\ecoro
\bproof
 From~\zitaa{MR3014201}{\clemm{4.11(a)}} we get the equations \(\rank\Hu{n}=\sum_{k=0}^n\rank\Spu{2k}\) and \(\det\Hu{n}=\prod_{k=0}^n\det\Spu{2k}\). In view of~\zitaa{MR3014201}{\clemm{4.11(b)}}, we see furthermore that \(\rank\Hau{n-1}=\sum_{k=0}^{n-1}\rank\Spu{2k+1}\) and \(\det\Hau{n-1}=\prod_{k=0}^{n-1}\det\Spu{2k+1}\). Applying \rtheo{T1016}, we obtain the asserted equations.
\eproof

\bcorol{C1623}
 Let \(\alpha\in\R\), let \(n\in\N\), and let \(\seq{\su{j}}{j}{0}{2n+1}\in\Kggquu{2n+1}{\alpha}\). For all \(k\in\mn{0}{2n+1}\) denote by \(\seq{\su{j}^\sntaa{k}{\alpha}}{j}{0}{2n+1-k}\) the \tsaalphaa{k}{\seq{\su{j}}{j}{0}{2n+1}}. Then
 \begin{align*}
  \rank\Hu{n}&=\sum_{k=0}^n\rank\su{0}^\sntaa{2k}{\alpha}&
  &\text{and}&
  \det\Hu{n}&=\prod_{k=0}^n\det\su{0}^\sntaa{2k}{\alpha}.
 \end{align*}
 Furthermore,
 \[
  \rank\Hau{n}
  =\lrk\sum_{k=0}^{n-1}\rank\su{0}^\sntaa{2k+1}{\alpha}\rrk+\rank(\su{0}^\sntaa{2n+1}{\alpha}+\Zuu{2n+1}{2n+1})
 \]
 and
 \[
  \det\Hau{n}
  =\lrk\prod_{k=0}^{n-1}\det\su{0}^\sntaa{2k+1}{\alpha}\rrk\det(\su{0}^\sntaa{2n+1}{\alpha}+\Zuu{2n+1}{2n+1}).
 \]
\ecoro
\bproof
 From~\zitaa{MR3014201}{\clemm{4.11(a)}} we get the equations \(\rank\Hu{n}=\sum_{k=0}^n\rank\Spu{2k}\) and \(\det\Hu{n}=\prod_{k=0}^n\det\Spu{2k}\). In view of~\zitaa{MR3014201}{\clemm{4.11(b)}}, we see furthermore that \(\rank\Hau{n}=\sum_{k=0}^n\rank\Spu{2k+1}\) and \(\det\Hau{n}=\prod_{k=0}^n\det\Spu{2k+1}\). Using \rtheo{T1016}, we obtain the asserted equations.
\eproof

 Now we obtain the second main result of this paper, which indicates that \rtheo{T1016} can be simplified for the subclass \(\Kggeqkappaalpha\):
\btheol{T1615}
 Let \(\alpha\in\R\), let \(\kappa\in\NO\cup\set{+\infty}\), and let \(\seq{\su{j}}{j}{0}{\kappa}\in\Kggequu{\kappa}{\alpha}\). Then \(\seq{\su{0}^\sntaa{j}{\alpha}}{j}{0}{\kappa}\) is exactly the \trasp{\alpha} of \(\seq{\su{j}}{j}{0}{\kappa}\).
\etheo
\bproof
 Obviously, we have \(\seq{\su{j}}{j}{0}{\kappa}\in\Kggquu{\kappa}{\alpha}\). Denote by \(\seq{\Spu{j}}{j}{0}{\kappa}\) the \traspa{\alpha}{\seq{\su{j}}{j}{0}{\kappa}}.
 \baeqi{0}
  \il{T1615.C1} First we consider the case \(\kappa\in\NO\). Because of \(\seq{\su{j}}{j}{0}{\kappa}\in\Kggquu{\kappa}{\alpha}\) and \rtheo{T1016}, then \(\Spu{\kappa}=\su{0}^\sntaa{\kappa}{\alpha}+\Zuu{\kappa}{\kappa}\) and, in the case \(\kappa\geq1\), furthermore \(\Spu{j}=\su{0}^\sntaa{j}{\alpha}\) for all \(j\in\mn{0}{\kappa-1}\). In view of \(\seq{\su{j}}{j}{0}{\kappa}\in\Kggequu{\kappa}{\alpha}\) and \rlemm{L1143}, we obtain moreover \(\Zuu{\kappa}{\kappa}=\Oqq\), which completes the proof in this case.
  \il{T1615.C2} Finally, we consider the case \(\kappa=+\infty\). Let \(l\in\NO\) and let the sequence \(\seq{r_j}{j}{0}{l}\) be given by \(r_j\defg\su{j}\) for all \(j\in\mn{0}{l}\). In view of \(\seq{\su{j}}{j}{0}{\kappa}\in\Kggquu{\kappa}{\alpha}\), then \(\seq{r_j}{j}{0}{l}\in\Kggequalpha{l}\). For all \(k\in\mn{0}{l}\), denote by \(\seq{v_j^{(k)}}{j}{0}{l-k}\) the \tsaalphaa{k}{\seq{r_j}{j}{0}{l}}. By virtue of \rrema{R1501}, then \(v_0^{(l)}=\su{0}^\sntaa{l}{\alpha}\). According to the above already proved~\ref{T1615.C1}, we get, in view of \(l\in\NO\) and \(\seq{r_j}{j}{0}{l}\in\Kggequu{l}{\alpha}\), furthermore \(R_l=v_0^{(l)}\), where \(\seq{R_j}{j}{0}{l}\) denotes the \traspa{\alpha}{\seq{r_j}{j}{0}{l}}. By \rrema{R1516}, moreover \(R_l=\Spu{l}\). Hence, \(\su{0}^\sntaa{l}{\alpha}=v_0^{(l)}=R_l=\Spu{l}\) for all \(l\in\NO\). The proof is complete.\qedhere
 \eaeqi
\eproof

\bcorol{C1624}
 Let \(\alpha\in\R\), let \(\kappa\in\NO\cup\set{+\infty}\), and let \(\seq{\su{j}}{j}{0}{\kappa}\in\Kggequu{\kappa}{\alpha}\). For all \(k\in\mn{0}{\kappa}\) denote by \(\seq{\su{j}^\sntaa{k}{\alpha}}{j}{0}{\kappa-k}\) the \tsaalphaa{k}{\seq{\su{j}}{j}{0}{\kappa}}. Then \(\rank\Hu{n}=\sum_{k=0}^n\rank\su{0}^\sntaa{2k}{\alpha}\) and \(\det\Hu{n}=\prod_{k=0}^n\det\su{0}^\sntaa{2k}{\alpha}\) for all \(n\in\NO\) with \(2n\leq\kappa\) and \(\rank\Hau{n}=\sum_{k=0}^n\rank\su{0}^\sntaa{2k+1}{\alpha}\) and \(\det\Hau{n}=\prod_{k=0}^n\det\su{0}^\sntaa{2k+1}{\alpha}\) for all \(n\in\NO\) with \(2n+1\leq\kappa\).
\ecoro
\bproof
 Obviously, we have \(\seq{\su{j}}{j}{0}{\kappa}\in\Kggquu{\kappa}{\alpha}\). From~\zitaa{MR3014201}{\clemm{4.11}} we get then that \(\rank\Hu{n}=\sum_{k=0}^n\rank\Spu{2k}\) and \(\det\Hu{n}=\prod_{k=0}^n\det\Spu{2k}\) for all \(n\in\NO\) with \(2n\leq\kappa\) and, furthermore, \(\rank\Hau{n}=\sum_{k=0}^n\rank\Spu{2k+1}\) and \(\det\Hau{n}=\prod_{k=0}^n\det\Spu{2k+1}\) for all \(n\in\NO\) with \(2n+1\leq\kappa\). Using \rtheo{T1615}, we obtain the asserted equations.
\eproof

Now we characterize the membership of a sequence from \(\Kggqkappaalpha\) to the classes \(\Kggeqkappaalpha\), \(\Kgqkappaalpha\), and \(\Kggdqkappaalpha\) in terms of the sequence of its \tlasnt{\alpha}s.
\bpropl{P1507}
 Let \(\alpha\in\R\), let \(m\in\NO\), and let \(\seq{\su{j}}{j}{0}{m}\in\Kggquu{m}{\alpha}\). For all \(k\in\mn{0}{m}\) denote by \(\seq{\su{j}^\sntaa{k}{\alpha}}{j}{0}{m-k}\) the \tsaalphaa{k}{\seq{\su{j}}{j}{0}{m}}. In view of \eqref{P[k]} and \eqref{Z}, then the following statements are equivalent:
 \baeqi{0}
  \il{P1507.i} The sequence \(\seq{\su{j}}{j}{0}{m}\) belongs to \(\Kggequu{m}{\alpha}\).
  \il{P1507.ii} \(\Kerna{\su{0}^\sntaa{k}{\alpha}}\subseteq\Kerna{\su{m-k}^\sntaa{k}{\alpha}}\) for all \(k\in\mn{0}{m-1}\) in the case \(m\geq1\).
  \il{P1507.iii} \(\Puu{m-k}{k}=\Oqq\) for all \(k\in\mn{0}{m-1}\) in the case \(m\geq1\).
  \il{P1507.iv} \(\Zuu{m}{m}=\Oqq\).
 \eaeqi
\eprop
\bproof
 In the case \(m=0\) we see from \(\Kggequalpha{0}=\Kggqualpha{0}\) and \eqref{Z} that both conditions~\rstat{P1507.i} and~\rstat{P1507.iv} are fulfilled. Furthermore, the conditions~\rstat{P1507.ii} and~\rstat{P1507.iii} are empty in this case, and hence, all four conditions~\rstat{P1507.i},~\rstat{P1507.ii},~\rstat{P1507.iii}, and~\rstat{P1507.iv} are equivalent.
 
 Now we consider the case \(m\geq1\).
 
 \bimp{P1507.i}{P1507.ii}
  Let \(k\in\mn{0}{m-1}\). Because of~\rstat{P1507.i} and \rthmp{P1410}{P1410.b}, we have \(\seq{\su{j}^\sntaa{k}{\alpha}}{j}{0}{m-k}\in\Kggequu{m-k}{\alpha}\), which, in view of \rpropp{P1442}{P1442.a}, implies \(\seq{\su{j}^\sntaa{k}{\alpha}}{j}{0}{m-k}\in\Dqqu{m-k}\). Taking into account \rdefi{D1658}, we obtain then \(\Kerna{\su{0}^\sntaa{k}{\alpha}}\subseteq\Kerna{\su{m-k}^\sntaa{k}{\alpha}}\). Hence,~\rstat{P1507.ii} is fulfilled.
 \eimp
 
 \bimp{P1507.ii}{P1507.iii}
  Let \(k\in\mn{0}{m-1}\). Because of \(\seq{\su{j}}{j}{0}{m}\in\Kggquu{m}{\alpha}\) and \rthmp{P1410}{P1410.a}, we have \(\seq{\su{j}^\sntaa{k}{\alpha}}{j}{0}{m-k}\in\Kggquu{m-k}{\alpha}\), which, in view of \rlemp{L1738}{L1738.a}, implies \(\su{j}^\sntaa{k}{\alpha}\in\CHq\) for all \(j\in\mn{0}{m-k}\). Using~\rstat{P1507.ii}, we obtain then \(\Bilda{\su{m-k}^\sntaa{k}{\alpha}}\subseteq\Bilda{\su{0}^\sntaa{k}{\alpha}}\). Taking into account \eqref{P[k]} and \rpartss{R1631.c}{R1631.b} of \rrema{R1631}, we get \(\Puu{m-k}{k}=\Oqq\). Hence~\rstat{P1507.iii} is fulfilled.
 \eimp
 
 \bimp{P1507.iii}{P1507.iv}
  Use \eqref{Z}.
 \eimp
 
 \bimp{P1507.iv}{P1507.i}
  From \rtheo{T1016} and~\rstat{P1507.iv} we see that \(\Spu{j}=\su{0}^\sntaa{j}{\alpha}\) holds true for all \(j\in\mn{0}{m}\), where \(\seq{\Spu{j}}{j}{0}{m}\) denotes the \traspa{\alpha}{\seq{\su{j}}{j}{0}{m}}. Because of \rrema{R1116}, we obtain \(\Kerna{\su{0}^\sntaa{m-1}{\alpha}}\subseteq\Kerna{\su{0}^\sntaa{m}{\alpha}}\). Thus,  \(\Kerna{\Spu{m-1}}\subseteq\Kerna{\Spu{m}}\) which, in view of \(\seq{\su{j}}{j}{0}{m}\in\Kggquu{m}{\alpha}\), \rdefi{D1021}, and \rlemmss{L0912}{L0919}, implies~\rstat{P1507.i}.
 \eimp
\eproof

\bpropl{P1626}
 Let \(\alpha\in\R\), let \(\kappa\in\NO\cup\set{+\infty}\), and let \(\seq{\su{j}}{j}{0}{\kappa}\in\Kggquu{\kappa}{\alpha}\). For all \(k\in\mn{0}{\kappa}\) denote by \(\seq{\su{j}^\sntaa{k}{\alpha}}{j}{0}{\kappa-k}\) the \tsaalphaa{k}{\seq{\su{j}}{j}{0}{\kappa}}. Then the following statements are equivalent:
 \baeqi{0}
  \il{P1626.i} The sequence \(\seq{\su{j}}{j}{0}{\kappa}\) belongs to \(\Kgquu{\kappa}{\alpha}\).
  \il{P1626.ii} For all \(k\in\mn{0}{\kappa}\), the matrix \(\su{0}^\sntaa{k}{\alpha}\) is \tns{}.
 \eaeqi
 If~\rstat{P1626.i} is fulfilled, then \(\su{0}^\sntaa{k}{\alpha}\in\Cgq\) for all \(k\in\mn{0}{\kappa}\) and, furthermore, \(\Puu{j}{k}=\Oqq\) for all \(k\in\mn{0}{\kappa}\) and all \(j\in\mn{0}{\kappa-k}\) and \(\Zuu{l}{m}=\Oqq\) for all \(m\in\mn{0}{\kappa}\) and all \(l\in\mn{m}{\kappa}\).
\eprop
\bproof
 \bimp{P1626.i}{P1626.ii}
  Because of~\rstat{P1626.i} and \rpropp{P1442}{P1442.d}, we have \(\seq{\su{j}}{j}{0}{\kappa}\in\Kggequalpha{\kappa}\). Thus, \rtheo{T1615} implies \(\su{0}^\sntaa{j}{\alpha}=\Spu{j}\) for all \(j\in\mn{0}{\kappa}\), where \(\seq{\Spu{j}}{j}{0}{\kappa}\) denotes the \traspa{\alpha}{\seq{\su{j}}{j}{0}{\kappa}}. From~\rstat{P1626.i} and \rthmp{121.P1337}{121.P1337.d} we obtain furthermore \(\Spu{j}\in\Cgq\) and hence \(\det\Spu{j}\neq0\) for all \(j\in\mn{0}{\kappa}\). Thus,~\rstat{P1626.ii} is fulfilled and \(\su{0}^\sntaa{k}{\alpha}\in\Cgq\) for all \(k\in\mn{0}{\kappa}\).
 \eimp
 
 \bimp{P1626.ii}{P1626.i}
  Because of~\rstat{P1626.ii} and \eqref{P[k]}, we have \(\Puu{j}{k}=\Oqq\) for all \(k\in\mn{0}{\kappa}\) and all \(j\in\mn{0}{\kappa-k}\) which, in view of \eqref{Z}, implies \(\Zuu{l}{m}=\Oqq\) for all \(m\in\mn{0}{\kappa}\) and all \(l\in\mn{m}{\kappa}\). From \rremasss{R1610}{R1613}{R1440}, \rcoross{C1500}{C1623}, and~\rstat{P1626.ii} we obtain then \(\det\Hu{n}   =\prod_{k=0}^n\det\su{0}^\sntaa{2k}{\alpha}\neq0\) for all \(n\in\NO\) with \(2n\leq\kappa\) and \(\det(-\alpha\Hu{n}+\Ku{n})   =\prod_{k=0}^n\det\su{0}^\sntaa{2k+1}{\alpha}\neq0\) for all \(n\in\NO\) with \(2n+1\leq\kappa\), which, because of \(\seq{\su{j}}{j}{0}{\kappa}\in\Kggquu{\kappa}{\alpha}\), implies~\rstat{P1626.i}.
 \eimp
\eproof

\bpropl{P1311}
 Let \(\alpha\in\R\), let \(m\in\NO\), and let \(\seq{\su{j}}{j}{0}{m}\in\Kggquu{m}{\alpha}\). Denote by \(\seq{\su{j}^\sntaa{m}{\alpha}}{j}{0}{0}\) the \tsaalphaa{m}{\seq{\su{j}}{j}{0}{m}}. Then \(\seq{\su{j}}{j}{0}{m}\in\Kgquu{m}{\alpha}\) if and only if \(\det\su{0}^\sntaa{m}{\alpha}\neq0\).
\eprop
\bproof
 Use \rprop{P1626} and \rrema{R1116}.
\eproof
 
\bpropl{P1544}
 Let \(\alpha\in\R\), let \(m\in\NO\), and let \(\seq{\su{j}}{j}{0}{m}\in\Kggqualpha{m}\). Denote by \(\seq{\su{j}^\sntaa{m}{\alpha}}{j}{0}{0}\) the \tsaalphaa{m}{\seq{\su{j}}{j}{0}{m}}. Then the following statements are equivalent:
 \baeqi{0}
  \il{P1544.i} The sequence \(\seq{\su{j}}{j}{0}{m}\) belongs to \(\Kggdqualpha{m}\).
  \il{P1544.ii} \(\su{0}^\sntaa{m}{\alpha}+\Zuu{m}{m}=\Oqq\).
 \eaeqi
 If~\rstat{P1544.i} is fulfilled, then \(\su{0}^\sntaa{m}{\alpha}=\Oqq\), \(\Puu{j}{k}=\Oqq\) for all \(k\in\mn{0}{m}\) and all \(j\in\mn{0}{m-k}\), and \(\Zuu{l}{n}=\Oqq\) for all \(n\in\mn{0}{m}\) and all \(l\in\mn{n}{m}\).
\eprop
\bproof
 \bimp{P1544.i}{P1544.ii}
  Because of~\rstat{P1544.i} and \rpropp{P1442}{P1442.d}, we have \(\seq{\su{j}}{j}{0}{m}\in\Kggequalpha{m}\). Thus, \rlemm{L1143} yields \(\Puu{j}{k}=\Oqq\) for all \(k\in\mn{0}{m}\) and all \(j\in\mn{0}{m-k}\), and \(\Zuu{l}{n}=\Oqq\) for all \(n\in\mn{0}{m}\) and all \(l\in\mn{n}{m}\), whereas \rtheo{T1615} yields \(\su{0}^\sntaa{m}{\alpha}=\Spu{m}\), where \(\seq{\Spu{j}}{j}{0}{m}\) denotes the \traspa{\alpha}{\seq{\su{j}}{j}{0}{m}}. From~\rstat{P1544.i} and~\zitaa{MR3014201}{\cprop{5.3}} we obtain furthermore \(\Spu{m}=\Oqq\). Hence,~\rstat{P1544.ii} is fulfilled.
 \eimp
  
 \bimp{P1544.ii}{P1544.i}
  Because of \(\seq{\su{j}}{j}{0}{m}\in\Kggqualpha{m}\), the application of \rtheo{T1016} yields \(\Spu{m}=\su{0}^\sntaa{m}{\alpha}+\Zuu{m}{m}\), which, in view of~\rstat{P1544.ii}, implies \(\Spu{m}=\Oqq\). Taking into account \(\seq{\su{j}}{j}{0}{m}\in\Kggqualpha{m}\), from~\zitaa{MR3014201}{\cprop{5.3}} we obtain then~\rstat{P1544.i}.
 \eimp
\eproof

The following example shows that in the situation of \rprop{P1544} condition~\rstat{P1544.ii} cannot be weakened by replacing it by the condition \(\su{0}^\sntaa{m}{\alpha}=\Oqq\).

\begin{exa}\label{E1723}
 Let \(\alpha\in\R\), let \(\su{0}\defg\Oqq\), and let \(\su{1}\defg\Iq\). According to \eqref{s_a}, then \(\sau{0}=\Iq\). Because of \eqref{H} and \eqref{Ha}, we have then \(\Hu{0}=\su{0}=\Oqq\in\Cggq\) and \(\Hau{0}=\sau{0}=\Iq\in\Cggq\), which, in view of \eqref{Kgg2n+1}, implies \(\seq{\su{j}}{j}{0}{1}\in\Kggqualpha{1}\). By \rdefi{D1059}, furthermore \(
  \su{0}^\sntaa{1}{\alpha}
  =-\su{0}\su{1}^\reza{\alpha}\su{0}
  =-\su{0}\su{1}^\reza{\alpha}\cdot\Oqq
  =\Oqq
 \). Because of \eqref{La} and \eqref{L}, moreover \(\Lau{0}=\sau{0}=\Iq\neq\Oqq\) which, in view of \eqref{Kggcd2n+1} and \eqref{Hggcd}, implies \(\seq{\su{j}}{j}{0}{1}\notin\Kggdqualpha{1}\). Hence, \(\seq{\su{j}}{j}{0}{1}\in\Kggqualpha{1}\setminus\Kggdqualpha{1}\) although \(\su{0}^\sntaa{1}{\alpha}=\Oqq\).
\end{exa}

\bpropl{P1633}
 Let \(\alpha\in\R\), let \(\kappa\in\N\cup\set{+\infty}\), let \(m\in\mn{0}{\kappa-1}\), and let \(\seq{\su{j}}{j}{0}{\kappa}\in\Kggqkappaalpha\). For all \(k\in\mn{m}{\kappa}\) denote by \(\seq{\su{j}^\sntaa{k}{\alpha}}{j}{0}{\kappa}\) the \tsaalphaa{k}{\seq{\su{j}}{j}{0}{\kappa}}. Then the following statements are equivalent:
 \baeqi{0}
  \il{P1633.i} The sequence \(\seq{\su{j}}{j}{0}{\kappa}\) belongs to \(\Kggdoqkappaalpha{m}\).
  \il{P1633.ii} \(\su{0}^\sntaa{m}{\alpha}=\Oqq\).
 \eaeqi
 If~\rstat{P1633.i} is fulfilled, then \(\su{j}^\sntaa{k}{\alpha}=\Oqq\) for all \(k\in\mn{m+1}{\kappa}\) and all \(j\in\mn{0}{\kappa-k}\), and \(\Puu{j}{k}=\Oqq\) for all \(k\in\mn{0}{\kappa-1}\) and all \(j\in\mn{0}{\kappa-1-k}\) and, furthermore, \(\Zuu{l}{n}=\Oqq\) for all \(n\in\mn{0}{\kappa-1}\) and all \(l\in\mn{n}{\kappa-1}\).
\eprop
\bproof
 Because of \(\seq{\su{j}}{j}{0}{\kappa}\in\Kggqkappaalpha\) and \(m\in\mn{0}{\kappa-1}\), we have \(\seq{\su{j}}{j}{0}{m}\in\Kggqualpha{m}\) and, in view of \(\kappa\in\N\cup\set{+\infty}\) and \rrema{R1622}, furthermore \(\Puu{j}{k}=\Oqq\) for all \(k\in\mn{0}{\kappa-1}\) and all \(j\in\mn{0}{\kappa-1-k}\) and \(\Zuu{l}{n}=\Oqq\) for all \(n\in\mn{0}{\kappa-1}\) and all \(l\in\mn{n}{\kappa-1}\). In particular, \(\Zuu{m}{m}=\Oqq\).
 
 \bimp{P1633.i}{P1633.ii}
  From~\rstat{P1633.i} and \eqref{Kggcdm} we get \(\seq{\su{j}}{j}{0}{m}\in\Kggdqualpha{m}\). Thus, \rprop{P1544} yields~\rstat{P1633.ii}. From \rrema{R1116} we obtain furthermore \(\rank\su{j}^\sntaa{k}{\alpha}\leq\rank\su{0}^\sntaa{m}{\alpha}=0\) and thus \(\su{j}^\sntaa{k}{\alpha}=\Oqq\) for all \(k\in\mn{m+1}{\kappa}\) and all \(j\in\mn{0}{\kappa-k}\).
 \eimp
  
 \bimp{P1633.ii}{P1633.i}
  Because of~\rstat{P1633.ii} and \(\Zuu{m}{m}=\Oqq\), we have \(\su{0}^\sntaa{m}{\alpha}+\Zuu{m}{m}=\Oqq\). The application of \rprop{P1544} yields then \(\seq{\su{j}}{j}{0}{m}\in\Kggdqualpha{m}\), which, in view of \eqref{Kggcdm}, implies~\rstat{P1633.i}.
 \eimp
\eproof

 The following example shows that in the situation of \rprop{P1633} we can from condition~\rstat{P1633.i} in general not conclude that \(\su{j}^\sntaa{m}{\alpha}=\Oqq\) for all \(j\in\mn{1}{\kappa-m}\).

\begin{exa}\label{E0849}
 Let \(\alpha\in\R\), let \(\su{0}\defg\Oqq\), and let \(\su{1}\defg\Iq\). According to \rexam{E1723}, then \(\seq{\su{j}}{j}{0}{1}\in\Kggqualpha{1}\). In particular, \(\seq{\su{j}}{j}{0}{0}\in\Kggqualpha{0}\). Because of \eqref{L}, we have \(\Lu{0}=\su{0}=\Oqq\). In view of \eqref{Kggcd2n} and \eqref{Hggcd}, thus \(\seq{\su{j}}{j}{0}{0}\in\Kggdqualpha{0}\) and, according to \eqref{Kggcdm}, we see that \(\seq{\su{j}}{j}{0}{1}\in\Kggdoqualpha{0}{1}\). Taking into account \rdefi{D1632}, moreover \(\su{1}^\sntaa{0}{\alpha}=\su{1}=\Iq\). Hence, \(\su{1}^\sntaa{0}{\alpha}\neq\Oqq\) although \(\seq{\su{j}}{j}{0}{1}\in\Kggdoqualpha{0}{1}\).
\end{exa}

\bpropl{P1541}
 Let \(\alpha\in\R\) and let \(\seq{\su{j}}{j}{0}{\infty}\in\Kggqinfalpha\). For all \(k\in\NO\), denote by \(\seq{\su{j}^\sntaa{k}{\alpha}}{j}{0}{\infty}\) the \tsaalphaa{k}{\seq{\su{j}}{j}{0}{\infty}}. Then the following statements are equivalent:
 \baeqi{0}
  \il{P1541.i} The sequence \(\seq{\su{j}}{j}{0}{\infty}\) belongs to \(\Kggdqinfalpha\).
  \il{P1541.ii} There exists an \(m\in\NO\) such that \(\su{0}^\sntaa{m}{\alpha}=\Oqq\).
 \eaeqi
 If~\rstat{P1541.ii} is fulfilled and if \(m\in\NO\) is such that \(\su{0}^\sntaa{m}{\alpha}=\Oqq\), then \(\su{j}^\sntaa{k}{\alpha}=\Oqq\) for all \(j\in\NO\) and all \(k\in\mn{m}{\infty}\), \(\Puu{j}{k}=\Oqq\) for all \(j,k\in\NO\), and \(\Zuu{l}{n}=\Oqq\) for all \(n\in\NO\) and all \(l\in\mn{n}{\infty}\).
\eprop
\bproof
 The equivalence of~\rstat{P1541.i} and~\rstat{P1541.ii} is a consequence of \eqref{G1023} and \rprop{P1633}. Now suppose that~\rstat{P1541.ii} is fulfilled and let \(m\in\NO\) be such that \(\su{0}^\sntaa{m}{\alpha}=\Oqq\). \rthmp{P1410}{P1410.a} yields then \(\seq{\su{j}^\sntaa{m}{\alpha}}{j}{0}{\infty}\in\Kggqinfalpha\), which, in view of \rlemp{L1738}{L1738.c}, implies \(\Bilda{\su{j}^\sntaa{m}{\alpha}}\subseteq\Bilda{\su{0}^\sntaa{m}{\alpha}}=\set{\Ouu{q}{1}}\) and hence \(\su{j}^\sntaa{m}{\alpha}=\Oqq\) for all \(j\in\NO\). From \rrema{R1116} we obtain furthermore \(\rank\su{j}^\sntaa{k}{\alpha}\leq\rank\su{0}^\sntaa{m}{\alpha}=0\) and thus \(\su{j}^\sntaa{k}{\alpha}=\Oqq\) for all \(k\in\mn{m+1}{\infty}\) and all \(j\in\NO\). Because of \(\seq{\su{j}}{j}{0}{\infty}\in\Kggqinfalpha=\Kggeqinfalpha\) and \rlemm{L1143}, we have \(\Puu{j}{k}=\Oqq\) for all \(j,k\in\NO\) and \(\Zuu{l}{n}=\Oqq\) for all \(n\in\NO\) and all \(l\in\mn{n}{\infty}\), which completes the proof.
\eproof

 The following two \rthmss{T1000}{T1658} are further main results of this paper. They describe the connection between the \trasp{\alpha}s of a sequence from \(\Kggquu{m}{\alpha}\) and \(\Kggeqkappaalpha\) and its \taaSSt{k}{\alpha}.
\begin{thm}\label{T1000}
 Let \(\alpha\in\R\), let \(m\in\NO\), let \(\seq{\su{j}}{j}{0}{ m}\in\Kggquu{ m}{\alpha}\), and let \(k\in\mn{0}{ m}\). Denote by \(\seq{\Spu{j}}{j}{0}{ m}\) and \((T_j)_{j=0}^{ m-k}\) the \trasp{\alpha}s of \(\seq{\su{j}}{j}{0}{m}\) and \((t_j)_{j=0}^{ m-k}\), respectively, where \((t_j)_{j=0}^{ m-k}\) denotes the \tsaalphaa{k}{\seq{\su{j}}{j}{0}{ m}}. Then
 \[
  T_{m-k}
  =
  \begin{cases}
   \Spu{m}\incase{k=0}\\
   \Spu{m}-\sum_{r=0}^{k-1}\Puu{m-r}{r}\incase{k\geq1}
  \end{cases}
 \]
 and, in the case \(k<m\), furthermore \(T_j=\Spu{k+j}\) for all \(j\in\mn{0}{m-k-1}\).
\end{thm}
\begin{proof}
 From \rthmp{P1410}{P1410.a} we obtain \(\seq{t_j}{j}{0}{ m-k}\in\Kggqualpha{ m-k}\). For all \(l\in\mn{0}{m-k}\), denote by \(\seq{t_j^\sntaa{l}{\alpha}}{j}{0}{ m-k-l}\) the \tsaalphaa{l}{\seq{t_j}{j}{0}{ m-k}}. According to \rrema{R1510}, we have \(\su{j}^\sntaa{k+l}{\alpha}=t_j^\sntaa{l}{\alpha}\) for all \(l\in\mn{0}{m-k}\) and all \(j\in\mn{0}{m-k-l}\). The application of \rtheo{T1016} yields \(\Spu{m}=\su{0}^\sntaa{m}{\alpha}+\Zuu{m}{m}\). 
 
 We first consider the case \(k=m\). The application of \rtheo{T1016} to the sequence \(\seq{t_j}{j}{0}{ m-k}\) yields, in view of \(m-k=0\) and \eqref{Z}, then \(T_{m-k}=t_0^\sntaa{m-k}{\alpha}+\Oqq\). Hence, \(
  T_{m-k}
  =t_0^\sntaa{m-k}{\alpha}
  =\su{0}^\sntaa{m}{\alpha}
  =\Spu{m}-\Zuu{m}{m}
  =\Spu{m}-\Zuu{m}{k}
 \). Taking into account \eqref{Z}, this implies \(T_{m-k}=\Spu{m}\) in the case \(k=0\) and \(T_{m-k}=\Spu{m}-\sum_{r=0}^{k-1}\Puu{m-r}{r}\) in the case \(k\geq1\) which completes the proof in the case \(k=m\).
 
 Now suppose \(k<m\). Then \(m\geq1\) and \(m-k\geq1\). The application of \rtheo{T1016} to the sequence \(\seq{t_j}{j}{0}{ m-k}\) yields, in view of \eqref{Z} and \eqref{P[k]}, thus
 \[
  T_{m-k}
  =t_0^\sntaa{m-k}{\alpha}+\sum_{l=0}^{m-k-1}\lek t_{m-k-l}^\sntaa{l}{\alpha}- t_{0}^\sntaa{l}{\alpha}( t_{0}^\sntaa{l}{\alpha})^\MP  t_{m-k-l}^\sntaa{l}{\alpha}( t_{0}^\sntaa{l}{\alpha})^\MP  t_{0}^\sntaa{l}{\alpha}\rek
 \]
 and \(T_j=t_0^\sntaa{j}{\alpha}\) for all \(j\in\mn{0}{m-k-1}\). Furthermore, \(m\geq1\) and \rtheo{T1016} yield \(\Spu{j}=\su{0}^\sntaa{j}{\alpha}\) for all \(j\in\mn{0}{m-1}\). Taking into account \eqref{P[k]} and \eqref{Z}, we have thus
 \[
  \begin{split}
   T_{m-k}
   &=\su{0}^\sntaa{m}{\alpha}+\sum_{l=0}^{m-k-1}\lek \su{m-k-l}^\sntaa{k+l}{\alpha}- \su{0}^\sntaa{k+l}{\alpha}( \su{0}^\sntaa{k+l}{\alpha})^\MP  \su{m-k-l}^\sntaa{k+l}{\alpha}( \su{0}^\sntaa{k+l}{\alpha})^\MP  \su{0}^\sntaa{k+l}{\alpha}\rek\\
   &=\Spu{m}-\Zuu{m}{m}+\sum_{l=0}^{m-k-1}\Puu{m-k-l}{k+l}
   =\Spu{m}-\sum_{r=0}^{m-1}\Puu{m-r}{r}+\sum_{r=k}^{m-1}\Puu{m-r}{r}
  \end{split}
 \]
 and \(T_j=\su{0}^\sntaa{k+j}{\alpha}=\Spu{k+j}\) for all \(j\in\mn{0}{ m-k-1}\). In particular, \(T_{m-k}=\Spu{m}\) in the case \(k=0\) and \(T_{m-k}=\Spu{m}-\sum_{r=0}^{k-1}\Puu{m-r}{r}\) in the case \(k\geq1\), which completes the proof.
\end{proof}

Now we consider for a sequence \(\seq{\su{j}}{j}{0}{\kappa}\in\Kggequu{\kappa}{\alpha}\) the same task as it was done in \rtheo{T1000} for the case of \taSnnd{\alpha} sequences.
\begin{thm}\label{T1658}
 Let \(\alpha\in\R\), let \(\kappa\in\NO\cup\set{+\infty}\), let \(\seq{\su{j}}{j}{0}{\kappa}\in\Kggequu{\kappa}{\alpha}\), and let \(k\in\mn{0}{\kappa}\). Denote by \(\seq{t_j}{j}{0}{\kappa-k}\) the \tsaalphaa{k}{\seq{\su{j}}{j}{0}{\kappa}} and by \(\seq{\Spu{j}}{j}{0}{\kappa}\) the \traspa{\alpha}{\seq{\su{j}}{j}{0}{\kappa}}. Then \(\seq{\Spu{k+j}}{j}{0}{\kappa-k}\) is exactly the \traspa{\alpha}{\seq{t_j}{j}{0}{\kappa-k}}.
\end{thm}
\begin{proof}
 Let \(l\in\mn{0}{\kappa-k}\). Then \(k+l\in\mn{0}{\kappa}\), and the application of \rtheo{T1615} yields \(\su{0}^\sntaa{k+l}{\alpha}=\Spu{k+l}\). According to \rrema{R1510}, we have \(\su{0}^\sntaa{k+l}{\alpha}=t_0^\sntaa{l}{\alpha}\), where \(\seq{t_j^\sntaa{l}{\alpha}}{j}{0}{\kappa-k-l}\) denotes the \tsaalphaa{l}{\seq{t_j}{j}{0}{\kappa-k}}. From \rthmp{P1410}{P1410.b} we obtain \(\seq{t_j}{j}{0}{\kappa-k}\in\Kggequalpha{\kappa-k}\). The application of \rtheo{T1615} to the sequence \(\seq{t_j}{j}{0}{\kappa-k}\) yields then \(t_0^\sntaa{l}{\alpha}=T_l\), where \((T_j)_{j=0}^{\kappa-k}\) is the \traspa{\alpha}{(t_j)_{j=0}^{\kappa-k}}. Thus, we have finally \(
  \Spu{k+j}
  =\su{0}^\sntaa{k+j}{\alpha}
  =t_0^\sntaa{j}{\alpha}
  =T_j
 \) for all \(j\in\mn{0}{\kappa-k}\), which completes the proof.
\end{proof}
\rthmss{T1000}{T1658} indicate that against to the background of our \tSchur{}-type algorithm the \trasp{\alpha} can be interpreted as a \tSchur{}-type parametrization.

\section{Recovering the original sequence from its first \hascht{\alpha} and its first matrix}\label{S1009}
The considerations in \rsect{S1649} suggest the study of a natural inverse problem associated with the \tseinsalpha{} of a (finite or infinite) sequence of complex \tpqa{matrices}. The main theme of this section is the treatment of this inverse problem, which will be explained below in more detail. It should be mentioned that a similar task was treated in~\zitaa{MR3014199}{\cSect{10}} against to the background of the \tSchur{}-type algorithm studied there. The first study of inverse problems of this kind goes back to the papers~\zita{MR2276733MR2302055}, where the inverse problem associated with the \tSchur{}-Potapov algorithm for strict \tpqa{\tSchur{}} sequences was handled.

Let \(\alpha\in\C\) and \(\kappa\in\N\cup\set{+\infty}\). For each sequence \(\seq{\su{j}}{j}{0}{\kappa}\) of complex \tpqa{matrices} the \tseinsalpha{} \(\seq{\su{j}^\sntaa{1}{\alpha}}{j}{0}{\kappa-1}\) is given by \rdefi{D1059}. Conversely, we consider the question: If the \tseinsalpha{} \(\seq{\su{j}^\sntaa{1}{\alpha}}{j}{0}{\kappa-1}\) and the matrix \(\su{0}\) are known, how one can recover the original sequence \(\seq{\su{j}}{j}{0}{\kappa}\). If \(m\in\N\) and if \(\seq{\su{j}}{j}{0}{m}\) belongs to \(\Kggqualpha{m}\) then \rtheo{T1016}, \rdefi{D1632}, and \rrema{R0929} yield such a possibility to express \(\seq{\su{j}}{j}{0}{m}\) by \(\seq{\su{j}^\sntaa{1}{\alpha}}{j}{0}{m-1}\) and \(\su{0}\). In view of the definition of the \trasp{\alpha} and the formulas \eqref{L}, \eqref{The}, \eqref{Thea}, and \eqref{La}, this way of computation is not very comfortable. That's why it seems to be more advantageous to construct a recursive procedure to recover the original sequence \(\seq{\su{j}}{j}{0}{m}\) from its \tseinsalpha{} and the matrix \(\su{0}\). To realize this aim, first we introduce the central construction of this section.
\bdefil{D1712}
 Let \(\alpha\in\C\), let \(\kappa\in\NO\cup\set{+\infty}\), let \(\seq{t_j}{j}{0}{\kappa}\) be a sequence of complex \tpqa{matrices}, and let \(A\) be a complex \tpqa{matrix}. The sequence \(\seq{t_j^\sminuseinsalphaa{A}}{j}{0}{\kappa+1}\)\index{$\seq{t_j^\sminuseinsalphaa{A}}{j}{0}{\kappa+1}$} recursively defined by
 \begin{align*}
  t_0^\sminuseinsalphaa{A}&\defg A&
 &\text{and}&
  t_j^\sminuseinsalphaa{A}&\defg\alpha^jA+\sum_{l=1}^j\alpha^{j-l}AA^\MP\lek\sum_{k=0}^{l-1}t_{l-k-1}A^\MP(t_k^\sminuseinsalphaa{A})^\splusalpha\rek
 \end{align*}
 for all \(j\in\mn{1}{\kappa+1}\) is called the \emph{\tsminuseinsalphaaa{\seq{t_j}{j}{0}{\kappa}}{A}}.
\edefi

\bremal{R1504}
 Let \(\alpha\in\C\), let \(\kappa\in\NO\cup\set{+\infty}\), let \(\seq{t_j}{j}{0}{\kappa}\) be a sequence of complex \tpqa{matrices}, and let \(A\) be a complex \tpqa{matrix}. Denote by \(\seq{\su{j}}{j}{0}{\kappa+1}\) the \tsminuseinsalphaaa{\seq{t_j}{j}{0}{\kappa}}{A}. In view of \rdefi{D1712}, one can easily see that, for all \(m\in\mn{0}{\kappa}\), the sequence \(\seq{\su{j}}{j}{0}{m+1}\) depends only on the matrices \(A\) and \(t_0,t_1,\dotsc,t_m\) and is hence exactly the \tsminuseinsalphaaa{\seq{t_j}{j}{0}{m}}{A}.
\erema

The following observation expresses an essential feature of our construction.
\bremal{R1514}
 Let \(\alpha\in\C\), let \(\kappa\in\NO\cup\set{+\infty}\), let \(\seq{t_j}{j}{0}{\kappa}\) be a sequence of complex \tpqa{matrices}, and let \(A\) be a complex \tpqa{matrix}. Denote by \(\seq{\su{j}}{j}{0}{\kappa+1}\) the \tsminuseinsalphaaa{\seq{t_j}{j}{0}{\kappa}}{A}. From \rdefiss{D1712}{D1658} we easily see then that \(\seq{\su{j}}{j}{0}{\kappa+1}\) belongs to \(\Dpqu{\kappa+1}\).
\erema

\begin{lem}\label{L0831}
 Let \(\alpha\in\C\), let \(\kappa\in\NO\cup\set{+\infty}\), let \(\seq{t_j}{j}{0}{\kappa}\) be a sequence of complex \tpqa{matrices}, and let \(A\) be a complex \tpqa{matrix}. Denote by \(\seq{\su{j}}{j}{0}{\kappa+1}\) the \tsminuseinsalphaaa{\seq{t_j}{j}{0}{\kappa}}{A} and by \(\seq{\su{j}^\splusalpha}{j}{0}{\kappa+1}\) the \tsplusalphata{\seq{\su{j}}{j}{0}{\kappa+1}}. Then \(\su{0}=A\) and \(s_j=\alpha\su{j-1}+AA^\MP\sum_{k=0}^{j-1}t_{j-1-k}A^\MP\su{k}^\splusalpha\) for all \(j\in\mn{1}{\kappa+1}\).
\end{lem}
\begin{proof}
 Using \rdefi{D1712}, we obtain \(\su{0}=A\) and \(\su{1}=\alpha A+AA^\MP t_0A^\MP\su{0}^\splusalpha=\alpha\su{0}+AA^\MP t_0A^\MP\su{0}^\splusalpha\) and, in the case \(\kappa\geq1\), for all \(j\in\mn{2}{\kappa+1}\), furthermore
 \[
  \begin{split}
   &\alpha\su{j-1}+AA^\MP\sum_{k=0}^{j-1}t_{j-1-k}A^\MP\su{k}^\splusalpha\\
   &=\alpha\lek\alpha^{j-1}A+\sum_{l=1}^{j-1}\alpha^{j-1-l}AA^\MP\lrk\sum_{k=0}^{l-1}t_{l-k-1}A^\MP\su{k}^\splusalpha\rrk\rek+AA^\MP\sum_{k=0}^{j-1}t_{j-1-k}A^\MP\su{k}^\splusalpha\\
   &=\alpha^jA+\sum_{l=1}^{j-1}\alpha^{j-l}AA^\MP\lrk\sum_{k=0}^{l-1}t_{l-k-1}A^\MP\su{k}^\splusalpha\rrk+AA^\MP\sum_{k=0}^{j-1}t_{j-k-1}A^\MP\su{k}^\splusalpha\\
   &=\alpha^jA+\sum_{l=1}^j\alpha^{j-l}AA^\MP\lrk\sum_{k=0}^{l-1}t_{l-k-1}A^\MP\su{k}^\splusalpha\rrk
   =s_j,
  \end{split}
 \]
 which completes the proof.
\end{proof}

For a given number \(\alpha\in\C\), a given \(\kappa\in\NO\cup\set{+\infty}\), and a given sequence \(\seq{t_j}{j}{0}{\kappa}\) of complex \tpqa{matrices}, we want to determine a complex \tpqa{matrix} \(A\) such that the sequence \(\seq{t_j}{j}{0}{\kappa}\) turns out to be exactly the \tseinsalpha{} of the \tsminuseinsalpha{} \(\seq{t_j^\sminuseinsalphaa{A}}{j}{0}{\kappa+1}\) corresponding to \([\seq{t_j}{j}{0}{\kappa},A]\). To realize this aim, we still need a little preparation.

\blemml{L1538}
 Let \(\alpha\in\C\), let \(\kappa\in\NO\cup\set{+\infty}\), let \(\seq{t_j}{j}{0}{\kappa}\) be a sequence of complex \tpqa{matrices}, and let \(A\) be a complex \tpqa{matrix}. Denote by \(\seq{\su{j}}{j}{0}{\kappa+1}\) the \tsminuseinsalphaaa{\seq{t_j}{j}{0}{\kappa}}{A} and let the sequence \(\seq{r_j}{j}{0}{\kappa+1}\) be given by \(r_j\defg\su{j}^\srautealpha\) for all \(j\in\mn{0}{\kappa+1}\). For all \(j\in\mn{0}{\kappa+1}\), then
 \begin{align*}
  r_j&=
  \begin{cases}
   A^\MP\incase{j=0}\\
   -A^\MP t_{j-1}A^\MP\incase{j\geq1}
  \end{cases}&
  &\text{and}&
  r_j^\rez&=\su{j}^\splusalpha.
 \end{align*}
\elemm
\bproof
 Let the sequence \(\seq{u_j}{j}{0}{\kappa+1}\) be given by \(u_0\defg A^\MP\) and \(u_j\defg-A^\MP t_{j-1}A^\MP\) for all \(j\in\mn{1}{\kappa+1}\) and let the sequence \(\seq{v_j}{j}{0}{\kappa+1}\) be given by \(v_j\defg\su{j}^\splusalpha\) for all \(j\in\mn{0}{\kappa+1}\). Using \eqref{[+]_0}, \rdefi{D1712}, \rremp{R1631}{R1631.a}, and \rdefi{D1430}, we get then \(
  v_0
  =\su{0}
  =A
  =(A^\MP)^\MP
  =u_0^\MP
  =u_0^\rez
 \) and, taking additionally this and \rdefi{D1455} into account, furthermore
 \bsp
  v_1
  &=-\alpha \su{0}+\su{1}
  =-\alpha A+\alpha^1A+\sum_{l=1}^1\alpha^{1-l}AA^\MP\lrk\sum_{k=0}^{l-1}t_{l-k-1}A^\MP v_k\rrk\\
  &=-(A^\MP)^\MP(-A^\MP t_{1-1}A^\MP)v_0
  =-u_0^\MP u_1u_0^\rez
  =u_1^\rez.
 \esp
 
 Consequently, if \(\kappa=0\), for all \(j\in\mn{0}{\kappa+1}\), then
 \bgl{L1538.1}
  v_j
  =u_j^\sraute.
 \eg
 
 Let us consider the case \(\kappa\geq1\). Then we have already proved that there exists a number \(m\in\mn{1}{\kappa}\) such that \eqref{L1538.1} is satisfied for all \(j\in\mn{0}{m}\). Because of \rdefiss{D1455}{D1712}, \eqref{L1538.1}, \rremp{R1631}{R1631.a}, and \rdefi{D1430}, we conclude that
 \bsp
  v_{m+1}
  &=-\alpha \su{m}+\su{m+1}\\
  &=-\alpha\lek\alpha^mA+\sum_{l=1}^m\alpha^{m-l}AA^\MP\lrk\sum_{k=0}^{l-1}t_{l-k-1}A^\MP v_k\rrk\rek\\
  &\qquad+\lek\alpha^{m+1}A+\sum_{l=1}^{m+1}\alpha^{(m+1)-l}AA^\MP\lrk\sum_{k=0}^{l-1}t_{l-k-1}A^\MP v_k\rrk\rek\\
  &=-\sum_{l=1}^m\alpha^{m-l+1}AA^\MP\lrk\sum_{k=0}^{l-1}t_{l-k-1}A^\MP u_k^\rez\rrk+\sum_{l=1}^{m+1}\alpha^{m-l+1}AA^\MP\lrk\sum_{k=0}^{l-1}t_{l-k-1}A^\MP u_k^\rez\rrk\\
  &=AA^\MP\sum_{k=0}^mt_{m-k}A^\MP u_k^\rez
  =-(A^\MP)^\MP\sum_{k=0}^m(-A^\MP t_{m-k}A^\MP)u_k^\rez
  =-u_0^\MP\sum_{k=0}^mu_{m-k+1}u_k^\rez
  =u_{m+1}^\rez.
 \esp
 Hence, \eqref{L1538.1} is proved inductively for all \(j\in\mn{0}{\kappa+1}\). In view of \rdefi{D1658}, the sequence \(\seq{u_j}{j}{0}{\kappa+1}\) obviously belongs to \(\Dpqu{\kappa+1}\). Taking additionally into account \eqref{L1538.1}, then~\zitaa{MR3014197}{\ccoro{4.22}} yields \(v_j^\rez=u_j\) for all \(j\in\mn{0}{\kappa+1}\). In view of \eqref{reza}, we have thus \(
  r_j
  =\su{j}^\rezalpha
  =v_j^\rez
  =u_j
 \) for all \(j\in\mn{0}{\kappa+1}\). Using this and \eqref{L1538.1} then \(
  r_j^\rez
  =u_j^\rez
  =v_j
  =\su{j}^\splusalpha
 \) follows for all \(j\in\mn{0}{\kappa+1}\).
\eproof

\blemml{L1455}
 Let \(\alpha\in\C\), let \(\kappa\in\NO\cup\set{+\infty}\), let \(\seq{t_j}{j}{0}{\kappa}\) be a sequence of complex \tpqa{matrices}, and let \(A\) be a complex \tpqa{matrix}. Denote by \(\seq{\su{j}}{j}{0}{\kappa+1}\) the \tsminuseinsalphaaa{\seq{t_j}{j}{0}{\kappa}}{A} and by \(\seq{\su{j}^\seinsalpha}{j}{0}{\kappa}\) the \tseinsalphaa{\seq{\su{j}}{j}{0}{\kappa+1}}. Then:
 \benui
  \il{L1455.a} \(\su{j}^\seinsalpha=AA^\MP t_jA^\MP A\) for all \(j\in\mn{0}{\kappa}\).
  \il{L1455.b} If \(\Kerna{A}\subseteq\bigcap_{j=0}^\kappa\Kerna{t_j}\) and \(\bigcup_{j=0}^\kappa\Bilda{t_j}\subseteq\Bilda{A}\), then \(\su{j}^\seinsalpha=t_j\) for all \(j\in\mn{0}{\kappa}\).
 \eenui
\elemm
\bproof
 Using \rdefiss{D1059}{D1712} and \rlemm{L1538}, we obtain
 \[
  \su{j}^\seinsalpha
  =-\su{0}\su{j+1}^\srautealpha\su{0}
  =-A(-A^\MP t_{j+1-1}A^\MP)A
  =AA^\MP t_jA^\MP A
 \]
 for all \(j\in\mn{0}{\kappa}\). Thus \rpart{L1455.a} is proved. \rPart{L1455.b} is an immediate consequence of~\eqref{L1455.a} and \rpartss{R1631.c}{R1631.b} of \rrema{R1631}.
\eproof

\bremal{R1324}
 Let \(\alpha\in\C\), let \(\kappa\in\NO\cup\set{+\infty}\), let \(\seq{t_j}{j}{0}{\kappa}\in\Dpqkappa\), and let \(A\in\Cpq\) be such that \(\Kerna{A}\subseteq\Kerna{t_0}\) and \(\Bilda{t_0}\subseteq\Bilda{A}\). In view of \rdefi{D1658}, then \(\Kerna{A}\subseteq\bigcap_{j=0}^\kappa\Kerna{t_j}\) and \(\bigcup_{j=0}^\kappa\Bilda{t_j}\subseteq\Bilda{A}\).
\erema

\bpropl{P1313}
 Let \(\alpha\in\C\), let \(\kappa\in\NO\cup\set{+\infty}\), let \(\seq{t_j}{j}{0}{\kappa}\in\Dpqkappa\), and let \(A\in\Cpq\) be such that \(\Kerna{A}\subseteq\Kerna{t_0}\) and \(\Bilda{t_0}\subseteq\Bilda{A}\). Denote by \(\seq{\su{j}}{j}{0}{\kappa+1}\) the \tsminuseinsalphaaa{\seq{t_j}{j}{0}{\kappa}}{A}. Then \(\seq{t_j}{j}{0}{\kappa}\) is exactly the \tseinsalphaa{\seq{\su{j}}{j}{0}{\kappa+1}}.
\eprop
\bproof
 Use \rrema{R1324} and \rlemp{L1455}{L1455.b}.
\eproof

Our next considerations can be sketched as follows. Let \(\alpha\in\C\), \(\kappa\in\NO\cup\set{+\infty}\), and \(\seq{\su{j}}{j}{0}{\kappa+1}\in\Dpqu{\kappa+1}\). Then we want to recover \(\seq{\su{j}}{j}{0}{\kappa+1}\) from its \tseinsalpha{} and its first matrix \(\su{0}\). For this reason, we still need a little preparation.

\blemml{L1041}
 Let \(\alpha\in\C\), let \(\kappa\in\NO\cup\set{+\infty}\), and let \(\seq{\su{j}}{j}{0}{\kappa+1}\) be a sequence of complex \tpqa{matrices}. Denote by \(\seq{t_j}{j}{0}{\kappa}\) the \tseinsalphaa{\seq{\su{j}}{j}{0}{\kappa+1}} and by \(\seq{w_j}{j}{0}{\kappa+1}\) the \tsminuseinsalphaaa{\seq{t_j}{j}{0}{\kappa}}{\su{0}}. Then \(w_j=\su{0}\su{0}^\MP\su{j}\su{0}^\MP\su{0}\) for all \(j\in\mn{0}{\kappa+1}\).
\elemm
\bproof
 Let the sequence \(\seq{r_j}{j}{0}{\kappa+1}\) be given by \(r_j\defg w_j^\srautealpha\) for all \(j\in\mn{0}{\kappa+1}\). \rlemm{L1538} yields then \(r_0=\su{0}^\MP\) and \(r_j=-\su{0}^\MP t_{j-1}\su{0}^\MP\) for all \(j\in\mn{1}{\kappa+1}\). Using \rlemm{L1445} and \rdefi{D1430} we obtain \(\su{0}^\srautealpha=\su{0}^\MP\).  For all \(j\in\mn{1}{\kappa+1}\) we get from \rdefi{D1059} furthermore \(-\su{0}^\MP t_{j-1}\su{0}^\MP=\su{0}^\MP\su{0}\su{j}^\srautealpha\su{0}\su{0}^\MP\). According to \rlemm{L1352}, the sequence \(\seq{\su{j}^\rezalpha}{j}{0}{\kappa+1}\) belongs to \(\Dqpu{\kappa+1}\). Taking additionally into account \(\su{0}^\srautealpha=\su{0}^\MP\) and \rdefi{D1658}, we obtain \(\bigcup_{j=0}^{\kappa+1}\Bilda{\su{j}^\srautealpha}\subseteq\Bilda{\su{0}^\MP}\)\ and \(\Kerna{\su{0}^\MP}\subseteq\bigcap_{j=0}^{\kappa+1}\Kerna{\su{j}^\srautealpha}\) which, in view of \rpartsss{R1631.a}{R1631.c}{R1631.b} of \rrema{R1631}, implies \(\su{0}^\MP\su{0}\su{j}^\srautealpha\su{0}\su{0}^\MP=\su{j}^\srautealpha\) for all \(j\in\mn{1}{\kappa+1}\). We have, for all \(j\in\mn{0}{\kappa+1}\), thus \(r_j=\su{j}^\srautealpha\) and hence \(w_j^\srautealpha=\su{j}^\srautealpha\). The application of \rlemm{L1454} then yields \(w_0w_0^\MP w_jw_0^\MP w_0=\su{0}\su{0}^\MP\su{j}\su{0}^\MP\su{0}\) for all \(j\in\mn{0}{\kappa+1}\). According to \rrema{R1514}, the sequence \(\seq{w_j}{j}{0}{\kappa+1}\) belongs to \(\Dpqu{\kappa+1}\). Taking additionally into account \rdefi{D1658} and \rpartss{R1631.c}{R1631.b} of \rrema{R1631}, we get consequently \(w_0w_0^\MP w_jw_0^\MP w_0=w_j\) for all \(j\in\mn{0}{\kappa+1}\), which completes the proof.
\eproof

\bpropl{P0930}
 Let \(\alpha\in\C\), let \(\kappa\in\NO\cup\set{+\infty}\), and let \(\seq{\su{j}}{j}{0}{\kappa+1}\in\Dpqu{\kappa+1}\). Denote by \(\seq{t_j}{j}{0}{\kappa}\) the \tseinsalphaa{\seq{\su{j}}{j}{0}{\kappa+1}}. Then \(\seq{\su{j}}{j}{0}{\kappa+1}\) is exactly the \tsminuseinsalphaaa{\seq{t_j}{j}{0}{\kappa}}{\su{0}}.
\eprop
\bproof
 Use \rlemm{L1041}, \rdefi{D1658}, and \rpartss{R1631.c}{R1631.b} of \rrema{R1631}.
\eproof
\rpropss{P1313}{P0930} indicate the particular role of the class \(\Dpqkappa\) of \tftd{} sequences of complex \tpqa{matrices} (see \rdefi{D1658}) in the context of this section. Roughly speaking, the aim of our following considerations can be described as follows: Let \(\alpha\in\R\), let \(\kappa\in\NO\cup\set{+\infty}\), and let \(\seq{t_j}{j}{0}{\kappa}\) be a sequence which belongs to \(\Kggqkappaalpha\) or to one of its distinguished subclasses \(\Kggeqkappaalpha\), \(\Kgqkappaalpha\), and \(\Kggdqkappaalpha\). Then we are looking for matrices \(A\in\Cqq\) such that the \tsminuseinsalphaaa{\seq{t_j}{j}{0}{\kappa}}{A} belongs to \(\Kggqualpha{\kappa+1}\), \(\Kggequalpha{\kappa+1}\), \(\Kgqualpha{\kappa+1}\), and \(\Kggdqualpha{\kappa+1}\), respectively. First we derive some formulas which express useful interrelations between essential block \tHankel{} matrices occurring in our considerations.

\bpropl{P1052}
 Let \(\alpha\in\C\), let \(\kappa\in\NO\cup\set{+\infty}\), let \(\seq{t_j}{j}{0}{\kappa}\) be a sequence of complex \tpqa{matrices}, and let \(A\) be a complex \tpqa{matrix}. Denote by \(\seq{\su{j}}{j}{0}{\kappa+1}\) the \tsminuseinsalphaaa{\seq{t_j}{j}{0}{\kappa}}{A}. Then:
 \benui
  \il{P1052.a} \(\Hu{0}=A\).
  \il{P1052.b} Let \(n\in\N\) with \(2n-1\leq\kappa\). Then the matrices \(\Dlu{n}\) and \(\Dru{n}\) given via \eqref{Dl} and \eqref{Dr} are invertible and
  \bgl{P1052.A}
   \Hu{n}
   =\Dlu{n}^\inv\lrk\diagA{A,\lek\Iu{n}\kp(AA^\MP)\rek(-\alpha\Huo{n-1}{t}+\Kuo{n-1}{t})\lek\Iu{n}\kp(A^\MP A)\rek}\rrk\Dru{n}^\inv.
  \eg
  In particular,
  \bgl{P1052.C}
   \rank\Hu{n}
   =\rank(A)+\rank\lrk\lek\Iu{n}\kp(AA^\MP)\rek(-\alpha\Huo{n-1}{t}+\Kuo{n-1}{t})\lek\Iu{n}\kp(A^\MP A)\rek\rrk
  \eg
  and, in the case \(p=q\), furthermore \(\det\Hu{n}=\det(A)\det(-\alpha\Huo{n-1}{t}+\Kuo{n-1}{t})\).
  \il{P1052.c} Let \(n\in\NO\) with \(2n\leq\kappa\). Then the matrices \(\Dlu{n}^\splusalpha\) and \(\Dru{n}^\splusalpha\) given via \eqref{D[+]} are invertible and
  \bgl{P1052.B}
   \Hau{n}
   =(\Dlu{n}^\splusalpha)^\inv\lek\Iu{n+1}\kp(AA^\MP)\rek\Huo{n}{t}\lek\Iu{n+1}\kp(A^\MP A)\rek(\Dru{n}^\splusalpha)^\inv.
  \eg
  In particular,
  \bgl{P1052.D}
   \rank\Hau{n}
   =\rank\lrk\lek\Iu{n+1}\kp(AA^\MP)\rek\Huo{n}{t}\lek\Iu{n+1}\kp(A^\MP A)\rek\rrk
  \eg
  and, in the case \(p=q\), furthermore \(\det\Hau{n}=(\det A)(\det A)^\MP\det\Huo{n}{t}\).
 \eenui
\eprop
\bproof
 First observe that \rdefi{D1712} yields \(\su{0}=A\), that \rrema{R1514} yields \(\seq{\su{j}}{j}{0}{\kappa+1}\in\Dpqu{\kappa+1}\), and that \rlemp{L1455}{L1455.a} yields \(\su{j}^\seinsalpha=AA^\MP t_jA^\MP A\) for all \(j\in\mn{0}{\kappa}\), where \(\seq{\su{j}^\seinsalpha}{j}{0}{\kappa}\) denotes the \tseinsalphaa{\seq{\su{j}}{j}{0}{\kappa+1}}.
 
 \eqref{P1052.a}  Use \eqref{H} and \(\su{0}=A\).
 
 \eqref{P1052.b} In view of \rrema{R1247}, the matrices \(\Dlu{n}\) and \(\Dru{n}\) are invertible. Because of \(\seq{\su{j}}{j}{0}{\kappa+1}\in\Dpqu{\kappa+1}\) and \rprop{P1256}, we have \(\seq{\su{j}}{j}{0}{2n}\in\Dpqu{2n}\) which, in view of \rrema{R1608}, implies \(\Thetauu{n}{2n}=\Ouu{(n+1)p}{(n+1)q}\). Since \(\seq{\su{j}}{j}{0}{2n}\in\Dpqu{2n}\) holds, \rpropp{P1442}{P1442.a} implies \(\seq{\su{j}}{j}{0}{2n}\in\Dtpqu{2n}\). Thus, \rlemm{L1705} and \rrema{R1012} yield \eqref{L1705.A}. Since \(\su{j}^\seinsalpha=AA^\MP t_jA^\MP A\) holds true for all \(j\in\mn{0}{\kappa}\), we have, in view of \eqref{Hs} and \eqref{Ks}, furthermore \(-\alpha\Hu{n-1}^\seinsalpha+\Ku{n-1}^\seinsalpha=[\Iu{n}\kp(AA^\MP)](-\alpha\Huo{n-1}{t}+\Kuo{n-1}{t})[\Iu{n}\kp(A^\MP A)]\). Using \(\su{0}=A\), we thus get
 \begin{multline*}
   \diagA{A,\lek\Iu{n}\kp(AA^\MP)\rek(-\alpha\Huo{n-1}{t}+\Kuo{n-1}{t})\lek\Iu{n}\kp(A^\MP A)\rek}\\
   =\diaga{\su{0},-\alpha\Hu{n-1}^\seinsalpha+\Ku{n-1}^\seinsalpha}
   =\diaga{\su{0},-\alpha\Hu{n-1}^\seinsalpha+\Ku{n-1}^\seinsalpha}+\Thetauu{n}{2n}
   =\Dlu{n}\Hu{n}\Dru{n}.
 \end{multline*}
 Since the matrices \(\Dlu{n}\) and \(\Dru{n}\) are invertible, then \eqref{P1052.A} 
 and \eqref{P1052.C}
 follow.

 Now suppose \(p=q\). In view of \rrema{R1247} we obtain then \(\det(\Dlu{n}\Hu{n}\Dru{n})=\det\Hu{n}\). Furthermore, a straightforward calculation shows that
 \begin{multline*}
   \det\lrk\diagA{A,\lek\Iu{n}\kp(AA^\MP)\rek(-\alpha\Huo{n-1}{t}+\Kuo{n-1}{t})\lek\Iu{n}\kp(A^\MP A)\rek}\rrk\\
   =\det(A)\det(-\alpha\Huo{n-1}{t}+\Kuo{n-1}{t}).
 \end{multline*}
 Hence, \(\det\Hu{n}=\det(A)\det(-\alpha\Huo{n-1}{t}+\Kuo{n-1}{t})\) follows.
 
 \eqref{P1052.c} In view of \eqref{D[+]} and \rrema{R1247}, the matrices \(\Dlu{n}^\splusalpha\) and \(\Dru{n}^\splusalpha\) are invertible. Because of \(\seq{\su{j}}{j}{0}{\kappa+1}\in\Dpqu{\kappa+1}\) and \rprop{P1256}, we have \(\seq{\su{j}}{j}{0}{2n+1}\in\Dpqu{2n+1}\), which, in view of \rrema{R1608}, implies \(\Thetauu{n}{2n+1}=\Ouu{(n+1)p}{(n+1)q}\). Since \(\seq{\su{j}}{j}{0}{2n+1}\in\Dpqu{2n+1}\) holds, \rpropp{P1442}{P1442.a} implies \(\seq{\su{j}}{j}{0}{2n+1}\in\Dtpqu{2n+1}\). Thus, \rlemp{L1548}{L1548.a} and \rrema{R1012} yield \(\Hu{n}^\seinsalpha=\Dlu{n}^\splusalpha(\Hau{n}-\Thetauu{n}{2n+1})\Dru{n}^\splusalpha\). Since \(\su{j}^\seinsalpha=AA^\MP t_jA^\MP A\) holds true for all \(j\in\mn{0}{\kappa}\), we have, in view of \eqref{Hs}, furthermore \(\Hu{n}^\seinsalpha=[\Iu{n+1}\kp(AA^\MP)]\Huo{n}{t}[\Iu{n+1}\kp(A^\MP A)]\). Thus, we get
 \[
  \begin{split}
   \lek\Iu{n+1}\kp(AA^\MP)\rek\Huo{n}{t}\lek\Iu{n+1}\kp(A^\MP A)\rek
   &=\Hu{n}^\seinsalpha
   =\Dlu{n}^\splusalpha(\Hau{n}-\Thetauu{n}{2n+1})\Dru{n}^\splusalpha\\
   &=\Dlu{n}^\splusalpha\Hau{n}\Dru{n}^\splusalpha.
  \end{split}
 \]
 Since the matrices \(\Dlu{n}^\splusalpha\) and \(\Dru{n}^\splusalpha\) are invertible, then \eqref{P1052.B} 
 and \eqref{P1052.D} 
 follow.
 
 Now suppose \(p=q\). In view of \eqref{D[+]} and \rrema{R1247}, we obtain then the equation \(\det(\Dlu{n}^\splusalpha\Hau{n}\Dru{n}^\splusalpha)=\det\Hau{n}\). Furthermore, we easily get
 \[
  \det\lrk\lek\Iu{n+1}\kp(AA^\MP)\rek\Huo{n}{t}\lek\Iu{n+1}\kp(A^\MP A)\rek\rrk
  =(\det A)(\det A)^\MP\det\Huo{n}{t}.
 \]
 Hence, \(\det\Hau{n}=(\det A)(\det A)^\MP\det\Huo{n}{t}\) follows.
\eproof

\bcorol{C1036}
 Let \(\alpha\in\C\), let \(\kappa\in\NO\cup\set{+\infty}\), let \(\seq{t_j}{j}{0}{\kappa}\) be a sequence of complex \tpqa{matrices}, and let \(A\) be a complex \tpqa{matrix} such that \(\Kerna{A}\subseteq\bigcap_{j=0}^\kappa\Kerna{t_j}\) and \(\bigcup_{j=0}^\kappa\Bilda{t_j}\subseteq\Bilda{A}\). Denote by \(\seq{\su{j}}{j}{0}{\kappa+1}\) the \tsminuseinsalphaaa{\seq{t_j}{j}{0}{\kappa}}{A}. Then:
 \benui
  \il{C1036.a} \(\Hu{0}=A\).
  \il{C1036.b} Let \(n\in\N\) with \(2n-1\leq\kappa\). Then the matrices \(\Dlu{n}\) and \(\Dru{n}\) are invertible and
  \[
   \Hu{n}
   =\Dlu{n}^\inv\lrk\diaga{A,-\alpha\Huo{n-1}{t}+\Kuo{n-1}{t})}\rrk\Dru{n}^\inv.
  \]
  In particular, \(\rank\Hu{n}=\rank(A)+\rank(-\alpha\Huo{n-1}{t}+\Kuo{n-1}{t})\) and, in the case \(p=q\), furthermore \(\det\Hu{n}   =\det(A)\det(-\alpha\Huo{n-1}{t}+\Kuo{n-1}{t})\).
  \il{C1036.c} Let \(n\in\NO\) with \(2n\leq\kappa\). Then the matrices \(\Dlu{n}^\splusalpha\) and \(\Dru{n}^\splusalpha\) are invertible and
  \bgl{C1036.A}
   \Hau{n}
   =(\Dlu{n}^\splusalpha)^\inv\Huo{n}{t}(\Dru{n}^\splusalpha)^\inv.
  \eg
  In particular,
  \bgl{C1036.B}
   \rank\Hau{n}
   =\rank\Huo{n}{t}
  \eg
  and, in the case \(p=q\), furthermore \(\det\Hau{n}=\det\Huo{n}{t}\).
 \eenui
\ecoro
\bproof
 Using \rpartss{R1631.c}{R1631.b} of \rrema{R1631} we obtain \(AA^\MP t_jA^\MP A=t_j\) for all \(j\in\mn{0}{\kappa}\) which, in view of \eqref{Hs} and \eqref{Ks}, implies \([\Iu{n+1}\kp(AA^\MP)]\Huo{n}{t}[\Iu{n+1}\kp(A^\MP A)]=\Huo{n}{t}\) for all \(n\in\NO\) with \(2n\leq\kappa\) and \([\Iu{n}\kp(AA^\MP)](-\alpha\Huo{n-1}{t}+\Kuo{n-1}{t})[\Iu{n}\kp(A^\MP A)]=-\alpha\Huo{n-1}{t}+\Kuo{n-1}{t}\) for all \(n\in\N\) with \(2n-1\leq\kappa\). The application of \rprop{P1052} yields then \eqref{C1036.a}, \eqref{C1036.b}, \eqref{C1036.A}, and \eqref{C1036.B}. For all \(n\in\NO\) with \(2n\leq\kappa\), from \eqref{C1036.A}, \eqref{D[+]}, and \rrema{R1247} we get in the case \(p=q\) that \(\det\Hau{n}=\det\Huo{n}{t}\).
\eproof

\blemml{L1508}
 Let \(\alpha\in\R\), let \(\kappa\in\NO\cup\set{+\infty}\), let \(\seq{t_j}{j}{0}{\kappa}\) be a sequence of \tH{} complex \tqqa{matrices}, and let \(A\) be a \tH{} complex \tqqa{matrix}. Then the \tsminuseinsalphaaa{\seq{t_j}{j}{0}{\kappa}}{A} is a sequence of \tH{} complex \tqqa{matrices}.
\elemm
\bproof
 Denote by \(\seq{\su{j}}{j}{0}{\kappa+1}\) the \tsminuseinsalphaaa{\seq{t_j}{j}{0}{\kappa}}{A} and let the sequence \(\seq{r_j}{j}{0}{\kappa+1}\) be given by \(r_j\defg\su{j}^\rezalpha\) for all \(j\in\mn{0}{\kappa+1}\). From \rlemm{L1538} we infer then \(r_{j}^\rez=\su{j}^\splusalpha\) for all \(j\in\mn{0}{\kappa+1}\) and, furthermore, \(r_{0}=A^\MP\) and \(r_{j}=-A^\MP t_{j-1}A^\MP\) for all \(j\in\mn{1}{\kappa+1}\). In view of \rremp{R1631}{R1631.a}, we have \((A^\MP)^\ad=(A^\ad)^\MP=A^\MP\). We then conclude that \(r_{j}^\ad=r_{j}\) holds true for all \(j\in\mn{0}{\kappa+1}\). By virtue of~\zitaa{MR3014197}{\ccoro{5.17}}, then \((r_{j}^\rez)^\ad=r_{j}^\rez\) for all \(j\in\mn{0}{\kappa+1}\). \rremp{R0816:D-H}{R0816.h} yields then \(\su{j}^\ad=\su{j}\) for all \(j\in\mn{0}{\kappa+1}\).
\eproof

 The final investigations in this section are aimed at determining conditions which ensure that the \tsminuseinsalpha{} belongs to the class \(\Kggqualpha{\kappa+1}\) or one of its prominent subclasses.

\bpropl{P1539}
 Let \(\alpha\in\R\), let \(\kappa\in\NO\cup\set{+\infty}\), let \(\seq{t_j}{j}{0}{\kappa}\in\Kggqkappaalpha\), and let \(A\in\Cggq\). Then the \tsminuseinsalphaaa{\seq{t_j}{j}{0}{\kappa}}{A} belongs to \(\Kggqualpha{\kappa+1}\).
\eprop
\bproof
 Denote by \(\seq{\su{j}}{j}{0}{\kappa+1}\) the \tsminuseinsalphaaa{\seq{t_j}{j}{0}{\kappa}}{A}. It is sufficient to check that, for all \(n\in\NO\) with \(2n\leq\kappa+1\), the matrix \(\Hu{n}\) is \tnnH{} and that, for all \(n\in\NO\) with \(2n\leq\kappa\), the matrix \(\Hau{n}\) is \tnnH{} as well. According to \rpropp{P1052}{P1052.a} and \(A\in\Cggq\), we have \(\Hu{0}=A\in\Cggq\). Now we consider an arbitrary \(n\in\N\) with \(2n\leq\kappa+1\). Because of \rpropp{P1052}{P1052.b} the matrices \(\Dlu{n}\) and \(\Dru{n}\) are invertible and \eqref{P1052.A} holds true. 
 From \rlemp{L1738}{L1738.a} we know that \(\seq{t_j}{j}{0}{\kappa}\) is a sequence of \tH{} complex \tqqa{matrices}. Taking additionally into account \(A\in\Cggq\) and \rlemm{L1508}, we see then that \(\seq{\su{j}}{j}{0}{\kappa+1}\) is a sequence of \tH{} complex \tqqa{matrices}. In view of \rlemm{L1320}, hence \(\Dlu{n}^\ad=\Dru{n}\). Using \rremp{R1631}{R1631.a} and \(A\in\Cggq\), we obtain \((AA^\MP)^\ad=A^\MP A\). Thus, \eqref{P1052.A} implies
 \bgl{P1539.1}
  \Hu{n}
  =\Dlu{n}^\inv\lrk\diagA{A,\lek\Iu{n}\kp(AA^\MP)\rek(-\alpha\Huo{n-1}{t}+\Kuo{n-1}{t})\lek\Iu{n}\kp(AA^\MP)\rek^\ad}\rrk\Dlu{n}^\invad.
 \eg
 Using \rrema{R1440}, we get \(\seq{t_j}{j}{0}{2n-1}\in\Kggqualpha{2n-1}\), which, in view of \eqref{s_a} and \eqref{Kgg2n+1}, implies \(\seq{-\alpha t_j+t_{j+1}}{j}{0}{2n-2}\in\Hggqu{2n-2}\). Taking into account \eqref{Hs} and \eqref{Ks}, we see then that \(-\alpha\Huo{n-1}{t}+\Kuo{n-1}{t}\) is a \tnnH{} matrix. In view of \(A\in\Cggq\) and \eqref{P1539.1}, it follows that \(\Hu{n}\) is a \tnnH{} matrix. Thus, for all \(n\in\NO\) with \(2n\leq\kappa+1\), the matrix \(\Hu{n}\) is proved to be \tnnH.
 
 Finally, we consider an arbitrary \(n\in\NO\) with \(2n\leq\kappa\). Because of \rpropp{P1052}{P1052.c}, the matrices \(\Dlu{n}^\splusalpha\) and \(\Dru{n}^\splusalpha\) are invertible and \eqref{P1052.B} holds true.
 We know that \(\seq{\su{j}}{j}{0}{\kappa+1}\) is a sequence of \tH{} complex \tqqa{matrices}. Because of \rremp{R0816:D-H}{R0816.h}, we see then that \(\seq{\su{j}^\splusalpha}{j}{0}{\kappa+1}\) is a sequence of \tH{} complex \tqqa{matrices}. In view of \eqref{D[+]} and \rlemm{L1320}, hence \((\Dlu{n}^\splusalpha)^\ad=\Dru{n}^\splusalpha\). Taking additionally into account \(A^\MP A=(AA^\MP)^\ad\), thus
 \bgl{P1539.2}
  \Hau{n}
  =(\Dlu{n}^\splusalpha)^\inv\lek\Iu{n+1}\kp(AA^\MP)\rek\Huo{n}{t}\lek\Iu{n+1}\kp(AA^\MP)\rek^\ad(\Dlu{n}^\splusalpha)^\invad
 \eg
 follows. Using \rrema{R1440}, we get \(\seq{t_j}{j}{0}{2n}\in\Kggqualpha{2n}\), which, in view of \eqref{Kgg2n}, implies \(\seq{t_j}{j}{0}{2n}\in\Hggqu{2n}\). From \eqref{Hs} we see then that \(\Huo{n}{t}\) is a \tnnH{} matrix. Hence, \eqref{P1539.2} implies that \(\Hau{n}\) is a \tnnH{} matrix. Thus, for all \(n\in\NO\) with \(2n\leq\kappa\), the matrix \(\Hau{n}\) is proved to be \tnnH.
\eproof

\bpropl{P1555}
 Let \(\alpha\in\R\), let \(\kappa\in\NO\cup\set{+\infty}\), let \(\seq{t_j}{j}{0}{\kappa}\in\Kggeqkappaalpha\), and let \(A\in\Cggq\). Then the \tsminuseinsalphaaa{\seq{t_j}{j}{0}{\kappa}}{A} belongs to \(\Kggequalpha{\kappa+1}\).
\eprop
\bproof
 In the case \(\kappa=+\infty\), the assertion immediately follows from \rprop{P1539}. Now let \(\kappa\in\NO\). Then there exists a complex \tqqa{matrix} \(t_{\kappa+1}\) such that \(\seq{t_j}{j}{0}{\kappa+1}\in\Kggqualpha{\kappa+1}\). From \rprop{P1539} we know then that the \tsminuseinsalphaaa{\seq{t_j}{j}{0}{\kappa+1}}{A} belongs to \(\Kggqualpha{\kappa+2}\). In view of \rrema{R1504}, the proof is complete.
\eproof

Our next theme is to study the \tsminuseinsalpha{} under the view of \trasp{\alpha}. The two following results can be considered in some sense as inverse ones with respect to \rtheo{T1658}.

\begin{thm}\label{T1331}
 Let \(\alpha\in\R\), let \(m\in\NO\), let \(\seq{t_j}{j}{0}{m}\in\Kggqualpha{m}\), and let \(A\in\Cggq\) be such that \(\Kerna{A}\subseteq[\Kerna{t_0}]\cap[\Kerna{t_m}]\). Denote by \(\seq{\su{j}}{j}{0}{m+1}\) the \tsminuseinsalphaaa{\seq{t_j}{j}{0}{m}}{A} and by \(\seq{T_j}{j}{0}{m}\) and \(\seq{\Spu{j}}{j}{0}{m+1}\) the \trasp{\alpha}s of \(\seq{t_j}{j}{0}{m}\) and \(\seq{\su{j}}{j}{0}{m+1}\), respectively. Then \(\Spu{0}=A\) and \(\Spu{j}=T_{j-1}\) for all \(j\in\mn{1}{m+1}\).
\end{thm}
\begin{proof}
 Because of \rdefi{D1021}, \eqref{L}, and \rdefi{D1712}, we get \(\Spu{0}=\Lu{0}=\su{0}=A\). According to \rprop{P1539}, the sequence \(\seq{\su{j}}{j}{0}{m+1}\) belongs to \(\Kggqualpha{m+1}\). By virtue of \rlemp{L1738}{L1738.a}, we have \(t_0^\ad=t_0\) and \(t_m^\ad=t_m\), which, in view of \(A\in\Cggq\) and \(\Kerna{A}\subseteq[\Kerna{t_0}]\cap[\Kerna{t_m}]\), implies \([\Bilda{t_0}]\cup[\Bilda{t_m}]\subseteq\Bilda{A}\). \rpropp{P1442}{P1442.c} yields furthermore \(\seq{t_j}{j}{0}{m}\in\Dtqqu{m}\). Taking additionally into account \rdefiss{D1459}{D1658}, we conclude that \(\Kerna{A}\subseteq\bigcap_{j=0}^m\Kerna{t_j}\) and \(\bigcup_{j=0}^m\Bilda{t_j}\subseteq\Bilda{A}\).
 From \rlemp{L1455}{L1455.b} we see then that \(\seq{t_j}{j}{0}{m}\) is exactly the \tseinsalphaa{\seq{\su{j}}{j}{0}{m+1}}. The application of \rtheo{T1000} with \(k=1\) yields then \(T_m=\Spu{m+1}-\Puu{m+1}{0}\) (where \(\Puu{m+1}{0}\) is given via \eqref{P[k]}) and, in the case \(m\geq1\), furthermore \(T_l=\Spu{1+l}\) for all \(l\in\mn{0}{m-1}\). \rrema{R1514} yields \(\seq{\su{j}}{j}{0}{m+1}\in\Dqqu{m+1}\). Taking into account \eqref{P[k]}, \rdefiss{D1632}{D1658}, and \rpartss{R1631.c}{R1631.b} of \rrema{R1631}, we obtain then \(\Puu{m+1}{0}=\Oqq\). Thus, \(T_l=\Spu{1+l}\) for all \(l\in\mn{0}{m}\), which completes the proof.
\end{proof}

\begin{thm}\label{T0957}
 Let \(\alpha\in\R\), let \(\kappa\in\NO\cup\set{+\infty}\), let \(\seq{t_j}{j}{0}{\kappa}\in\Kggequalpha{\kappa}\), and let \(A\in\Cggq\) be such that \(\Kerna{A}\subseteq\Kerna{t_0}\). Denote by \(\seq{\su{j}}{j}{0}{\kappa+1}\) the \tsminuseinsalphaaa{\seq{t_j}{j}{0}{\kappa}}{A} and by \(\seq{T_j}{j}{0}{\kappa}\) and \(\seq{\Spu{j}}{j}{0}{\kappa+1}\) the \trasp{\alpha} of \(\seq{t_j}{j}{0}{\kappa}\) and \(\seq{\su{j}}{j}{0}{\kappa+1}\), respectively. Then \(\Spu{0}=A\) and \(\Spu{j}=T_{j-1}\) for all \(j\in\mn{1}{\kappa+1}\).
\end{thm}
\begin{proof}
 Because of \rdefi{D1021}, \eqref{L}, and \rdefi{D1712}, we have \(\Spu{0}=\Lu{0}=\su{0}=A\). According to \rprop{P1555}, the sequence \(\seq{\su{j}}{j}{0}{\kappa+1}\) belongs to \(\Kggequalpha{\kappa+1}\). Obviously, we have \(\seq{t_j}{j}{0}{\kappa}\in\Kggqualpha{\kappa}\). Because of \rlemp{L1738}{L1738.a}, hence \(t_0^\ad=t_0\), which, in view of \(A\in\Cggq\) and \(\Kerna{A}\subseteq\Kerna{t_0}\), implies \(\Bilda{t_0}\subseteq\Bilda{A}\). \rpropp{P1442}{P1442.a} yields furthermore \(\seq{t_j}{j}{0}{\kappa}\in\Dqqu{\kappa}\).
 From \rprop{P1313} we see then that \(\seq{t_j}{j}{0}{\kappa}\) is exactly the \tseinsalphaa{\seq{\su{j}}{j}{0}{\kappa+1}}. The application of \rtheo{T1658} with \(k=1\) yields then \(T_l=\Spu{1+l}\) for all \(l\in\mn{0}{\kappa}\), which completes the proof.
\end{proof}

\bpropl{P1615}
 Let \(\alpha\in\R\), let \(\kappa\in\NO\cup\set{+\infty}\), let \(\seq{t_j}{j}{0}{\kappa}\in\Kgqkappaalpha\), and let \(A\in\Cgq\). Then the \tsminuseinsalphaaa{\seq{t_j}{j}{0}{\kappa}}{A} belongs to \(\Kgqualpha{\kappa+1}\).
\eprop
\bproof
 From \rpropp{P1442}{P1442.d} we obtain \(\seq{t_j}{j}{0}{\kappa}\in\Kggequalpha{\kappa}\). Since \(A\) belongs to \(\Cgq\), we have \(A\in\Cggq\) and \(\det A\neq0\). In particular, \(\Kerna{A}\subseteq\Kerna{t_0}\). Denote by \(\seq{\su{j}}{j}{0}{\kappa+1}\) the \tsminuseinsalphaaa{\seq{t_j}{j}{0}{\kappa}}{A}. From \rtheo{T0957} we get then \(\Spu{0}=A\) and \(\Spu{j}=T_{j-1}\) for all \(j\in\mn{1}{\kappa+1}\), where  \(\seq{T_j}{j}{0}{\kappa}\) and \(\seq{\Spu{j}}{j}{0}{\kappa+1}\) denote the \trasp{\alpha}s of \(\seq{t_j}{j}{0}{\kappa}\) and \(\seq{\su{j}}{j}{0}{\kappa+1}\), respectively. According to \rthmp{121.P1337}{121.P1337.d}, we have \(T_l\in\Cgq\) for all \(l\in\mn{0}{\kappa}\). In view of \(A\in\Cgq\), thus \(\Spu{j}\in\Cgq\) for all \(j\in\mn{0}{\kappa+1}\) which, in view of \rthmp{121.P1337}{121.P1337.d}, implies \(\seq{\su{j}}{j}{0}{\kappa+1}\in\Kgqualpha{\kappa+1}\).
\eproof

\bpropl{P1459}
 Let \(\alpha\in\R\), let \(m\in\NO\), let \(\seq{t_j}{j}{0}{m}\in\Kggdqualpha{m}\), and let \(A\in\Cggq\) be such that \(\Kerna{A}\subseteq\Kerna{t_0}\) in the case \(m\geq2\). Then the \tsminuseinsalphaaa{\seq{t_j}{j}{0}{m}}{A} belongs to \(\Kggdqualpha{m+1}\).
\eprop
\bproof
 Because of \rpropp{P1442}{P1442.d}, we have \(\seq{t_j}{j}{0}{m}\in\Kggequalpha{m}\). The application of \rprop{P1555} yields then \(\seq{\su{j}}{j}{0}{m+1}\in\Kggequalpha{m+1}\), where \(\seq{\su{j}}{j}{0}{m+1}\) denotes the \tsminuseinsalphaaa{\seq{t_j}{j}{0}{m}}{A}. In particular, \(\seq{t_j}{j}{0}{m}\in\Kggqualpha{m}\) and \(\seq{\su{j}}{j}{0}{m+1}\in\Kggqualpha{m+1}\). We show now that \(\Spu{m+1}=\Oqq\) holds true, where \(\seq{\Spu{j}}{j}{0}{m+1}\) denotes the \trasp{\alpha} of \(\seq{\su{j}}{j}{0}{m+1}\).
 
 First we consider the case \(m=0\). Because of \eqref{Kggcd2n} then \(\seq{t_j}{j}{0}{0}\) belongs to \(\Hggdqu{0}\) which, in view of \eqref{L} and \eqref{Hggcd}, implies \(t_0=\Luo{0}{t}=\Oqq\). From \rlemm{L0831} we obtain hence \(\su{0}=A\) and \(s_1=\alpha\su{0}+AA^\MP t_0A^\MP\su{0}^\splusalpha=\alpha A\). In view of \rdefi{D1021}, \eqref{La}, \eqref{L}, and \eqref{s_a}, this implies \(
  \Spu{1}
  =\Lau{0}
  =\sau{0}
  =-\alpha\su{0}+\su{1}
  =-\alpha A+\alpha A
  =\Oqq
 \).
 
 Now we consider the case \(m=1\). Because of \eqref{s_a} and \eqref{Kggcd2n+1}, then the sequence \(\seq{u_j}{j}{0}{0}\) given by \(u_0\defg-\alpha t_0+t_1\) belongs to \(\Hggdqu{0}\) which, in view of \eqref{L} and \eqref{Hggcd}, implies \(u_0=\Luo{0}{u}=\Oqq\). Thus, \(t_1=\alpha t_0\). Using \rdefi{D1455} and \rlemm{L0831}, then \(\su{0}=A\),
 \[
  \su{1}
  =\alpha\su{0}+AA^\MP t_0A^\MP\su{0}^\splusalpha
  =\alpha\su{0}+AA^\MP t_0A^\MP\su{0}
  =\alpha A+AA^\MP t_0A^\MP A,
 \]
 and
 \[
  \begin{split}
   \su{2}
   &=\alpha\su{1}+AA^\MP(t_1A^\MP\su{0}^\splusalpha+t_0A^\MP\su{1}^\splusalpha)
   =\alpha\su{1}+AA^\MP\lek t_1A^\MP\su{0}+t_0A^\MP(-\alpha\su{0}+\su{1})\rek\\
   &=\alpha\su{1}+AA^\MP(\alpha t_0A^\MP\su{0}-\alpha t_0A^\MP\su{0}+t_0A^\MP\su{1})
   =\alpha\su{1}+AA^\MP t_0A^\MP\su{1}\\
   &=\alpha(\alpha A+AA^\MP t_0A^\MP A)+AA^\MP t_0A^\MP(\alpha A+AA^\MP t_0A^\MP A)\\
   &=\alpha^2 A+2\alpha AA^\MP t_0A^\MP A+AA^\MP t_0A^\MP AA^\MP t_0A^\MP A\\
   &=(\alpha A+AA^\MP t_0A^\MP A)A^\MP(\alpha A+AA^\MP t_0A^\MP A)
   =\su{1}\su{0}^\MP\su{1}.
  \end{split}
 \]
 In view of \rdefi{D1021}, \eqref{L}, \eqref{yz}, \eqref{H}, and \eqref{s_a}, this implies \(
  \Spu{2}
  =\Lu{1}
  =\su{2}-\zuu{1}{1}\Hu{0}^\MP\yuu{1}{1}
  =\su{2}-\su{1}\su{0}^\MP\su{1}
  =\Oqq
 \).
 
 Finally, we  consider the case \(m\geq2\). From \rtheo{T0957} we obtain then \(\Spu{m+1}=T_m\), where  \(\seq{T_j}{j}{0}{m}\) denotes the \trasp{\alpha} of \(\seq{t_j}{j}{0}{m}\). According to~\zitaa{MR3014201}{\cprop{5.3}}, we have \(T_m=\Oqq\) and hence \(\Spu{m+1}=\Oqq\).
 
 Thus, in each case, \(\Spu{m+1}=\Oqq\) holds true, which, in view of \(\seq{\su{j}}{j}{0}{m+1}\in\Kggqualpha{m+1}\) and~\zitaa{MR3014201}{\cprop{5.3}}, implies \(\seq{\su{j}}{j}{0}{m+1}\in\Kggdqualpha{m+1}\).
\eproof

 The following example shows that, in the case \(m\geq2\), the assumption \(\Kerna{A}\subseteq\Kerna{t_0}\) in \rprop{P1459} cannot be omitted.

\begin{exa}\label{E1518}
 Let \(q\defg2\), let \(\alpha\defg-1\), let \(t_0\defg[\begin{smallmatrix}1&0\\0&1\end{smallmatrix}\bigr]\), let \(t_1\defg\bigl[\begin{smallmatrix}1&1\\1&0\end{smallmatrix}\bigr]\), let \(t_2\defg\bigl[\begin{smallmatrix}2&1\\1&1\end{smallmatrix}\bigr]\), and let \(A\defg\bigl[\begin{smallmatrix}1&0\\0&0\end{smallmatrix}\bigr]\). According to \eqref{Hs}, then
 \[
  \Huo{1}{t}
  =
  \begin{bmatrix}
   t_0&t_1\\
   t_1&t_2
  \end{bmatrix}
  =
  \begin{bmatrix}
   1&0&1&1\\
   0&1&1&0
  \end{bmatrix}^\ad
  \begin{bmatrix}
   1&0&1&1\\
   0&1&1&0
  \end{bmatrix}
  \in\Cggo{4},
 \]
 which implies \(\seq{t_j}{j}{0}{2}\in\Hggqu{2}\). Obviously, \(
  -\alpha t_0+t_1
  =
  \bigl[\begin{smallmatrix}
   2&1\\
   1&1
  \end{smallmatrix}\bigr]
  \in\Cggo{2}
 \), which, in view of \eqref{Hs}, implies \(\seq{-\alpha t_j+t_{j+1}}{j}{0}{0}\in\Hggqu{0}\). According to \eqref{s_a} and \eqref{Kgg2n}, thus \(\seq{t_j}{j}{0}{2}\in\Kggqualpha{2}\). Because of \eqref{L}, \eqref{yz}, and \eqref{Hs}, furthermore \(
  \Luo{1}{t}
  =t_2-\zuuo{1}{1}{t}(\Huo{0}{t})^\MP\yuuo{1}{1}{t}
  =t_2-t_1t_0^\MP t_1
  =\Ouu{2}{2}
 \). By \eqref{Hggcd}, hence \(\seq{t_j}{j}{0}{2}\in\Hggdqu{2}\). According to \eqref{Kggcd2n}, then \(\seq{t_j}{j}{0}{2}\in\Kggdqualpha{2}\). Obviously, \(A\in\Cggo{2}\) and \(\Kerna{A}\nsubseteq\Kerna{t_0}\). Denote by \(\seq{\su{j}}{j}{0}{3}\) the \tsminuseinsalphaaa{\seq{t_j}{j}{0}{2}}{A}. Using \(A^2=A\), \(A^\MP=A\), \(At_1 A=A\), \(At_2A=2A\), and \rdefi{D1455}, we obtain from \rlemm{L0831} then \(\su{0}=A\),
 \begin{align*}
  \su{1}&=\alpha\su{0}+AA^\MP t_0A^\MP\su{0}^\splusalpha
  =\alpha\su{0}+AA^\MP t_0A^\MP\su{0}
  =-A+AA^\MP\cdot\Iu{2}\cdot A^\MP A
  =\Ouu{2}{2},\\
  \su{2}
  &=\alpha\su{1}+AA^\MP(t_1A^\MP\su{0}^\splusalpha+t_0A^\MP\su{1}^\splusalpha)
  =\alpha\su{1}+AA^\MP\lek t_1A^\MP\su{0}+t_0A^\MP(-\alpha\su{0}+\su{1})\rek\\
  &=-\Ouu{2}{2}+AA^\MP\lek t_1A^\MP A+\Iu{2}\cdot A^\MP(A+\Ouu{2}{2})\rek
  =2A,\\
 \shortintertext{and, consequently,}
  \su{3}
  &=\alpha\su{2}+AA^\MP(t_2A^\MP\su{0}^\splusalpha+t_1A^\MP\su{1}^\splusalpha+t_0A^\MP\su{2}^\splusalpha)\\
  &=\alpha\su{2}+AA^\MP\lek t_2A^\MP\su{0}+t_1A^\MP(-\alpha\su{0}+\su{1})+t_0A^\MP(-\alpha\su{1}+\su{2})\rek\\
  &=-2A+AA^\MP\lek t_2A^\MP A+t_1A^\MP(A+\Ouu{2}{2})+\Iu{2}\cdot A^\MP(\Ouu{2}{2}+2A)\rek
  =3A.
 \end{align*}
 In view of \eqref{s_a}, then \(\sau{0}=A\) and \(\sau{1}=2A\), and \(\sau{2}=5A\), which, in view of \eqref{La}, \eqref{L}, \eqref{yaza}, \eqref{Ha}, \eqref{yz}, and \eqref{Hs}, implies \(
  \Lau{1}
  =\sau{2}-\zauu{1}{1}(\Hau{0})^\MP\yauu{1}{1}\\
  =A
  \neq\Ouu{2}{2}
 \). By \eqref{Hggcd}, hence \(\seq{\sau{j}}{j}{0}{2}\notin\Hggdqu{2}\). According to \eqref{Kggcd2n+1}, then \(\seq{\su{j}}{j}{0}{3}\notin\Kggdqualpha{3}\).
\end{exa}

\bpropl{P0915}
 Let \(\alpha\in\R\), let \(\kappa\in\NO\cup\set{+\infty}\), let \(m\in\mn{0}{\kappa}\), let \(\seq{t_j}{j}{0}{\kappa}\in\Kggdoqkappaalpha{m}\), and let \(A\in\Cggq\) be such that \(\Kerna{A}\subseteq\Kerna{t_0}\) in the case \(m\geq2\). Then  the \tsminuseinsalphaaa{\seq{t_j}{j}{0}{\kappa}}{A} belongs to \(\Kggdoqualpha{m+1}{\kappa+1}\).
\eprop
\bproof
 Because of \eqref{Kggcdm}, we have \(\seq{t_j}{j}{0}{\kappa}\in\Kggqkappaalpha\) and \(\seq{t_j}{j}{0}{m}\in\Kggdqualpha{m}\). The application of \rprop{P1539} yields \(\seq{\su{j}}{j}{0}{\kappa+1}\in\Kggqualpha{\kappa+1}\), where \(\seq{\su{j}}{j}{0}{\kappa+1}\) denotes the \tsminuseinsalphaaa{\seq{t_j}{j}{0}{\kappa}}{A}. From \rprop{P1459} we obtain, in view of \rrema{R1504}, furthermore \(\seq{\su{j}}{j}{0}{m+1}\in\Kggdqualpha{m+1}\). According to \eqref{Kggcdm}, we have then \(\seq{\su{j}}{j}{0}{\kappa+1}\in\Kggdoqualpha{m+1}{\kappa+1}\).
\eproof
 
\bcorol{P1111}
 Let \(\alpha\in\R\), let \(\kappa\in\NO\cup\set{+\infty}\), let \(\seq{t_j}{j}{0}{\kappa}\in\Kggdqualpha{\kappa}\), and let \(A\in\Cggq\) be such that \(\Kerna{A}\subseteq\Kerna{t_0}\). Then the \tsminuseinsalphaaa{\seq{t_j}{j}{0}{\kappa}}{A} belongs to \(\Kggdqualpha{\kappa+1}\).
\ecoro
\bproof
 Use \rprop{P1459}, \eqref{G1023}, and \rprop{P0915}.
\eproof

\appendix
\section{Some facts from matrix theory}

\bremal{R1631}
 Let \(A\) be a complex \tpqa{matrix} and let \(m,n\in\N\). Then:
 \benui
  \il{R1631.a} \((A^\MP)^\MP=A\), \((A^\MP)^\ad=(A^\ad)^\MP\), \(\Kerna{A^\MP}=\Kerna{A^\ad}\), and \(\Bilda{A^\MP}=\Bilda{A^\ad}\).
  \il{R1631.b} If \(B\) is a complex \taaa{m}{q}{matrix}, then \(\Kerna{A}\subseteq\Kerna{B}\) if and only if \(BA^\MP A=B\).
  \il{R1631.c} If \(C\) is a complex \taaa{p}{n}{matrix}, then \(\Bilda{C}\subseteq\Bilda{A}\) if and only if \(AA^\MP C=C\).
  \il{R1631.d} If \(U\) is complex \taaa{m}{p}{matrix} with \(U^\ad U =\Ip\) and if \(V\) is a complex \taaa{q}{n}{matrix} with \(VV^\ad =\Iq\), then \((UAV)^\MP=V^\ad A^\MP U^\ad\).
 \eenui
\erema

Note that the sets \(\nudqu{n}\) and \(\nodqu{n}\) are defined in \rnota{N1039}.
\bremal{R1745}
 Let \(n\in\NO\). Then \(\nudqu{n}\) and \(\nodqu{n}\) are subgroups of the general linear group \(\GLaa{(n+1)q}{\C}\) of all \tns{} complex \taaa{(n+1)q}{(n+1)q}{matrices} with \(\nudqu{n}\cap\nodqu{n}=\set{\Iu{(n+1)q}}\). Furthermore, \(\det A=1\) for all \(A\in\nudqu{n}\cup\nodqu{n}\), \(\nodqu{n}=\setaa{L^\ad}{L\in\nudqu{n}}\), and \(\nudqu{n}=\setaa{U^\ad}{U\in\nodqu{n}}\).
\erema

\bremal{R1042}
 Let \(m,n\in\NO\). In view of \rnota{N1039}, then \(\diaga{L,M}\in\nudqu{m+n+1}\) for all \(L\in\nudqu{m}\) and all \(M\in\nudqu{n}\) and \(\diaga{U,V}\in\nodqu{m+n+1}\) for all \(U\in\nodqu{m}\) and all \(V\in\nodqu{n}\).
\erema

\bremal{R1427}
 Let \(n\in\NO\), let \((A_j)_{j=0}^n\)  and \((B_j)_{j=0}^n\) be  sequences of complex \tpqa{matrices}, and let \(E,V\in \nuduu{p}{n}\) and \(F,W\in\nodqu{n}\) be such that \(
  E\cdot\diaga{A_0, A_1,\dotsc, A_n}\cdot F
  =V\cdot\diaga{B_0,B_1, \dotsc, B_n}\cdot W
 \). Then we easily see that \(A_j=B_j\) holds true for all \(j\in\mn{0}{n}\).
\erema

\bibliography{114arxiv}
\bibliographystyle{abbrv}

\vfill\noindent
\begin{minipage}{0.5\textwidth}
 Universit\"at Leipzig\\
Fakult\"at f\"ur Mathematik und Informatik\\
PF~10~09~20\\
D-04009~Leipzig
\end{minipage}
\begin{minipage}{0.49\textwidth}
 \begin{flushright}
  \texttt{
   fritzsche@math.uni-leipzig.de\\
   kirstein@math.uni-leipzig.de\\
   maedler@math.uni-leipzig.de
  } 
 \end{flushright}
\end{minipage}

\end{document}